\documentclass[twoside]{article}

\usepackage{a4wide}
\usepackage[T1]{fontenc}	
\usepackage[latin1]{inputenc}	
\usepackage{calc}		
\usepackage{color}
\usepackage[colorlinks=true,citecolor=red,linkcolor=blue,breaklinks=true]{hyperref}

\usepackage{amsmath}	
\usepackage{amssymb}	
\usepackage{amsfonts}	
\usepackage{amsthm}	

\usepackage{etoolbox}

\DeclareFontFamily{U}{mathx}{}
\DeclareFontShape{U}{mathx}{m}{n}{<-> mathx10}{}
\DeclareSymbolFont{mathx}{U}{mathx}{m}{n}
\DeclareMathAccent{\widehat}{0}{mathx}{"70}
\DeclareMathAccent{\widecheck}{0}{mathx}{"71}

\usepackage{array}
\usepackage{curves}
\usepackage{epsfig}
\usepackage{graphics}
\usepackage{mathrsfs}

\newdimen\argwidth
\newcommand{\Jump}[1]{%
\setbox0=\hbox{$#1$}\argwidth=\wd0
\setbox0=\hbox{$\left[\box0\right]$}\advance\argwidth by -\wd0
\left[\kern.3\argwidth\box0\kern.3\argwidth\right]}

\newcommand{\bb}{\ensuremath{\mathbf{b}}}

\newcommand{\bRHS}{\ensuremath{\mathbf{RHS}}}

\newcommand{\bn}{\ensuremath{\mathbf{n}}}
\newcommand{\barG}{\ensuremath{\mathbf{\overline{G}}}}
\newcommand{\bu}{\ensuremath{\mathbf{u}}}

\newcommand{\be}{\ensuremath{\mathbf{e}}}
\newcommand{\bv}{\ensuremath{\mathbf{v}}}
\newcommand{\bx}{\ensuremath{\mathbf{x}}}

\newcommand{\bg}{\ensuremath{\mathbf{g}}}

\newcommand{\bG}{\ensuremath{\mathbf{G}}}

\newcommand{\bGtilde}{\ensuremath{\mathbf{\widetilde{G}}}}
\newcommand{\gtilde}{\ensuremath{\tilde{g}}}

\newcommand{\cS}{\ensuremath{\mathcal{S}}}
\newcommand{\cC}{\ensuremath{\mathcal{C}}}

\newcommand{\rbar}{\ensuremath{\overline{r}}}
\newcommand{\xbar}{\ensuremath{\overline{x}}}
\newcommand{\ybar}{\ensuremath{\overline{y}}}

\newcommand{\bPsi}{{\ensuremath{\boldsymbol{\Psi}}}}

\newcommand{\R}{\ensuremath{\mathbb{R}}}
\newcommand{\T}{\ensuremath{\mathbb{T}}}

\usepackage{bbm}
\newcommand{\PP}{\ensuremath{\mathbb{P}}}
\newcommand{\QQ}{\ensuremath{\mathbb{Q}}}
\newcommand{\bbC}{\ensuremath{\mathbb{C}}}
\newcommand{\bbA}{\ensuremath{\mathbb{A}}}

\newcommand{\bdPP}{\ensuremath{\pmb{\mathrm{d}\mathbb{P}}}}

\newcommand{\bdQQ}{\ensuremath{\pmb{\mathrm{d}\mathbb{Q}}}}
\newcommand{\RT}{\ensuremath{\pmb{\mathbb{RT}}}}

\newcommand{\bN}{\ensuremath{\pmb{\mathbb{N}}}}

\newcommand{\Bdiv}{\ensuremath{\pmb{\mathrm{d}\mathbb{B}}^{\mathrm{div}}}}
\newcommand{\Bcurl}{\ensuremath{\pmb{\mathrm{d}\mathbb{B}}^{\mathrm{curl}}}}
\newcommand{\bbB}{\ensuremath{\pmb{\mathbb{B}}}}

\newcommand{\fraku}{\ensuremath{\mathfrak{u}}}

\DeclareMathOperator{\Span}{Span}
\DeclareMathOperator{\Range}{Range}

\newcommand{\bnabla}{\ensuremath{\boldsymbol{\nabla}}}

\newcommand{\dnablaperp}{\ensuremath{\nabla^{\perp} _{\mathscr{D}'}}}
\newcommand{\dnabla}{\ensuremath{\nabla_{\mathscr{D}'}}}
\newcommand{\nablaperp}{\ensuremath{{\nabla^{\perp}}}}
\newcommand{\nablaperpStar}{\ensuremath{{\left(\nabla^{\perp}\right)^{\star}}}}
\newcommand{\nablaStar}{\ensuremath{{\left(\nabla\right)^{\star}}}}

\newcommand{\ex}[1]{\ensuremath{\mathrm{e}^{\,#1}}}

\newcounter{propc}
\newtheorem{prop}[propc]{Proposition}

\newcounter{defc}
\newtheorem{definition}[defc]{Definition}

\theoremstyle{definition}
\newcounter{remc}
\newtheorem{remark}[remc]{Remark}

\newcommand{\Sum}[2]{\ensuremath{\textstyle{\sum\limits_{#1}^{#2}}}}
\newcommand{\Int}[2]{\ensuremath{\mathchoice%
{{\displaystyle\int_{#1}^{#2}}}
{{\displaystyle\int_{#1}^{#2}}}
{\int_{#1}^{#2}}
{\int_{#1}^{#2}}
}}

\newcommand{\dd}{\ensuremath{{\rm d}}}

\newcommand{\norme}[1]{\ensuremath{\Arrowvert #1 \Arrowvert}}
\newcommand{\scalar}[2]{\ensuremath{\left\langle #1 | #2 \right\rangle}}

\newcommand{\diag}[1]{\stackrel{\displaystyle{#1}}{\xrightarrow{\qquad\qquad}}}

\usepackage{tikz-cd}
\usetikzlibrary{positioning, shapes.geometric}

\author{Vincent Perrier \\
    Team Cagire, INRIA Bordeaux Sud-Ouest. \\
    Laboratoire de Math\'ematiques et de leurs applications \\
    B\^atiment IPRA, 
    Universit\'e de Pau et des Pays de l'Adour, \\ 
    Avenue de l'Universit\'e, 
    64\,013 Pau Cedex}

\title{Development of discontinuous Galerkin methods  for hyperbolic systems
  that preserve a curl or a divergence constraint: the case of linear systems}

\begin{document}

\maketitle

\begin{abstract}
  Some hyperbolic systems are known to include implicit preservation
  of differential constraints: these are for example the time conservation
  of the curl or the divergence of a vector that appear as an implicit
  constraint.
  In this article, we show that this kind of constraint can be easily
  conserved at the discrete level with the classical discontinuous Galerkin
  method, provided the right approximation space is used for the
  vectorial space, and under some mild assumption on the numerical flux.
  Even if the approximation space that is used does not always
    match with the one used classicaly for the discontinuous Galerkin method,
    the approximation space is \emph{fully discontinuous}, i.e. the method
  \emph{is not} a staggered method.
  We start by recalling a discrete de-Rham framework in which discontinuous
  approximation spaces for vectors fits. 
  The discrete adjoint divergence and curl, which are \emph{global}
  divergence and curl,
  are proven to be exactly preserved by the discontinuous Galerkin
  method under a
  small assumption on the numerical flux. 
  Numerical tests are performed in dimension 2 on the wave
  system, the  Maxwell system and the induction
  equation, on Cartesian meshes, and unstructured triangular and quadrangular meshes
  and confirm that the differential constraints  are  preserved at
  machine precision while keeping the high order of accuracy.
\end{abstract}

\tableofcontents

\section{Introduction}

In this article, we are interested in the discrete conservation of
differential constraints that appear implicitly in a hyperbolic system
of conservation law. Suppose for example that a hyperbolic system
includes a vectorial unknown $\bu$ for which the conservation law is
$$\partial_t \bu + \bnabla \cdot \barG = 0,$$
where $\barG$ is a matrix. If $\barG$ is proportional to the identity matrix,
namely if a scalar function $g$ exists such that $\barG = g \mathrm{I}_d$,
then the conservation law becomes
$$\partial_t \bu + \nabla g = 0,$$
and taking the curl of the equation on $\bu$ gives formally the
conservation of the curl of $\bu$:
$$\partial_t \left( \nabla \times \bu \right) =
\nabla \times \left( \partial_t \bu \right) =
- \nabla \times \left( \nabla g \right) = 0.$$
In the same manner, if $\barG$ is antisymmetric, then the divergence of $\bu$ is
constant.

This kind of implicit differential constraint appears in a large number
of systems including incompressible Navier-Stokes system,
Maxwell system, magnetohydrodynamics (MHD), and wave and elastodynamic
problems written in first order formulation. 
Preservation of such constraints at the discrete level has been addressed
by several strategies, which can be gathered as
\begin{itemize}
\item {\bf Staggering of unknowns.}
  Staggering unknowns consists in non collocated distribution of unknowns,
  for example in defining scalar unknowns in each cells, and
  vector unknowns in the sides of the mesh. 
  The staggering of unknowns was proposed for the incompressible
  Navier-Stokes system  in the MAC scheme \cite{lebedev1964difference},
  and for the Yee scheme \cite{yee1966numerical} for Maxwell equations, see
  also \cite{balsara2001divergence,balsara2004second,balsara1999staggered}
  for the MHD system.
  Staggering of unknowns has been the object of a huge body of papers,
  including see e.g. \cite{nicolaides1992analysis,nicolaides1996analysis,
    eymard2010convergence,gallouet2009convergent,eymard2010convergent} for
  the analysis of the MAC scheme. 
  Among the large family of staggered schemes enters also the 
  work on compatible discretization \cite{hyman1997natural,lipnikov2014mimetic},
  or the staggered discontinuous Galerkin method \cite{tavelli2017pressure}.
  This includes also the work on
  discrete exterior calculus with finite
  elements, including theoretical work around the Hodge Laplacian and its
  well posed mixed formulation
  \cite{arnold2006finite,arnold2010finite,arnold2014periodic},
  see also the book \cite{arnold2018finite}, work on
  electromagnetism
  \cite{bossavit1988whitney,bossavit1998computational,bossavit1998geometry,hiptmair2001discrete,hiptmair2002finite}
  and discretization of the
  Lie-advection equation
  \cite{heumann2010eulerian,heumann2013stabilized,heumann2012fully}
  and application
  \cite{pagliantini2016computational}.
  Staggered discretizations were extended also to polytopal
  meshes, mainly for incompressible systems see e.g.
  the virtual element method \cite{beirao2013basic},
  the Hybrid High order (HHO) methods \cite{di2020hybrid}
  or the Compatible Discrete Operators (CDO) method
  \cite{bonelle2014compatible,bonelle2015analysis,milani2022artificial}.
  The main difficulty of staggered discretizations for nonlinear
  hyperbolic systems is to keep a correct definition of the conservation;
  especially, keeping conservation while shock-limiting staggered data is
  challenging. 
\item {\bf Projection method.}
  The projection method was mainly used for divergence cleaning
  in \cite{brackbill1980effect} for the MHD system and in 
  \cite{bell1989second} for incompressible Navier-Stokes system,
  see also \cite[the globally divergence free method]{cockburn2004locally}. The
  projection method is a predictor-corrector method. Suppose for
  example that the divergence of a vector $\bv$ should be preserved.
  Suppose also that a predictor step provided a candidate update for
  $\bv^{n+1}$ of $\bv^n$, but that the divergence is not preserved. Then
  a potential $\varphi$ such that
  $$\Delta \varphi = \nabla \cdot \left( \bv^{n+1} - \bv^n \right),$$
  may be computed. Then $\tilde{\bv}^{n+1} := \bv^{n+1} - \nabla \varphi$
  is a projection of $\bv^{n+1}$ for which the divergence is preserved.
  The main drawback of this method is the cost of the
  inversion of an elliptic system at each time step, and the definition
  of the boundary conditions for this elliptic system. 
\item {\bf Generalized Lagrange Multiplier method.} This method,
  first developed in \cite{munz2000divergence,dedner2002hyperbolic}
  consists in considering the divergence to be preserved as an additional
  variable. Then an additional equation for this variable, and a relaxation
  process ensure that the divergence is asymptotically preserved. This
  method was extended to curl preservation in \cite{dumbser2020glm}.
  The difficulty with this system is its higher computational
  cost because of the additional variables, and the tuning of the
  numerical relaxation parameters. 
\end{itemize}

Another category of method was especially designed for the conservation
of the zero divergence of the magnetic field in the Maxwell system,
the MHD system or the
induction model system; these systems are slightly
different from the ones previously discussed,
because the divergence of the magnetic field is not directly preserved
by the system. Instead, the divergence is solution of a
conservative transport equation, so that an initially divergence free
magnetic field is divergence free for all time. These two categories
of methods are 
\begin{itemize}
\item {\bf the Godunov-Powell method.} This method consists in
  relying on the formulation of Godunov 
  \cite{godunov1972symmetric,godunov1961interesting} and
  Powell \cite{powell1994approximate,powell1999solution} for the MHD
  equation, and in trying to control the
  divergence of the magnetic field \cite{fuchs2009stable,mishra2010stability}
  (note however that the magnetic field is not divergence free, the aim
  of the method is only to keep this divergence "low").
\item {\bf Constrained Transport Method.} This method was originally
  proposed in \cite{evans1988simulation} as an alternative of the
  Yee scheme \cite{yee1966numerical}, and under a staggered fashion
  see also \cite{li2012arbitrary} for an application to the
  high order discretization of the MHD system.
  This method is based on staggered ideas, and in the Maxwell context,
  the challenge consists in the computation of a reliable staggered electric
  field based on a magnetic field that, still \emph{may not be staggered}: 
  indeed, in \cite{toth2000b}, several methods were compared, including
  different versions of the contrained transport method, and it was shown that
  the method can be collocated, see also 
  \cite{torrilhon2004constraint,torrilhon2005locally,helzel2011unstaggered}.
  Concerning the contrained transport method, we also refer to 
  the review \cite{teyssier2019numerical} and references therein,  
  \cite{veiga2021arbitrary} for a high order constrained transport method
  based on spectral differences, and \cite{jeltsch2006curl} for an extension
  to the preservation of a curl.
\end{itemize}

Apart from these schemes that are especially designed for preserving exactly
a discrete curl or a discrete divergence, some collocated numerical schemes
seem to
be \emph{naturally} able to preserve these constraints. This is for
example the case of numerical schemes developed within the
low Mach number community \cite{Dellacherie2010_cell_geometry,guillard2009behavior,guillard2017behaviour,Jung2024}, which inspired the present article, and
which was extended for low order quadrangular meshes in \cite{Jung2024b}.
Even if these references do no state any relation with curl preservation,
they are actually preserving a curl that will be defined in this article.
Another family of schemes that seem to preserve exactly divergence or curl
constraints are node based numerical schemes, see e.g. 
\cite{boscheri2023unconventional} for the divergence preservation or
\cite{barsukow2019stationarity,barsukow2021truly} for curl
preservation. Considering their stencils, these schemes seem to have a
close relationship with \cite{torrilhon2005locally,jeltsch2006curl}, which
are based on the unstaggered constrained transport method.

In this article, we wish to develop two-dimensional
discontinuous Galerkin methods
that \emph{naturally} preserve a global divergence or a curl \emph{exactly}, which
match with \cite{Dellacherie2010_cell_geometry,guillard2009behavior,guillard2017behaviour,Jung2024} for straight triangular meshes but which differ on quadrangular meshes because we use alternative approximation space for vectors,
proposed in \cite{PerrierExact2024}. 

This article is organized as follows. In \autoref{sec:ApproximationSpaces}, the
approximation spaces developed in \cite{PerrierExact2024} are recalled,
and their relation with the de-Rham complex is recalled. 
Based on these approximation spaces and matching discrete operators, 
we prove in \autoref{sec:DiscreteCurlDiv}
that the classical discontinuous Galerkin method is able to preserve a
curl or a divergence constraint provided the right approximation space
is used, and under a small hypothesis regarding the diffusion direction of
the numerical flux. Then in \autoref{sec:LieDerivative}, we explain how
the numerical scheme can be also extended to the induction equation; we
especially prove that if the vector field is correctly initialized, then
it is divergence free (still in the adjoint sense) for all time.
In \autoref{sec:numericalResults}, the numerical scheme is tested for
the preservation of a divergence with the two dimensional Maxwell system,
the preservation of a curl with the two dimensional first order formulation
of the wave system, and the preservation of the divergence free field
with the induction equation. In this numerical section, convergence tests
are also performed. This article finishes with the conclusion in
\autoref{sec:conclusion}.

\section{Approximation spaces}
\label{sec:ApproximationSpaces}

\subsection{Continous de-Rham complex}

The most general context for introducing the gradient, curl and
divergence operator relies on exterior calculus. In dimension $2$, this
consists in considering the de-Rham complex:
\begin{equation}
  \label{eq:deRhamDifferentialForms}
  \Lambda^0 ( \Omega )
  \diag{\dd^0}
  \Lambda^1 (  \Omega )
  \diag{\dd^1}
  \Lambda^2 (  \Omega ), 
\end{equation}
where $\Lambda^k ( \Omega)$ is the set of smooth $k$-differential forms on
$\Omega$, and $\dd$ is the exterior derivative (see
\cite[Chap. 1.4]{Cartan1967}). When dealing with partial differential equations,
it is usually more convenient to describe \eqref{eq:deRhamDifferentialForms}
in terms of \emph{proxies}. In dimension $2$, a 
differential form of $\Lambda^1$ is usually represented in two manners:
\begin{itemize}
\item Either by the scalar product by a vector field $\bu$:
  \begin{equation}
    \label{eq:Proxy}
    \forall \fraku \in \Lambda^1 \quad \exists! \bu \in \R^2 \quad 
    \forall \bv_1 \in \R^2
    \qquad \fraku ( \bv_1) = \bu \cdot \bv_1,
  \end{equation}
\item or by the determinant with a vector field $\tilde{\bu}$:
  \begin{equation}
    \label{eq:ProxyOrtho}
    \forall \fraku \in \Lambda^1 \quad \exists! \tilde{\bu} \in \R^2 \quad 
    \forall \bv_1 \in \R^2
    \qquad \fraku ( \bv_1) = \det( \bv_1, \tilde{\bu}).
  \end{equation}
  Note that $\bu \cdot \bv_1 = \det( \bv_1, \tilde{\bu})$ if and only 
  if $\tilde{\bu}$ is the image of $\bu$ by the
  $\pi/2$ rotation which we denote by a $\perp$ exponent:
  $\tilde{\bu} = \bu^{\perp}$.
\end{itemize}
From now on, we denote by $\nabla$ the vector operator
  $\left(\partial_x, \partial_y \right)$ and
  $\nablaperp$ the vector operator $\left(-\partial_y, \partial_x \right)$.
Then, depending on the choice of proxy,
\eqref{eq:deRhamDifferentialForms} becomes,
if \eqref{eq:Proxy} is chosen: 
\begin{equation}
  \label{eq:deRhamGradCurl}
  C^{\infty} \left( \Omega \right)
  \diag{\nabla}
  \left( C^{\infty} \left( \Omega \right) \right) ^2
  \diag{\nablaperp \cdot}
  C^{\infty} \left( \Omega \right),
\end{equation}
whereas if \eqref{eq:ProxyOrtho} is chosen,
\eqref{eq:deRhamDifferentialForms} becomes:
\begin{equation}
  \label{eq:deRhamCurlDiv}
  C^{\infty} \left( \Omega \right)
  \diag{\nablaperp}
  \left( C^{\infty} \left( \Omega \right) \right) ^2
  \diag{- \nabla \cdot}
  C^{\infty} \left( \Omega \right).
\end{equation}
A fundamental result on the de-Rham complex, 
is the link between the dimension of the cohomology
groups of \eqref{eq:deRhamDifferentialForms}
and the Betti numbers \cite{allen2001algebraic},
which are characteristic of the topology of the domain:
\begin{equation}
  \label{eq:HarmonicGap}
  \left\{
  \begin{array}{r@{\, = \, }l}
    b_0 & \dim \left( \ker \nabla \right), \\
    b_1 & \dim \left( \ker \left( \nablaperp \cdot \right)
    / \Range \left( \nabla \right)\right), \\
    b_2 & \dim \left( C^{\infty} ( \Omega)
    / \Range \left( \nablaperp \cdot \right)\right). \\
  \end{array}
  \right.
\end{equation}
In this article, we focus on the two-dimensional torus $\T^2$,
  for which the Betti numbers are $b_0 = b_2 = 1$ and $b_1 = 2$.

\subsection{Conforming discrete de-Rham complex}

From an approximation point of view, the study of discrete
counterpart of the complex of \eqref{eq:deRhamDifferentialForms} has been
an intensive research topic over the last forty years, including the
work of Whitney~\cite{Whitney1957}, Bossavit and Hiptmair on electromagnetism
\cite{bossavit1988whitney,bossavit1998computational,bossavit1998geometry,
  hiptmair2002finite}, and
the work on formalization of Arnold and collaborators of finite element
exterior calculus
\cite{arnold2006finite,arnold2010finite} which led to the
reference book~\cite{arnold2018finite}. A classical conformal
discrete counterpart of the de-Rham diagram of
\eqref{eq:deRhamDifferentialForms} on triangles 
relies on the family of ${\cal P}^{-} _r \Lambda^k$
of \cite[Chapter 7]{arnold2018finite}, for which a key property
is preservation of the equalities \eqref{eq:HarmonicGap} at the discrete level. 
The family ${\cal P}^{-} _r \Lambda^k$ can then be translated back
into continuous/Nédélec/Raviart-Thomas complex
\cite{raviart1977mixed,raviart1977primal,nedelec1980mixed} for the
proxies. For example, the discrete counterpart of \eqref{eq:deRhamGradCurl} is:
\begin{equation}
  \label{eq:deRhamGradCurlDiscrete}
  \PP_{k+1}
  \diag{\nabla}
  \bN_k
  \diag{\nablaperp \cdot}
  \dd \PP_k,
\end{equation}
where $\PP_{k+1}$ is the continuous finite elements space of degree
$k+1$, $\bN_k$ is two-dimensional Nédélec approximation space
(vectorial finite element space that is tangential conforming), and
$\dd \PP_{k}$ is the discontinuous Galerkin approximation space.
The discrete counterpart of \eqref{eq:deRhamCurlDiv} is:
\begin{equation}
  \label{eq:deRhamCurlDivDiscrete}
  \PP_{k+1}
  \diag{\nablaperp}
  \RT_k
  \diag{- \nabla \cdot}
  \dd \PP_k,
\end{equation}
where $\RT_k$ is the two-dimensional Raviart-Thomas approximation space
(vectorial finite element space that is normal conforming).

\subsection{Nonconforming discrete de-Rham complex}

We begin by providing some definitions on the mesh we are working on. 
We work on the two-dimensional torus $\T^2$, which is supposed to
be oriented. This domain is divided into a set of cells which are either triangles or quadrangles.
The set of sides
of the mesh is denoted by $\cS$. Each side $S$ of the mesh is
  oriented once for all, and we denote by $\bn_S$ the clockwise normal to the
  side. Each side $S$ has two neighbouring cells, the one for which $\bn_S$
  is outgoing is the \emph{left cell}, and the one for
  which $\bn_S$ is ingoing is the \emph{right cell}.
The \emph{jump} of a scalar
$f$ that is continuous on the cells and discontinuous across the faces
is defined as
$$\forall S \in \cS \qquad \Jump{f}_S:= f_L - f_R,$$
where $f_L$ is the value on the left cell and $f_R$ is the value on the right
cell. 

Recently, it was proposed to relax the continuity
constraints induced by the conformal hypothesis \cite{licht2017complexes}.
Contrarily to the classical conformal approximation, the approximation of the
vector spaces 
proposed in \cite{licht2017complexes}
are completely discontinuous, without any
hypothesis  on the continuity of the normal component (which holds for
$\RT$) or the tangential component (which holds for $\bN$). 
This idea was extended in \cite{PerrierExact2024} for finding basis with
fewer degrees of freedom. 
The different approximation space for vectors proposed in
\cite{PerrierExact2024} will be denoted by
$\Bdiv_k$ and $\Bcurl_k$. 
In this case,
the classical derivation operators cannot be applied to these
approximation space, and the derivative in the sense of distributions is used
instead
$$
\forall \bu \in \Bdiv_k \qquad
\dnabla \cdot \bu :=
\left\{
\begin{array}{l@{\qquad}l}
  \forall C \in \cC &
  \left(\dnabla \cdot \bu \right)^C := \nabla \cdot \bu \\ 
  \forall S \in \cS &
  \left(\dnabla \cdot \bu \right)^S := -\Jump{\bu} \cdot \bn_S.\\ 
\end{array}
\right.
$$
and
$$
\forall \bu \in \Bcurl_k \qquad
\dnablaperp \cdot \bu :=
\left\{
\begin{array}{l@{\qquad}l}
  \forall C \in \cC &
  \left(\dnablaperp \cdot \bu \right)^C := \nablaperp \cdot \bu \\ 
  \forall S \in \cS &
  \left(\dnablaperp \cdot \bu \right)^S :=
  \Jump{\bu^\perp} \cdot \bn_S.\\ 
\end{array}
\right.
$$
We see that the image of the operators $\left( \dnabla \cdot \right)$ and
$\left( \dnablaperp \cdot \right)$ is in a Cartesian product of
cell space and face space, that we will denote in general $\bbC_k$. 
In \cite{PerrierExact2024}, several such vectorial approximation spaces
were proposed, for which we proved that the discrete and continuous
cohomology spaces are matching on $\T^2$. The proposed approximation spaces
may then be put in the following complex:
\begin{itemize}
\item If the choice of proxy is \eqref{eq:Proxy},
  \begin{equation}
    \label{eq:deRhamGradCurlDiscreteNC}
    \bbA_{k+1}
    \diag{\nabla}
    \Bcurl_k
    \diag{\dnablaperp \cdot}
    \bbC_k,
  \end{equation}
  where $\bbA_{k+1}$ is the space of continuous finite element spaces, i.e.
  either $\PP_{k+1}$ on triangular meshes, or
  $\QQ_{k+1}$ on quadrangular meshes, and $\Bcurl_k$ is discussed in the
  remaining part of the section.  
\item If the choice of proxy is \eqref{eq:ProxyOrtho}, then
  \begin{equation}
    \label{eq:deRhamCurlDivDiscreteNC}
    \bbA_{k+1}
    \diag{\nablaperp}
    \Bdiv_k
    \diag{- \dnabla \cdot}
    \bbC_k,
  \end{equation}
  where $\Bdiv_k$ is discussed in the
  remaining part of the section.  
\end{itemize}

Note that in \eqref{eq:deRhamCurlDivDiscreteNC} and \eqref{eq:deRhamGradCurlDiscreteNC}, the spaces $\bbC_k$ and $\bbA_{k+1}$ are the same, only
$\Bdiv_k$ and $\Bcurl_k$ are different.
In \cite{PerrierExact2024}, the following approximation spaces were proposed;
the first family is based on the conformal case, whereas the two other
proposed approximation spaces are optimal in the number of degrees of freedom. 
\begin{enumerate}
\item {\bf Approximation spaces directly based on discontinuous
  versions of Raviart-Tho\-mas/Nédélec finite elements.} In this version of the
  discontinuous spaces, proposed in \cite{licht2017complexes},
  the polynomial basis are exactly the same as for the
  conformal approximation on each  cell, but all the continuity
  constraints are relaxed. These spaces were denoted by
  $\dd \RT_k$ and $\dd \bN_k$. This leads to
  $$\bbC_k = \dd \PP_k ( \cC ) \times \dd \PP_k ( \cS ),$$
  for triangular meshes and to
  $$\bbC_k =  \dd \QQ_k ( \cC ) \times \dd \PP_k ( \cS ) ,$$
  for quadrangular meshes. These approximation space work fine, but
  some degrees of freedom are useless for ensuring the matching between
  the discrete and continuous cohomology groups.
  Also the Raviart-Thomas/Nédélec finite
  element basis should be generated which is not always straightforward on
  triangular meshes. 
\item {\bf Optimal approximation spaces on triangles.}
  On triangles, the classical approximation space for vectors, namely the one
  obtained by tensorizing the classical approximation space for
  scalars can be put in the diagrams \eqref{eq:deRhamGradCurlDiscreteNC}
  and \eqref{eq:deRhamCurlDivDiscreteNC}, as
  addressed in \cite{PerrierExact2024}; in
  this case, $\Bcurl_k = \Bdiv_k = \bdPP_k$, and
  $$\bbC_k = \dd \PP_{k-1} ( \cC ) \times \dd \PP_{k} ( \cS ).$$
\item {\bf Optimal approximation spaces on quadrangles.}
  The case of quadrangular meshes is quite surprising because the lowest order
  approximation space (namely with $\bbA_{k+1} = \QQ_1$) does not give the classical
  finite volume vector space, but rather an enriched version including
  three basis vector instead of two:
  $$\Bdiv_{0} = \Span \left(
  \left(
  \begin{array}{c}
    1 \\
    0
  \end{array}
  \right),
    \left(
  \begin{array}{c}
    0 \\
    1
  \end{array}
  \right),
  \left(
  \begin{array}{c}
    -x \\
    y
  \end{array}
  \right)
  \right),$$
  and
  $$\Bcurl_{0} = \Span \left(
  \left(
  \begin{array}{c}
    1 \\
    0
  \end{array}
  \right),
    \left(
  \begin{array}{c}
    0 \\
    1
  \end{array}
  \right),
  \left(
  \begin{array}{c}
    y \\
    x
  \end{array}
  \right)
  \right).$$
  In the general case, following \cite{PerrierExact2024}
  we denote by $\dd \QQ_{i,j}$ the space of polynomials
  of degree lower than $i$ in $x$ and lower than $j$ in $y$, and
  we set
  $$\Bdiv_k
  = \left(
    \begin{array}{c}
    \dd \QQ_{k,k} + \dd \QQ_{k+1,k-1}  \\
    \dd \QQ_{k,k} + \dd \QQ_{k-1,k+1}
    \end{array}
    \right)
  \oplus \Span \left(
  \begin{array}{c}
    - x^{k+1} y^k \\
    x^k y^{k+1}
  \end{array}
  \right)
  $$
  and
  $$\Bcurl_{k}
  =
  \left(
  \begin{array}{c}
    \dd \QQ_{k,k} + \dd \QQ_{k-1,k+1} \\
    \dd \QQ_{k,k} + \dd \QQ_{k+1,k-1}
  \end{array}
  \right)
  \oplus \Span \left(
  \begin{array}{c}
    x^k y^{k+1} \\
    x^{k+1} y^k \\
  \end{array}
  \right).
  $$
  Last, the space $\bbC_k$ is
  $$\bbC_k =
  \dd \widecheck{\QQ}_{k} ( \cC ) \times \dd \PP_{k-1} ( \cS ),
  $$
  with
  $$\dd \widecheck{\QQ}_{k} := \dd \QQ_{k,k-1} + \dd \QQ_{k-1,k},$$
  or more clearly, $\dd \widecheck{\QQ}_{k}$ contains all
  canonical monomials of 
  $\dd \QQ_k$ except for $x^k y^k$. As remarked in \cite{PerrierExact2024},
  the difference with the discontinuous Nédélec and Raviart-Thomas
  elements is, in this case, only of a single element, and the benefits
  of this optimal basis with respect to the classical one is less evident. 
\end{enumerate}

\subsection{Scalar products and adjoint curl and divergence}

The scalar product on the finite element spaces will be denoted with
$ \scalar{\cdot}{\cdot}$ with the finite element space as index, e.g.
$ \scalar{\cdot}{\cdot}_{\bbA_{k+1}}$. For the spaces $\bbA_{k+1}$,
$\Bcurl_{k}$ and $\Bdiv_{k}$, the scalar product is the standard
$L^2$ product. 

As the space $\bbC_k$ includes in general both components in the cell and
components in the face, the scalar product in $\bbC_k$ is the
graph scalar product:
$$\scalar{f}{g}_{\bbC_k}
:= \Sum{c \in \cC}{} \Int{c}{} f^C g^C + \Sum{S \in \cS}{} \Int{S}{} f^S g^S.$$
The definition of these scalar products allow to define the adjoint
differential operators as follows:

\begin{definition}[Adjoint differential operators]
  Based on the complex \eqref{eq:deRhamGradCurlDiscreteNC}, the adjoint
  gradient $\nabla^{\star}$ is defined as
  $$
  \begin{array}{l}
    \nabla^{\star} \, : \, \Bcurl_k \quad \longmapsto \quad \bbA_{k+1}
    \\
    \text{such that}\quad
  \forall \bu \in \Bcurl_k \quad
  \forall f \in \bbA_{k+1} \qquad
  \scalar{\nabla^{\star} \bu}{f}_{\bbA_{k+1}}
  =
  \scalar{\bu}{\nabla f}_{\Bcurl_k},
  \end{array}
  $$
  whereas the adjoint curl $\left( \nablaperp \cdot \right)^{\star}$ is defined
  as
  $$
  \begin{array}{l}
    \left( \nablaperp \cdot \right)^{\star} \, : \, \bbC_{k} \quad \longmapsto \quad \Bcurl_k
    \\
    \text{such that}\quad

    \forall \bu \in \Bcurl_k \quad
    \forall f \in \bbC_{k} \qquad
    \scalar{\left( \nablaperp \cdot \right) ^{\star} f }{\bu }_{\Bcurl_k}
    =
    \scalar{f }{  \nablaperp \cdot \bu }_{\bbC_k}.
  \end{array}
  $$
  Based on the complex \eqref{eq:deRhamCurlDivDiscreteNC}, the adjoint
  curl $\left( \nablaperp \right) ^{\star}$ is defined as
  $$
  \begin{array}{l}
    \left( \nablaperp \right) ^{\star} \, : \, \Bdiv_k \quad \longmapsto \quad \bbA_{k+1}
    \\
    \text{such that}\quad
  \forall \bu \in \Bdiv_k \quad
  \forall f \in \bbA_{k+1} \qquad
  \scalar{\left( \nablaperp \right) ^{\star} \bu}{f}_{\bbA_{k+1}}
  =
  \scalar{\bu}{\nablaperp f}_{\Bdiv_k},
  \end{array}
  $$
  whereas the adjoint divergence
  $\left( \nabla \cdot \right)^{\star}$ is defined
  as
  $$
  \begin{array}{l}
    \left( \nabla \cdot \right)^{\star} \, : \, \bbC_{k} \quad \longmapsto \quad \Bdiv_k
    \\
    \text{such that}\quad
  \forall \bu \in \Bdiv_k \quad
  \forall f \in \bbC_{k} \qquad
  \scalar{\left( \nabla \cdot \right) ^{\star} f }{\bu }_{\Bdiv_k}
  =
  \scalar{f }{  \nabla \cdot \bu }_{\bbC_k}.
  \end{array}
  $$
\end{definition}
Note that at the continuous level, the gradient and the divergence
are adjoint up to a sign, this means that at the discrete level,
$-\nabla^{\star}$ is an approximation of the divergence on
$\Bcurl_k$. In the same manner, $-\left( \nablaperp \right)^{\star}$ is
an approximation of the curl on $\Bdiv_k$.  

These adjoint differential operators play a key role in this article,
because they are the ones that will be
proven to be naturally preserved by the discontinuous Galerkin method. 

Last, we remark that it is possible to define these adjoint operators
as the composition of exterior derivative with Hodge-star operators.
This was done in the first version of this article \cite{firstVersion},
but was removed in the new version for clarity purpose.

\subsection{Projection operator in the $\bbC_k$ space}

  The spaces $\bbA$, $\Bcurl_k$ and
  $\Bdiv_k$ are rather usual approximation spaces, and
  the classical $L^2$ projector is used on these spaces. For these spaces, we
  denote by $\pi$ with the approximation space as an exponent the
  $L^2$ projection: $\pi^{\bbA_k}$, $\pi^{\Bcurl_k}$ and $\pi^{\Bdiv_k}$. \\
  \indent The approximation space $\bbC_k$ is less classical, as it is
  a Cartesian product of degrees of freedom on the cells and degrees of freedom
  on the faces, still it is similar as the approximation space of the
  Hybrid High Order methods \cite{di2016review,di2018introduction,di2020hybrid}.
  Therefore, following \cite{di2016review,di2018introduction,di2020hybrid}, we
  define the projection on $\bbC_k$ as follows
  \begin{definition}[Projection operator on $\bbC_k$]
    \label{def:projectionCk}
    Given a smooth function $g \in C^{\infty} \left( \Omega \right)$,
    the projection operator $\pi^{\bbC_k}$ is defined as
    $$
    \Sum{C \in \cC}{} \Int{C}{} g \varphi^C
    +
    \Sum{S \in \cS}{} \Int{S}{} g \varphi^S
    =
    \Sum{C \in \cC}{} \Int{C}{} \left( \pi^{\bbC_k} g \right)^C \varphi^C
    +
    \Sum{S \in \cS}{} \Int{S}{} \left( \pi^{\bbC_k} g \right)^S  \varphi^S.
    $$
  \end{definition}
  The projection operator defined in \autoref{def:projectionCk} ensures then
  the following proposition:
  \begin{prop}[Commutation property of the projection operator on $\bbC_k$]~ \\
    The following diagram
    \label{prop:projectionCk}
      \begin{center}
        \begin{tikzpicture}
          \node (Hdiv)  at (0,0)  {$\left( C^{\infty} \left( \Omega \right) \right)^2$};
          \node (L2)    at (3.5,0)  {$C^{\infty} \left( \Omega \right)$};
          \node (Bdivk) at (0,-2) {$\Bdiv_k$};
          \node (Ck)    at (3.5,-2) {$\bbC_k$};
          \draw[<-] (Hdiv) -- (L2) node[midway, above] {$\left( \nabla \cdot \right)^{\star}$};
          \draw[<-] (Bdivk) -- (Ck) node[midway, above] {$\left( \dnabla \cdot \right)^{\star}$};
          \draw[->] (Hdiv) -- (Bdivk) node[midway, right] {$\pi^{\Bdiv_k}$};
          \draw[->] (L2) -- (Ck) node[midway, right] {$\pi^{\bbC_k}$};
        \end{tikzpicture}
      \end{center}
      commutes, and so does the following diagram
      \begin{center}
        \begin{tikzpicture}
          \node (Hcurl)  at (0,0)  {$\left( C^{\infty} \left( \Omega \right) \right)^2$};
          \node (L2)    at (3.5,0)  {$C^{\infty} \left( \Omega \right)$};
          \node (Bcurlk) at (0,-2) {$\Bcurl_k$};
          \node (Ck)    at (3.5,-2) {$\bbC_k$};
          \draw[<-] (Hcurl) -- (L2) node[midway, above] {$\left( \nablaperp \cdot \right)^{\star}$};
          \draw[<-] (Bcurlk) -- (Ck) node[midway, above] {$\left( \dnablaperp \cdot \right)^{\star}$};
          \draw[->] (Hcurl) -- (Bcurlk) node[midway, right] {$\pi^{\Bcurl_k}$};
          \draw[->] (L2) -- (Ck) node[midway, right] {$\pi^{\bbC_k}$};
        \end{tikzpicture}
      \end{center}
  \end{prop}
  \begin{proof}
    The proof is performed for the first diagram, the proof for the second
    following the same lines. 
    We denote by $g$ an element of $C^{\infty} (\Omega)$. Then for any
    $\bv \in \left( C^{\infty} (\Omega) \right)^2$, we have
    $$\Int{\Omega}{} \nabla \cdot \left( g \bv \right)
    = \Int{\Omega}{} g \nabla \cdot \bv + \Int{\Omega}{} \nabla g \cdot \bv,$$
    but as $\Omega$ is the two dimensional torus, the left hand side is
    zero. This directly gives
    $$\left( \nabla \cdot \right)^{\star} g =
    - \nabla g.$$
    The projection of $\left(- \nabla g\right)$ on $\Bdiv_k$ ensures
    \begin{equation}
      \label{eq:pirondnabla}
      \forall \bPsi \in \Bdiv_k \qquad
      - \Sum{C \in \cC}{} \Int{C}{} \nabla g \cdot \bPsi
      =
      - \Sum{C \in \cC}{} \Int{C}{} \pi^{\Bdiv_k} \left( \nabla g \right)
      \cdot \bPsi
      =
      \Sum{C \in \cC}{} \Int{C}{} \pi^{\Bdiv_k} \circ \nabla^{\star}
      \left( g \right)
      \cdot \bPsi
      .
    \end{equation}
    On the other hand, the definition of the projection of $g$ on
    $\bbC_k$ ensures
    $$
    \forall \varphi \in \bbC_k \qquad 
    \Sum{C \in \cC}{} \Int{C}{} g \varphi^C
    +
    \Sum{S \in \cS}{} \Int{S}{} g \varphi^S
    =
    \Sum{C \in \cC}{} \Int{C}{} \left( \pi^{\bbC_k} g \right)^C \varphi^C
    +
    \Sum{S \in \cS}{} \Int{S}{} \left( \pi^{\bbC_k} g \right)^S  \varphi^S.
    $$
    Using the definition of $\left( \dnabla \cdot \right)^{\star}$ gives
    $$\forall \bPsi \in \Bdiv_k \qquad
    \Sum{C \in \cC}{} \Int{C}{} \left( \dnabla \cdot \right)^{\star}
    \circ \pi^{\bbC_k} \left( g \right) \cdot \bPsi
    =
    \Sum{C \in \cC}{} \Int{C}{} \left( \pi^{\bbC_k} g \right)^C
    \nabla \cdot \bPsi
    -
    \Sum{S \in \cS}{} \Int{S}{} \left( \pi^{\bbC_k} g \right)^S
    \Jump{\bPsi \cdot \bn_S}.
    $$
    Combining the two last equalities provides
    \begin{equation}
      \label{eq:nablarondpi}
      \forall \bPsi \in \Bdiv_k \qquad
      \Sum{C \in \cC}{} \Int{C}{} \left( \dnabla \cdot \right)^{\star}
      \circ \pi^{\bbC_k} \left( g \right) \cdot \bPsi
      =
      \Sum{C \in \cC}{} \Int{C}{}  g
      \nabla \cdot \bPsi
      -
      \Sum{S \in \cS}{} \Int{S}{}  g \Jump{\bPsi \cdot \bn}.
    \end{equation}
    It remains to apply the Stokes formula on each cell $C$ for getting
    $$
    \begin{array}{r@{\, = \, }l}
      - \Sum{C \in \cC}{} \Int{C}{}
      \nabla g \cdot \bPsi
      & - \Sum{C \in \cC}{} \Int{C}{}
      \nabla \cdot \left( g  \bPsi \right)
      + \Sum{C \in \cC}{} \Int{C}{} g \nabla \cdot \bPsi\\
      & - \Sum{C \in \cC}{} \Sum{S \in \partial C}{} \Int{S}{}
      g  \bPsi \cdot \bn^{\mathrm{out}}
      + \Sum{C \in \cC}{}  \Int{C}{}  g \nabla \cdot \bPsi\\
      & - \Sum{S \in \cS}{} \Int{S}{}
      \Jump{g  \bPsi \cdot \bn_S}
      + \Sum{C \in \cC}{}  \Int{C}{}  g \nabla \cdot \bPsi\\
      - \Sum{C \in \cC}{} \Int{C}{}
      \nabla g \cdot \bPsi
      & - \Sum{S \in \cS}{} \Int{S}{}
      g \Jump{\bPsi \cdot \bn_S}
      + \Sum{C \in \cC}{}  \Int{C}{} g \nabla \cdot \bPsi.\\
    \end{array}
    $$
    This last equality provides the equality of the left hand side of
    \eqref{eq:pirondnabla} with the right hand side of
    \eqref{eq:nablarondpi}, which induces the equality of the right hand side
    of \eqref{eq:pirondnabla} and the right hand side of
    \eqref{eq:nablarondpi}, which ends the proof.
  \end{proof}
  Note that the second commutative diagram of \autoref{prop:projectionCk}
  will be relevant for initializing divergence free fields in the
  discrete adjoint
  sense in \autoref{sec:LieDerivative}; especially, starting from a
  $g \in C^{\infty} (\Omega)$ the path
  $\pi^{\Bcurl_k} \circ  \left( \nablaperp \cdot \right)^{\star} (g)$ ensures 
  a high order approximation of $\nablaperp g$, whereas
  the path $\left( \dnablaperp \cdot \right)^{\star} \circ  \pi^{\bbC_k} (g)$
  ensures that the discrete field is divergence free in the
  discrete adjoint sense.

\section{Discrete preservation of the curl or the divergence}
\label{sec:DiscreteCurlDiv}

\subsection{Numerical scheme}

In this section, we consider a system of conservation law in which one of the
unknowns is a vector $\bu$.
We denote by $\barG$ the flux of the conservation law involving $\bu$
\begin{equation}
  \label{eq:consLawVector}
  \partial_t \bu + \bnabla \cdot \barG = 0.
\end{equation}
We consider the numerical resolution of \eqref{eq:consLawVector} on a
periodic domain. The discontinuous Galerkin method reads 
\begin{equation}
  \label{eq:DGConsLawVector}
  \text{Find $\bu \in \bbB$} \quad \forall \bv \in \bbB \qquad 
  \Sum{C \in \cC}{}
  \Int{C}{}
  \bv \cdot \partial_t \bu
  -\Sum{C \in \cC}{}
  \Int{C}{}
  \barG : \bnabla \bv
  +
  \Sum{S \in \cS}{}
  \Int{S}{}
  \Jump{\bv} \cdot \bGtilde=0,
\end{equation}
where $\bGtilde$ is the numerical flux, and $\bbB$ is the approximation
space for vectors that is not defined for the moment.

\subsection{Discrete preservation of the vorticity}

In this section, we consider the case in which $\barG$ is proportional
to the identity $\barG = g \mathrm{I}_d$, which means that 
we have $\bnabla \cdot \barG = \nabla g$. This means that
the conservation law on $\bu$ \eqref{eq:consLawVector} may be simplified as
\begin{equation}
  \label{eq:uCurl}
  \partial_t \bu + \nabla g = 0,
\end{equation}
where $g$ is a scalar that may depend nonlinearly on the variables of the
system. Then the discontinuous Galerkin discretization reads
\begin{equation}
  \label{eq:DGCurl}
  \text{Find $\bu \in \Bdiv_k$} \quad \forall \bv \in \Bdiv_k \qquad 
  \Sum{C \in \cC}{}
  \Int{C}{}
  \bv \cdot \partial_t \bu
  -\Sum{C \in \cC}{}
  \Int{C}{}
  g \nabla \cdot \bv
  +
  \Sum{S \in \cS}{}
  \Int{S}{}
  \Jump{\bv} \cdot \bGtilde=0,
\end{equation}
where $\bGtilde$ is the numerical flux. Equation \eqref{eq:uCurl} induces
formally the conservation of  $\nablaperp \cdot \bu$,
and the discontinuous Galerkin
method \eqref{eq:DGCurl} may have a similar property summarized in the following
proposition. 

\begin{prop}[Conservation of $\left( \nablaperp \bu \right)^{\star}$]
  \label{prop:curlConservation}
  Consider the numerical scheme \eqref{eq:DGCurl}. If
  $\bGtilde$ is parallel to $\bn_S$,
  then the numerical scheme \eqref{eq:DGCurl} induces
  $\partial_t \left( \nablaperpStar \bu \right) = 0$.
\end{prop}

\begin{proof}
  If $\bGtilde$ is parallel to $\bn_S$, then the numerical flux may be
  rewritten
  $$\bGtilde = \gtilde \bn_S,$$
  where $\gtilde$ is a scalar, 
  so that the numerical scheme \eqref{eq:DGCurl} may be rewritten
  \begin{equation*}
    \text{Find $\bu$} \quad \forall \bv \qquad 
    \Sum{C \in \cC}{}
    \Int{C}{}
    \bv \cdot \partial_t \bu
    -\Sum{C \in \cC}{}
    \Int{C}{}
    g \nabla \cdot \bv
    +
    \Sum{S \in \cS}{}
    \Int{S}{}
    \Jump{\bv \cdot \bn_S} \gtilde=0. 
  \end{equation*}
  In this last equation, suppose that $\bv = \nablaperp f$ for
  $f \in \bbA_{k+1}$.
  Then the  normal jump vanishes at each side:
  $$\forall S \in \cS \qquad \Jump{\bv \cdot \bn_S} = 0,$$
  and on each cell, its strong divergence is $0$:
  $$\nabla \cdot  \bv = 0,$$
  so that
  $$\forall f \in  \bbA_{k+1} \qquad 
  \Sum{C \in \cC}{}
  \Int{C}{}
  \left( \nablaperp f \right)  \cdot \partial_t \bu = 0.
  $$
  This last equation means that
  $$\forall f \in \bbA_{k+1} \qquad
  \scalar{\nablaperp f}{\partial_t \bu}_{\Bdiv_k} = 0.$$
  By definition of $\nablaperpStar$, we get
  $$\forall f \in \bbA_{k+1} \qquad
  \scalar{f}{\nablaperpStar \left(  \partial_t  \bu \right)}_{\Bdiv_k}
  = 0.$$
  It remains to prove that
  $\nablaperpStar \left( \partial_t \bu \right)
  = \partial_t \left( \nablaperpStar \bu \right)$.
  Denoting by $f$ an element of $\bbA_{k+1}$, we have
  $$
  \begin{array}{r@{\, = \, }l}
    \scalar{\partial_t \left( \nablaperpStar \bu \right)}{f}_{\bbA_{k+1}}
    &
    \Int{\Omega}{} \partial_t \left(\nablaperpStar \bu \right) f \\
    &
    \partial_t \left( \Int{\Omega}{} \left( \nablaperpStar \bu\right) f \right)
    \\
    &
    \partial_t \left( \Int{\Omega}{}  \bu \cdot \nablaperp f \right)
    \\
    &
    \Int{\Omega}{}  \left( \partial_t \bu \right) \cdot \nablaperp f
    \\
    &
    \scalar{\partial_t \bu}{\nablaperp f}_{\Bdiv_k} \\
    \scalar{\partial_t \left( \nablaperpStar \bu \right)}{f}_{\bbA_{k+1}}
    &
    \scalar{\nablaperpStar \left( \partial_t \bu \right)}{f}_{\bbA_{k+1}}, \\
  \end{array}
  $$
  which ends the proof. 
\end{proof}

\begin{remark}[full discretization in space: effect of quadrature formula]
    The \autoref{prop:curlConservation} was proven in the non fully spatial
    discretized case, namely without the use of quadrature formula. It is
    important to note that both the cell and face integral of the spatial
    discretisation are vanishing \emph{pointwise}, therefore they 
    vanish whatever the quadrature formula used. Concerning the
    adjoint curl that is preserved, the numerical computation of 
    $\left( \nablaperp \right)^{\star}$ should take exactly the same quadrature
    formula as the one used for computing  the
    mass matrix. For example, if the mass matrix is lumped, then
    the adjoint curl should be defined with the same lumped mass matrix. 
\end{remark}

\begin{remark}[Time discretization]
  \autoref{prop:curlConservation} adresses the semi-discrete case (discrete
  in space, but not in time). A proof for the fully discrete case can be done
  by following exactly the same lines as the proof of
  \autoref{prop:curlConservation}, by replacing the time derivative by its
  discretization, still this leads to a cumbersome proof. Actually, the
  extension to the time discrete case is a simple consequence
  of \cite[Section 3]{mclachlan2024functional}, which states that if
  the semi-discrete scheme preserves the discrete curl, then all
  $B$-series (Butcher series) \cite{hairer1974butcher,mclachlan2016b}
  time integrators (including all Runge-Kutta ones),
  also preserve the discrete curl.  
\end{remark}

\begin{remark}[Lax-Friedrich flux with normal diffusion]
  A widely used numerical flux is the Lax-Friedrich flux which reads
\begin{equation}
  \label{eq:LaxFriedrich}
  \bGtilde \left( \barG_L, \bu_L,\barG_R,\bu_R,\bn_S \right)
  = \dfrac{\barG_L \bn_S + \barG_R \bn_S }{2}
  + \dfrac{\lambda}{2} \left( \bu_L - \bu_R \right),
\end{equation}
where $\lambda$ is the maximum absolute value of the eigenvalues of the system.
Dealing with the
system \eqref{eq:uCurl}, the numerical flux \eqref{eq:LaxFriedrich}
can be simplified as
\begin{equation*}
  \bGtilde \left( g_L, \bu_L,g_R,\bu_R,\bn_S \right)
  = \dfrac{g_L \bn_S + g_R \bn_S}{2}
  + \dfrac{\lambda}{2} \left( \bu_L - \bu_R \right). 
\end{equation*}
The centered part of the flux is clearly parallel to the normal, and should
not be modified because it ensures the consistency of the numerical scheme.
The diffusive part may be decomposed into its normal and tangential part as
$$
\left( \bu_L - \bu_R \right)
= \bn_S \bn_S ^T \left( \bu_L - \bu_R \right)
+ \left( \mathrm{I}_d - \bn_S \bn_S ^T \right)
\left( \bu_L - \bu_R \right).
$$
Therefore an easy way to ensure the hypothesis of
\autoref{prop:curlConservation} is to use the following
\emph{Lax-Friedrich flux with purely normal diffusion}
\begin{equation}
  \label{eq:LaxFriedrichNormal}
  \bGtilde \left( g_L, \bu_L,g_R,\bu_R,\bn_S \right)
  = \dfrac{g_L \bn_S + g_R \bn_S}{2}
  + \dfrac{\lambda \bn_S \bn_S ^T }{2} \left( \bu_L - \bu_R \right). 
\end{equation}
\end{remark}

\begin{remark}[Link with the Hodge Laplacian]
  \label{rem:HodgeLaplacian}
  Putting diffusion operator to the equation \eqref{eq:consLawVector}
  consists, at least formally, in adding a vector Laplacian
  $\boldsymbol{\Delta}$ to the right hand side:
  \begin{equation}
    \label{eq:LaplaceVector}
    \boldsymbol{\Delta} \bu
    =
    \left(
    \begin{array}{c}
      \partial_{xx} \bu_x + \partial_{yy} \bu_x \\
      \partial_{xx} \bu_y + \partial_{yy} \bu_y 
    \end{array}
    \right)
    = \nabla \left( \nabla \cdot \bu \right)
    +\nablaperp \left( \nablaperp \cdot \bu \right).
  \end{equation}
  For ensuring preservation of the vorticity, it is sufficient to
  put a diffusion that preserves also the vorticity, namely only
  the $\nabla \left( \nabla \cdot \bu \right)$ component.
  In \cite[Remark 4]{firstVersion}, we proved that the diffusion induced
  by \eqref{eq:LaxFriedrichNormal} matches with a discrete version
  of the operator $\nabla \left( \nabla \cdot \bu \right)$.
\end{remark}

\subsection{Discrete preservation of the irrotational component}

In this section, we consider the case in which $\barG$ is antisymmetric;
in dimension 2,
$\barG$ may be written as
$$\left(
\begin{array}{cc}
  0 & -g \\
  g & 0 
\end{array}
\right),$$
where $g$ is a scalar that may depend nonlinearly on the variables of the
system. 
This means that $\bu$ ensures the following conservation law
\begin{equation}
  \label{eq:uDiv}
  \partial_t \bu + \nablaperp g  = 0.
\end{equation}
Then the discontinuous Galerkin discretization reads
\begin{equation}
  \label{eq:DGDiv}
  \text{Find $\bu \in \Bcurl_k$} \quad \forall \bv \in \Bcurl_k \qquad 
  \Sum{C \in \cC}{}
  \Int{C}{}
  \bv \cdot \partial_t \bu
  - \Sum{C \in \cC}{}
  \Int{C}{}
  g \nablaperp \cdot  \bv
  +
  \Sum{S \in \cS}{}
  \Int{S}{}
  \Jump{\bv} \cdot \bGtilde=0,
\end{equation}
where $\bGtilde$ is the numerical flux. 

\begin{prop}[Adjoint conservation of the divergence]
  \label{prop:divConservation}
  Consider the numerical scheme \eqref{eq:DGDiv}. 
  If $\bGtilde$ is orthogonal to $\bn_S$, then $\nablaStar \bu$ is
  conserved by the numerical scheme.
\end{prop}
The proof of \autoref{prop:divConservation} follows exactly the same
lines as \autoref{prop:curlConservation} and is not repeated. 

\begin{remark}[Lax-Friedrich flux with purely tangential diffusion]

  The Lax-Friedrich numerical flux \eqref{eq:LaxFriedrich} for
  \eqref{eq:uDiv} may be simplified as
  \begin{equation*}
    \bGtilde \left( \bg_L, \bu_L,\bg_R,\bu_R,\bn_S \right)
    = \dfrac{ g_L \bn_S ^\perp +  g_R \bn_S ^\perp}{2}
    + \dfrac{\lambda}{2} \left( \bu_L - \bu_R \right). 
  \end{equation*}
  The centered part of the flux, which is clearly orthogonal to $\bn_S$,
  ensures the consistency of the numerical scheme.
  Still relying on the decomposition of the diffusive part into
  its normal and tangential part, an easy way to ensure the hypothesis of
  \autoref{prop:divConservation} consists in using the following
  \emph{Lax-Friedrich flux with purely tangential diffusion}
  \begin{equation}
    \label{eq:LaxFriedrichTangential}
    \bGtilde \left( \bg_L, \bu_L,\bg_R,\bu_R,\bn_S \right)
    = \dfrac{ g_L \bn_S ^\perp +  g_R \bn_S ^\perp}{2}
    + \dfrac{\lambda \left( \mathrm{I}_d - \bn_S \bn_S ^T \right) }{2}
    \left( \bu_L - \bu_R \right). 
  \end{equation}
  Similarly to \autoref{rem:HodgeLaplacian}, we proved in
  \cite{firstVersion} that the diffusion of \eqref{eq:LaxFriedrichTangential}
  matches with a discrete version of the operator
  $\nablaperp \left( \nablaperp \cdot \right)$, which is the component of
  the vector Laplacian \eqref{eq:LaplaceVector} that preserves the divergence. 
\end{remark}

\section{Preservation of initially divergence free field for the
  induction equation}
\label{sec:LieDerivative}

In this section, we are interested in the induction equation, which
reads, in dimension $3$:
$$\partial_t \bu +  \nabla_{3d} \times \left( \bb \times \bu \right)
+ \bb \nabla_{3d} \cdot \bu = 0,$$
where $\bu$ is the unknown, $\bb$ is a given vector field, and
$\nabla_{3d}$ denotes the  three dimensional
nabla operator. We suppose that the
two vectors $\bb$ and $\bu$ have only a $x$ and a $y$ component, and
we suppose also that these components depend only on the variables $x$ and
$y$. Then 
$$
\begin{array}{r@{\, = \, }l}
\nabla_{3d} \times \left( \bb \times \bu \right)
&
\nabla_{3d} \times
\left(
\left(
\begin{array}{c}
  \bb_x \\
  \bb_y \\
  0 \\
\end{array}
\right)
\times
\left(
\begin{array}{c}
  \bu_x \\
  \bu_y \\
  0 \\
\end{array}
\right)
\right) \\
& \nabla_{3d} \times \left(
\begin{array}{c}
  0 \\
  0 \\
  \bb_x \bu_y - \bb_y \bu_x \\
\end{array}
\right) \\
& \left(
\begin{array}{c}
  \partial_y \left( \bb_x \bu_y - \bb_y \bu_x \right) \\
  - \partial_x \left( \bb_x \bu_y - \bb_y \bu_x \right) \\
   0\\
\end{array}
\right) \\
\nabla_{3d} \times \left( \bb \times \bu \right)
& \left(
\begin{array}{c}
  \nablaperp \left( \bb \cdot \bu^{\perp}  \right) \\
  0
\end{array}
\right). \\
\end{array}
$$
As a consequence, the two dimensional induction equation reads
\begin{equation}
  \label{eq:InductionEquation}
  \partial_t \bu   + \nablaperp \left( \bb \cdot \bu^{\perp}  \right)
  + \bb \nabla \cdot \bu= 0.
\end{equation}
This equation formally ensures the following advection equation on
$\nabla \cdot \bu$:
\begin{equation}
  \label{eq:ConservationDivergence}
  \partial_t \left( \nabla \cdot \bu \right)
  +  \nabla \cdot \left( 
  \left( \nabla \cdot \bu \right) \bb \right)
  = 0,
\end{equation}
which means that if $\nabla \cdot \bu$ is initially equal to $0$, it is always
equal to $0$. 
We are now interested in the discretization of
\eqref{eq:InductionEquation}.  
In \autoref{prop:divConservation}, we proved that
$\left( \nabla \right)^{\star}$ could be preserved by the discontinuous
Galerkin method provided the approximation space is $\Bcurl_k$ and the
numerical flux is tangential to the normal. This suggests to propose
 the following numerical scheme: 
\begin{equation}
  \label{eq:DiscreteInduction}
  \begin{array}{l}
    \text{find $\bu \in \Bcurl_k$ such
      that for all $\bv \in \Bcurl_k$} \\
    \qquad 
    \Sum{c \in \cC}{} \Int{c}{}
    \bv \cdot \partial_t \bu
    -
    \Sum{c \in \cC}{} \Int{c}{} \left( \bb \cdot  \bu^{\perp} \right)
    \nablaperp \cdot \bv
    +
    \Sum{S \in \cS}{} \Int{S}{} \Jump{\bv} \cdot \widetilde{\bG}
    +
    \Sum{c \in \cC}{} \Int{c}{} \bv \cdot \bb \left(
    - \nabla^\star \bu
    \right)
    = 0,
  \end{array}
\end{equation}
for which we can prove
\begin{prop}[Equation on $-\nabla^\star \bu$]
  \label{prop:EquationOnDivergence}
  If $\bu$ is solution of \eqref{eq:DiscreteInduction}, and if
  the numerical flux is orthogonal to $\bn_S$, then $-\nabla^\star \bu$ is
  solution of the continuous Galerkin discretization for the advection
  equation on the divergence
  \eqref{eq:ConservationDivergence}
  \begin{equation}
    \label{eq:AdvectionCG}
    \text{Find }D \in \bbA_{k+1} \quad \forall f \in \bbA_{k+1}
    \qquad 
    \Sum{c \in \cC}{} \Int{c}{}
    f \partial_t D
    - \Sum{c \in \cC}{}
    \Int{c}{} D \bb \cdot \nabla f = 0.
  \end{equation}
\end{prop}

\begin{proof}
  We consider \eqref{eq:DiscreteInduction}, which is tested with
  $\bv  = \nabla f$, where $f \in \bbA_{k+1}$. Then
  $$\nablaperp \cdot \bv = \nablaperp \cdot \left( \nabla f \right) = 0.$$
  As $\widetilde{\bG}$ is orthogonal to $\bn_S$, it may be rewritten
  $\tilde{g} \bn_S ^{\perp}$. This leads to
  $$
  \Jump{\bv} \cdot \widetilde{\bG}
  =
  \Jump{\bv} \cdot \tilde{g} \bn_S ^{\perp}
  =
  \Jump{\bv \cdot \bn_S ^\perp} \tilde{g},
  $$
  and if $f \in \bbA_{k+1}$, then
  $\Jump{\nabla f \cdot \bn_S ^{\perp}} = 0$.
  Therefore \eqref{eq:DiscreteInduction} tested with $\nabla f $ for
  $f \in \bbA_{k+1}$ gives
  \begin{equation}
    \label{eq:DivTemp}
    \forall f \in \bbA_{k+1} \qquad
    \Sum{c \in \cC}{} \Int{c}{}
    \nabla f \cdot \partial_t \bu
    +
    \Sum{c \in \cC}{} \Int{c}{} \nabla f  \cdot \bb \left(
    - \nabla^\star \bu
    \right)
    = 0.
  \end{equation}
  The time derivative term may be transformed as
  (the commutation between $\nabla^\star$ and $\partial_t$ can be proven
  as in the proof of \autoref{prop:curlConservation}):
  $$
  \begin{array}{r@{\, = \, }l}
    \Sum{c \in \cC}{} \Int{c}{}
  \nabla f \cdot \partial_t \bu
  & \scalar{\nabla f }{\partial_t \bu }_{\Bcurl_k} \\
  & \scalar{ f }{\nabla^\star \left( \partial_t \bu \right) }_{\bbA_{k+1}} \\
  & \scalar{ f }{\partial_t \nabla^\star \left( \bu \right) }_{\bbA_{k+1}} \\
  \Sum{c \in \cC}{} \Int{c}{}
  \nabla f \cdot \partial_t \bu
  & - \scalar{ f }
         {\partial_t \left( - \nabla^\star \left( \bu \right) \right)
         }_{\bbA_{k+1}}. \\
  \end{array}
  $$
  Equation \eqref{eq:DivTemp} becomes finally
  $$
  \forall f \in \bbA_{k+1} \qquad
  \Sum{c \in \cC}{} \Int{c}{}
  f \partial_t \left( - \nabla^\star \left( \bu \right) \right)
  -
  \Sum{c \in \cC}{} \Int{c}{} \nabla f  \cdot \bb \left(
  - \nabla^\star \bu
  \right)
  = 0,
  $$
  which means that $\left( - \nabla^\star  \bu \right)$ is
  solution of \eqref{eq:AdvectionCG}.
\end{proof}
Note that $- \nabla^\star \bu$ is solution of the
numerical scheme \eqref{eq:AdvectionCG}, which is known
to be unstable with explicit time stepping, as it leads to
  central-difference type approximations of differential
  operators \cite{brooks1982streamline}. However, it will be used in practical applications
with initial condition
$\nabla^\star \bu (t=0) = 0$, so that this instability should not raise any
problem provided the initialization is carefully performed.

\begin{prop}[Exact preservation of $\nabla^\star \bu = 0$]
  We consider equation \eqref{eq:InductionEquation}. Suppose
  that $\bu(t=0) = \bu^0$ derives from a potential
  $f^0$: $\bu^0 = \nablaperp f^0$, 
  and we denote by $f_h ^0$ the projection of $f^0$ in $\bbC_k$
  as defined in \autoref{def:projectionCk}. Then
  $\bu_h ^0 := -\left( \dnablaperp \cdot \right)^{\star} f_h ^0$ is such that
  $\nabla^\star \bu_h ^0 = 0$, and under hypothesis of
  \autoref{prop:EquationOnDivergence}, the constraint
  $\nabla^\star \bu_h = 0$ is preserved for all time by the numerical
  scheme~\eqref{eq:DiscreteInduction}.
\end{prop}

\begin{proof}
  The beginning of the proof relies only on relations between the kernel and
  range of an operator and its adjoint. Note that we are working here in
  finite dimension, so that these relations do not include closures of
  range or kernel. 
 
  If $f_h ^0 \in \bbC_k$, then
  $\left( \dnablaperp \cdot \right)^{\star} f_h ^0 \in \left( \ker
  \left( \dnablaperp \cdot
  \right) \right)^\perp$. But as $\Range \nabla \subset
  \ker \left( \dnablaperp \cdot
  \right)$, we also have
  $\left( \dnablaperp \cdot \right)^{\star} f_h ^0 \in
  \left( \Range \nabla \right)^\perp$. As
  $\left( \Range \nabla \right)^\perp = \ker \nabla^\star$, we have
  $\left( \dnablaperp \cdot \right)^{\star} f_h ^0 \in  \ker \nabla^\star$, and so
  $\nabla^\star \bu_h ^0 = 0$.

  Then by \autoref{prop:EquationOnDivergence},
  $\nabla^\star \bu_h$ ensures the numerical scheme \eqref{eq:AdvectionCG} which
  preserves $\nabla^\star \bu_h = 0$ for all time, which ends the proof.
\end{proof}

\begin{remark}[Divergence free initialization]
  \label{rem:DivergenceInit}
  We provide some details about how the initial condition is computed.
  Usually, the computation of the initial condition $\bu_h ^0$ is computed by
  projection of the initial condition $\bu^0$ on the finite element space. This
  leads to the following system to solve
  $$M \bu_h ^0 = \bRHS,$$
  where $M$ is the mass matrix of the finite element space of the velocity space:
  if $\bv$ denotes a basis of this finite element basis, the mass matrix is
  \begin{equation}
    \label{eq:MassMatrix}
    M_{i,j} = \Int{\Omega}{} \bv_i \cdot \bv_j,
  \end{equation}
  and $\bRHS$ is the right hand side, equal to
  \begin{equation}
    \label{eq:RHSInit}
    \bRHS_i = \Int{\Omega}{} \bu^0 \cdot \bv_i.
  \end{equation}
  Suppose now that the initial condition derives from a potential vector
  $f$: $\bu^0 = \nablaperp f^0$. Then at the discrete level, we wish to have
  $\bu_h ^0 = - \left( \dnablaperp \cdot \right)^{\star} f_h ^0$, namely
  $$\forall \bv \in \Bcurl_k
  \qquad
  \scalar{\bu_h ^0}{\bv}_{\Bcurl_k}
  = - \scalar{\left( \dnablaperp \cdot \right)^{\star}
    f_h ^0}{\bv}_{\Bcurl_k}
  =
  - \scalar{f_h ^0}{\dnablaperp \cdot \bv}_{\bbC_k}.$$
  This gives
  $$\forall \bv \in \Bcurl_k
  \qquad
  \scalar{\bu_h ^0}{\bv}_{\Bcurl_k}
  =
  - \Sum{c \in \cC}{} \Int{c}{} 
  f_h ^0 \nablaperp \cdot  \bv
  + \Sum{S \in \cS}{} \Int{S}{} 
  f_h ^0 \Jump{\bv \cdot \bn_S ^{\perp}}.
  $$
  Therefore, the matrix of the
  system to solve is still \eqref{eq:MassMatrix}, but the right hand side
  is no more \eqref{eq:RHSInit}, but
  $$\bRHS_i
  =
  - \Sum{c \in \cC}{} \Int{c}{} 
  f_h ^0 \nablaperp \cdot  \bv_i
  + \Sum{S \in \cS}{} \Int{S}{} 
  f_h ^0 \Jump{\bv_i \cdot \bn_S ^{\perp}}.
  $$
\end{remark}

\section{Numerical results}
\label{sec:numericalResults}

For all the numerical tests, the computational domain is the periodic
square domain $[0,1]^2$. We will consider three types of
mesh: Cartesian, unstructured quadrangular and unstructured triangular.
The meshes on which the conservation tests will be performed are represented
in \autoref{fig:Meshes}.

\begin{figure}
  \begin{center}
    \begin{tabular}{c@{\qquad}c}
\begin{tikzpicture}[scale=5,line width=0.8,color=blue!70]
    \coordinate (A1) at (0,0) {};
    \coordinate (A2) at (1,0) {};
    \coordinate (A3) at (1,1) {};
    \coordinate (A4) at (0,1) {};
    \coordinate (A5) at (0.09999999999975993,0) {};
    \coordinate (A6) at (0.1999999999995199,0) {};
    \coordinate (A7) at (0.2999999999993242,0) {};
    \coordinate (A8) at (0.3999999999991729,0) {};
    \coordinate (A9) at (0.4999999999990216,0) {};
    \coordinate (A10) at (0.599999999999148,0) {};
    \coordinate (A11) at (0.6999999999992742,0) {};
    \coordinate (A12) at (0.7999999999994699,0) {};
    \coordinate (A13) at (0.899999999999735,0) {};
    \coordinate (A14) at (1,0.09999999999975993) {};
    \coordinate (A15) at (1,0.1999999999995199) {};
    \coordinate (A16) at (1,0.2999999999993242) {};
    \coordinate (A17) at (1,0.3999999999991729) {};
    \coordinate (A18) at (1,0.4999999999990216) {};
    \coordinate (A19) at (1,0.599999999999148) {};
    \coordinate (A20) at (1,0.6999999999992742) {};
    \coordinate (A21) at (1,0.7999999999994699) {};
    \coordinate (A22) at (1,0.899999999999735) {};
    \coordinate (A23) at (0.9000000000001388,1) {};
    \coordinate (A24) at (0.8000000000002776,1) {};
    \coordinate (A25) at (0.7000000000005551,1) {};
    \coordinate (A26) at (0.6000000000009714,1) {};
    \coordinate (A27) at (0.5000000000013878,1) {};
    \coordinate (A28) at (0.4000000000012489,1) {};
    \coordinate (A29) at (0.3000000000011102,1) {};
    \coordinate (A30) at (0.2000000000008326,1) {};
    \coordinate (A31) at (0.1000000000004163,1) {};
    \coordinate (A32) at (0,0.9000000000001388) {};
    \coordinate (A33) at (0,0.8000000000002776) {};
    \coordinate (A34) at (0,0.7000000000005551) {};
    \coordinate (A35) at (0,0.6000000000009714) {};
    \coordinate (A36) at (0,0.5000000000013878) {};
    \coordinate (A37) at (0,0.4000000000012489) {};
    \coordinate (A38) at (0,0.3000000000011102) {};
    \coordinate (A39) at (0,0.2000000000008326) {};
    \coordinate (A40) at (0,0.1000000000004163) {};
    \coordinate (A41) at (0.09999999999982558,0.1000000000003507) {};
    \coordinate (A42) at (0.09999999999989122,0.2000000000007014) {};
    \coordinate (A43) at (0.09999999999995685,0.3000000000009316) {};
    \coordinate (A44) at (0.1000000000000225,0.4000000000010413) {};
    \coordinate (A45) at (0.1000000000000881,0.5000000000011512) {};
    \coordinate (A46) at (0.1000000000001538,0.6000000000007891) {};
    \coordinate (A47) at (0.1000000000002194,0.7000000000004268) {};
    \coordinate (A48) at (0.100000000000285,0.8000000000001967) {};
    \coordinate (A49) at (0.1000000000003507,0.9000000000000985) {};
    \coordinate (A50) at (0.1999999999996512,0.100000000000285) {};
    \coordinate (A51) at (0.1999999999997824,0.2000000000005701) {};
    \coordinate (A52) at (0.1999999999999137,0.3000000000007529) {};
    \coordinate (A53) at (0.2000000000000449,0.4000000000008337) {};
    \coordinate (A54) at (0.2000000000001763,0.5000000000009144) {};
    \coordinate (A55) at (0.2000000000003075,0.6000000000006066) {};
    \coordinate (A56) at (0.2000000000004388,0.7000000000002989) {};
    \coordinate (A57) at (0.2000000000005701,0.8000000000001161) {};
    \coordinate (A58) at (0.2000000000007013,0.9000000000000579) {};
    \coordinate (A59) at (0.2999999999995027,0.1000000000002194) {};
    \coordinate (A60) at (0.2999999999996814,0.2000000000004388) {};
    \coordinate (A61) at (0.2999999999998599,0.3000000000005744) {};
    \coordinate (A62) at (0.3000000000000386,0.4000000000006261) {};
    \coordinate (A63) at (0.3000000000002172,0.5000000000006779) {};
    \coordinate (A64) at (0.3000000000003958,0.6000000000004244) {};
    \coordinate (A65) at (0.3000000000005744,0.7000000000001708) {};
    \coordinate (A66) at (0.3000000000007529,0.8000000000000355) {};
    \coordinate (A67) at (0.3000000000009316,0.9000000000000177) {};
    \coordinate (A68) at (0.3999999999993805,0.1000000000001538) {};
    \coordinate (A69) at (0.399999999999588,0.2000000000003075) {};
    \coordinate (A70) at (0.3999999999997957,0.3000000000003958) {};
    \coordinate (A71) at (0.4000000000000034,0.4000000000004186) {};
    \coordinate (A72) at (0.4000000000002109,0.5000000000004412) {};
    \coordinate (A73) at (0.4000000000004186,0.600000000000242) {};
    \coordinate (A74) at (0.4000000000006261,0.7000000000000427) {};
    \coordinate (A75) at (0.4000000000008337,0.7999999999999546) {};
    \coordinate (A76) at (0.4000000000010413,0.8999999999999773) {};
    \coordinate (A77) at (0.4999999999992583,0.1000000000000881) {};
    \coordinate (A78) at (0.4999999999994948,0.2000000000001763) {};
    \coordinate (A79) at (0.4999999999997315,0.3000000000002172) {};
    \coordinate (A80) at (0.4999999999999681,0.4000000000002109) {};
    \coordinate (A81) at (0.5000000000002046,0.5000000000002048) {};
    \coordinate (A82) at (0.5000000000004412,0.6000000000000596) {};
    \coordinate (A83) at (0.5000000000006779,0.6999999999999148) {};
    \coordinate (A84) at (0.5000000000009146,0.7999999999998739) {};
    \coordinate (A85) at (0.5000000000011511,0.8999999999999369) {};
    \coordinate (A86) at (0.5999999999993305,0.1000000000000225) {};
    \coordinate (A87) at (0.5999999999995126,0.2000000000000449) {};
    \coordinate (A88) at (0.599999999999695,0.3000000000000386) {};
    \coordinate (A89) at (0.5999999999998774,0.4000000000000034) {};
    \coordinate (A90) at (0.6000000000000596,0.4999999999999681) {};
    \coordinate (A91) at (0.600000000000242,0.5999999999998774) {};
    \coordinate (A92) at (0.6000000000004245,0.6999999999997865) {};
    \coordinate (A93) at (0.6000000000006066,0.799999999999793) {};
    \coordinate (A94) at (0.6000000000007891,0.8999999999998967) {};
    \coordinate (A95) at (0.6999999999994023,0.09999999999995685) {};
    \coordinate (A96) at (0.6999999999995304,0.1999999999999137) {};
    \coordinate (A97) at (0.6999999999996586,0.29999999999986) {};
    \coordinate (A98) at (0.6999999999997865,0.3999999999997957) {};
    \coordinate (A99) at (0.6999999999999146,0.4999999999997314) {};
    \coordinate (A100) at (0.7000000000000427,0.5999999999996949) {};
    \coordinate (A101) at (0.7000000000001708,0.6999999999996586) {};
    \coordinate (A102) at (0.700000000000299,0.7999999999997123) {};
    \coordinate (A103) at (0.7000000000004269,0.8999999999998562) {};
    \coordinate (A104) at (0.7999999999995507,0.09999999999989119) {};
    \coordinate (A105) at (0.7999999999996313,0.1999999999997824) {};
    \coordinate (A106) at (0.7999999999997122,0.2999999999996814) {};
    \coordinate (A107) at (0.799999999999793,0.3999999999995881) {};
    \coordinate (A108) at (0.7999999999998739,0.4999999999994949) {};
    \coordinate (A109) at (0.7999999999999546,0.5999999999995128) {};
    \coordinate (A110) at (0.8000000000000353,0.6999999999995303) {};
    \coordinate (A111) at (0.800000000000116,0.7999999999996312) {};
    \coordinate (A112) at (0.8000000000001967,0.8999999999998158) {};
    \coordinate (A113) at (0.8999999999997754,0.09999999999982558) {};
    \coordinate (A114) at (0.8999999999998158,0.1999999999996512) {};
    \coordinate (A115) at (0.8999999999998562,0.2999999999995028) {};
    \coordinate (A116) at (0.8999999999998964,0.3999999999993805) {};
    \coordinate (A117) at (0.8999999999999369,0.4999999999992581) {};
    \coordinate (A118) at (0.8999999999999775,0.5999999999993303) {};
    \coordinate (A119) at (0.9000000000000177,0.6999999999994022) {};
    \coordinate (A120) at (0.9000000000000579,0.7999999999995507) {};
    \coordinate (A121) at (0.9000000000000985,0.8999999999997754) {};
    \draw (A1) -- (A5) -- (A41) -- (A40) -- cycle;
    \draw (A40) -- (A41) -- (A42) -- (A39) -- cycle;
    \draw (A39) -- (A42) -- (A43) -- (A38) -- cycle;
    \draw (A38) -- (A43) -- (A44) -- (A37) -- cycle;
    \draw (A37) -- (A44) -- (A45) -- (A36) -- cycle;
    \draw (A36) -- (A45) -- (A46) -- (A35) -- cycle;
    \draw (A35) -- (A46) -- (A47) -- (A34) -- cycle;
    \draw (A34) -- (A47) -- (A48) -- (A33) -- cycle;
    \draw (A33) -- (A48) -- (A49) -- (A32) -- cycle;
    \draw (A32) -- (A49) -- (A31) -- (A4) -- cycle;
    \draw (A5) -- (A6) -- (A50) -- (A41) -- cycle;
    \draw (A41) -- (A50) -- (A51) -- (A42) -- cycle;
    \draw (A42) -- (A51) -- (A52) -- (A43) -- cycle;
    \draw (A43) -- (A52) -- (A53) -- (A44) -- cycle;
    \draw (A44) -- (A53) -- (A54) -- (A45) -- cycle;
    \draw (A45) -- (A54) -- (A55) -- (A46) -- cycle;
    \draw (A46) -- (A55) -- (A56) -- (A47) -- cycle;
    \draw (A47) -- (A56) -- (A57) -- (A48) -- cycle;
    \draw (A48) -- (A57) -- (A58) -- (A49) -- cycle;
    \draw (A49) -- (A58) -- (A30) -- (A31) -- cycle;
    \draw (A6) -- (A7) -- (A59) -- (A50) -- cycle;
    \draw (A50) -- (A59) -- (A60) -- (A51) -- cycle;
    \draw (A51) -- (A60) -- (A61) -- (A52) -- cycle;
    \draw (A52) -- (A61) -- (A62) -- (A53) -- cycle;
    \draw (A53) -- (A62) -- (A63) -- (A54) -- cycle;
    \draw (A54) -- (A63) -- (A64) -- (A55) -- cycle;
    \draw (A55) -- (A64) -- (A65) -- (A56) -- cycle;
    \draw (A56) -- (A65) -- (A66) -- (A57) -- cycle;
    \draw (A57) -- (A66) -- (A67) -- (A58) -- cycle;
    \draw (A58) -- (A67) -- (A29) -- (A30) -- cycle;
    \draw (A7) -- (A8) -- (A68) -- (A59) -- cycle;
    \draw (A59) -- (A68) -- (A69) -- (A60) -- cycle;
    \draw (A60) -- (A69) -- (A70) -- (A61) -- cycle;
    \draw (A61) -- (A70) -- (A71) -- (A62) -- cycle;
    \draw (A62) -- (A71) -- (A72) -- (A63) -- cycle;
    \draw (A63) -- (A72) -- (A73) -- (A64) -- cycle;
    \draw (A64) -- (A73) -- (A74) -- (A65) -- cycle;
    \draw (A65) -- (A74) -- (A75) -- (A66) -- cycle;
    \draw (A66) -- (A75) -- (A76) -- (A67) -- cycle;
    \draw (A67) -- (A76) -- (A28) -- (A29) -- cycle;
    \draw (A8) -- (A9) -- (A77) -- (A68) -- cycle;
    \draw (A68) -- (A77) -- (A78) -- (A69) -- cycle;
    \draw (A69) -- (A78) -- (A79) -- (A70) -- cycle;
    \draw (A70) -- (A79) -- (A80) -- (A71) -- cycle;
    \draw (A71) -- (A80) -- (A81) -- (A72) -- cycle;
    \draw (A72) -- (A81) -- (A82) -- (A73) -- cycle;
    \draw (A73) -- (A82) -- (A83) -- (A74) -- cycle;
    \draw (A74) -- (A83) -- (A84) -- (A75) -- cycle;
    \draw (A75) -- (A84) -- (A85) -- (A76) -- cycle;
    \draw (A76) -- (A85) -- (A27) -- (A28) -- cycle;
    \draw (A9) -- (A10) -- (A86) -- (A77) -- cycle;
    \draw (A77) -- (A86) -- (A87) -- (A78) -- cycle;
    \draw (A78) -- (A87) -- (A88) -- (A79) -- cycle;
    \draw (A79) -- (A88) -- (A89) -- (A80) -- cycle;
    \draw (A80) -- (A89) -- (A90) -- (A81) -- cycle;
    \draw (A81) -- (A90) -- (A91) -- (A82) -- cycle;
    \draw (A82) -- (A91) -- (A92) -- (A83) -- cycle;
    \draw (A83) -- (A92) -- (A93) -- (A84) -- cycle;
    \draw (A84) -- (A93) -- (A94) -- (A85) -- cycle;
    \draw (A85) -- (A94) -- (A26) -- (A27) -- cycle;
    \draw (A10) -- (A11) -- (A95) -- (A86) -- cycle;
    \draw (A86) -- (A95) -- (A96) -- (A87) -- cycle;
    \draw (A87) -- (A96) -- (A97) -- (A88) -- cycle;
    \draw (A88) -- (A97) -- (A98) -- (A89) -- cycle;
    \draw (A89) -- (A98) -- (A99) -- (A90) -- cycle;
    \draw (A90) -- (A99) -- (A100) -- (A91) -- cycle;
    \draw (A91) -- (A100) -- (A101) -- (A92) -- cycle;
    \draw (A92) -- (A101) -- (A102) -- (A93) -- cycle;
    \draw (A93) -- (A102) -- (A103) -- (A94) -- cycle;
    \draw (A94) -- (A103) -- (A25) -- (A26) -- cycle;
    \draw (A11) -- (A12) -- (A104) -- (A95) -- cycle;
    \draw (A95) -- (A104) -- (A105) -- (A96) -- cycle;
    \draw (A96) -- (A105) -- (A106) -- (A97) -- cycle;
    \draw (A97) -- (A106) -- (A107) -- (A98) -- cycle;
    \draw (A98) -- (A107) -- (A108) -- (A99) -- cycle;
    \draw (A99) -- (A108) -- (A109) -- (A100) -- cycle;
    \draw (A100) -- (A109) -- (A110) -- (A101) -- cycle;
    \draw (A101) -- (A110) -- (A111) -- (A102) -- cycle;
    \draw (A102) -- (A111) -- (A112) -- (A103) -- cycle;
    \draw (A103) -- (A112) -- (A24) -- (A25) -- cycle;
    \draw (A12) -- (A13) -- (A113) -- (A104) -- cycle;
    \draw (A104) -- (A113) -- (A114) -- (A105) -- cycle;
    \draw (A105) -- (A114) -- (A115) -- (A106) -- cycle;
    \draw (A106) -- (A115) -- (A116) -- (A107) -- cycle;
    \draw (A107) -- (A116) -- (A117) -- (A108) -- cycle;
    \draw (A108) -- (A117) -- (A118) -- (A109) -- cycle;
    \draw (A109) -- (A118) -- (A119) -- (A110) -- cycle;
    \draw (A110) -- (A119) -- (A120) -- (A111) -- cycle;
    \draw (A111) -- (A120) -- (A121) -- (A112) -- cycle;
    \draw (A112) -- (A121) -- (A23) -- (A24) -- cycle;
    \draw (A13) -- (A2) -- (A14) -- (A113) -- cycle;
    \draw (A113) -- (A14) -- (A15) -- (A114) -- cycle;
    \draw (A114) -- (A15) -- (A16) -- (A115) -- cycle;
    \draw (A115) -- (A16) -- (A17) -- (A116) -- cycle;
    \draw (A116) -- (A17) -- (A18) -- (A117) -- cycle;
    \draw (A117) -- (A18) -- (A19) -- (A118) -- cycle;
    \draw (A118) -- (A19) -- (A20) -- (A119) -- cycle;
    \draw (A119) -- (A20) -- (A21) -- (A120) -- cycle;
    \draw (A120) -- (A21) -- (A22) -- (A121) -- cycle;
    \draw (A121) -- (A22) -- (A3) -- (A23) -- cycle;
\end{tikzpicture}
      &
\begin{tikzpicture}[scale=5,line width=0.8,color=blue!70]
    \coordinate (A1) at (0,0) {};
    \coordinate (A2) at (1,0) {};
    \coordinate (A3) at (1,1) {};
    \coordinate (A4) at (0,1) {};
    \coordinate (A5) at (0.1111111111109082,0) {};
    \coordinate (A6) at (0.2222222222217143,0) {};
    \coordinate (A7) at (0.333333333332501,0) {};
    \coordinate (A8) at (0.4444444444432878,0) {};
    \coordinate (A9) at (0.5555555555543833,0) {};
    \coordinate (A10) at (0.6666666666657874,0) {};
    \coordinate (A11) at (0.7777777777771916,0) {};
    \coordinate (A12) at (0.8888888888885959,0) {};
    \coordinate (A13) at (1,0.1111111111109082) {};
    \coordinate (A14) at (1,0.2222222222217143) {};
    \coordinate (A15) at (1,0.333333333332501) {};
    \coordinate (A16) at (1,0.4444444444432878) {};
    \coordinate (A17) at (1,0.5555555555543833) {};
    \coordinate (A18) at (1,0.6666666666657874) {};
    \coordinate (A19) at (1,0.7777777777771916) {};
    \coordinate (A20) at (1,0.8888888888885959) {};
    \coordinate (A21) at (0.8888888888884262,1) {};
    \coordinate (A22) at (0.777777777777932,1) {};
    \coordinate (A23) at (0.6666666666675918,1) {};
    \coordinate (A24) at (0.5555555555572513,1) {};
    \coordinate (A25) at (0.4444444444462943,1) {};
    \coordinate (A26) at (0.3333333333347207,1) {};
    \coordinate (A27) at (0.2222222222231471,1) {};
    \coordinate (A28) at (0.1111111111115736,1) {};
    \coordinate (A29) at (0,0.8888888888884262) {};
    \coordinate (A30) at (0,0.777777777777932) {};
    \coordinate (A31) at (0,0.6666666666675918) {};
    \coordinate (A32) at (0,0.5555555555572513) {};
    \coordinate (A33) at (0,0.4444444444462943) {};
    \coordinate (A34) at (0,0.3333333333347207) {};
    \coordinate (A35) at (0,0.2222222222231471) {};
    \coordinate (A36) at (0,0.1111111111115736) {};
    \coordinate (A37) at (0.508755091245145,0.4950428077804493) {};
    \coordinate (A38) at (0.732373771066312,0.7289571218324546) {};
    \coordinate (A39) at (0.3105444075249406,0.7267024139067387) {};
    \coordinate (A40) at (0.6994777808277892,0.2946552060047777) {};
    \coordinate (A41) at (0.269776222337017,0.2719832445307022) {};
    \coordinate (A42) at (0.4797081563651218,0.7832667877340422) {};
    \coordinate (A43) at (0.7763247227200936,0.5504358380670301) {};
    \coordinate (A44) at (0.2089705854966178,0.5011647435726869) {};
    \coordinate (A45) at (0.5261830204236875,0.2169542291844375) {};
    \coordinate (A46) at (0.8132571162334484,0.810419150867757) {};
    \coordinate (A47) at (0.1986176091296227,0.805144684316691) {};
    \coordinate (A48) at (0.8049290095854345,0.1916564042176882) {};
    \coordinate (A49) at (0.1886974585467099,0.1894068751924823) {};
    \coordinate (A50) at (0.5638279784618917,0.6104574840757054) {};
    \coordinate (A51) at (0.3830707318893582,0.5486589610400849) {};
    \coordinate (A52) at (0.6403338422190834,0.4303405827428439) {};
    \coordinate (A53) at (0.4508368296362739,0.3820156619411093) {};
    \coordinate (A54) at (0.6573380487871894,0.818746379040202) {};
    \coordinate (A55) at (0.1508173952713392,0.6152315242515467) {};
    \coordinate (A56) at (0.8111982692351223,0.3586503574118033) {};
    \coordinate (A57) at (0.3442899824532811,0.1823415987908169) {};
    \coordinate (A58) at (0.8278420472549306,0.6301749459016944) {};
    \coordinate (A59) at (0.6185485743005283,0.1717166804739727) {};
    \coordinate (A60) at (0.3885577802280353,0.8299507205338855) {};
    \coordinate (A61) at (0.1798491529468413,0.344031407253954) {};
    \coordinate (A62) at (0.6944143156969094,0.5308553948124416) {};
    \coordinate (A63) at (0.4336905338706993,0.669400786980654) {};
    \coordinate (A64) at (0.5716195305329728,0.3276340072371606) {};
    \coordinate (A65) at (0.356363657458026,0.440216967908619) {};
    \coordinate (A66) at (0.4869955560379161,0.8888816442858173) {};
    \coordinate (A67) at (0.1097131868598185,0.5151115616358697) {};
    \coordinate (A68) at (0.8855720301149671,0.5238978108268397) {};
    \coordinate (A69) at (0.5170542472459531,0.1110820683650001) {};
    \coordinate (A70) at (0.7021522788500931,0.9041308489596316) {};
    \coordinate (A71) at (0.09046621310319609,0.6903669689245789) {};
    \coordinate (A72) at (0.9011540430689917,0.3023699776414385) {};
    \coordinate (A73) at (0.2987257870097897,0.09653650044140948) {};
    \coordinate (A74) at (0.9103368406655487,0.6961163199300626) {};
    \coordinate (A75) at (0.6932709358133692,0.09042306018662782) {};
    \coordinate (A76) at (0.3099931245664933,0.9092609093726997) {};
    \coordinate (A77) at (0.09560385293403414,0.2986425490706365) {};
    \coordinate (A78) at (0.624288143392099,0.744753045634639) {};
    \coordinate (A79) at (0.2559113466918461,0.5974976446361265) {};
    \coordinate (A80) at (0.7374758983363169,0.4164425004604435) {};
    \coordinate (A81) at (0.3799149636125554,0.2557441054536824) {};
    \coordinate (A82) at (0.6020532458771199,0.5150794427177929) {};
    \coordinate (A83) at (0.5405040423678343,0.4084490286608732) {};
    \coordinate (A84) at (0.402010544052951,0.7517206800889338) {};
    \coordinate (A85) at (0.2505849569876462,0.380309524300411) {};
    \coordinate (A86) at (0.8989449492443384,0.8983272264902707) {};
    \coordinate (A87) at (0.1032550746863064,0.8972359377764949) {};
    \coordinate (A88) at (0.8973840769533368,0.1018723159699704) {};
    \coordinate (A89) at (0.1015681509432729,0.1017413752607486) {};
    \coordinate (A90) at (0.7231171550745163,0.8154565529097741) {};
    \coordinate (A91) at (0.1996681112632412,0.6869994020498515) {};
    \coordinate (A92) at (0.8066264265949994,0.2893110265884782) {};
    \coordinate (A93) at (0.2787443276054614,0.1851238039008127) {};
    \coordinate (A94) at (0.8215049067433008,0.7162134523762784) {};
    \coordinate (A95) at (0.7052186117696142,0.1839799032491825) {};
    \coordinate (A96) at (0.3009052246028832,0.8171025384836558) {};
    \coordinate (A97) at (0.1834766425036917,0.2793281510477015) {};
    \coordinate (A98) at (0.4719899328956793,0.5816205877393462) {};
    \coordinate (A99) at (0.4209438829989718,0.4654815181903776) {};
    \coordinate (A100) at (0.7465603384900987,0.6371279821672595) {};
    \coordinate (A101) at (0.6026388021093365,0.2508064444767395) {};
    \coordinate (A102) at (0.07839223773779971,0.6064212698935145) {};
    \coordinate (A103) at (0.6191724684299051,0.8806621731940115) {};
    \coordinate (A104) at (0.4073431945799385,0.9096233525909559) {};
    \coordinate (A105) at (0.1183501617446496,0.3817168450083969) {};
    \coordinate (A106) at (0.9087473111470848,0.6008991977992667) {};
    \coordinate (A107) at (0.5964055090184236,0.09058634524720877) {};
    \coordinate (A108) at (0.3825428816083688,0.1199304689852927) {};
    \coordinate (A109) at (0.8765598441390103,0.3916783593263462) {};
    \coordinate (A110) at (0.7979748879491598,0.9038913527177752) {};
    \coordinate (A111) at (0.09844665708367478,0.7915872999177173) {};
    \coordinate (A112) at (0.9057888941495067,0.7957476743316827) {};
    \coordinate (A113) at (0.2084255319833656,0.9032773469741769) {};
    \coordinate (A114) at (0.7935637384934968,0.09559769033230051) {};
    \coordinate (A115) at (0.09585633786409284,0.2030400058448737) {};
    \coordinate (A116) at (0.901379576183599,0.2047769324380561) {};
    \coordinate (A117) at (0.2029759074626876,0.09646897752462893) {};
    \coordinate (A118) at (0.3438327562628685,0.6342273295115328) {};
    \coordinate (A119) at (0.3604692304651219,0.3691162762640802) {};
    \coordinate (A120) at (0.4845860660808211,0.296405454512382) {};
    \coordinate (A121) at (0.4512960720688876,0.8338136394963934) {};
    \coordinate (A122) at (0.168222176866359,0.4501164853802866) {};
    \coordinate (A123) at (0.8295519184351208,0.5682935141225732) {};
    \coordinate (A124) at (0.5542143142813223,0.1668478374991561) {};
    \coordinate (A125) at (0.550264450411045,0.8299770938899462) {};
    \coordinate (A126) at (0.8245946532093016,0.4714086410040856) {};
    \coordinate (A127) at (0.1817634738759922,0.55871412964295) {};
    \coordinate (A128) at (0.4534299435069698,0.1704395178076721) {};
    \coordinate (A129) at (0.6574231420167046,0.6292324741062088) {};
    \coordinate (A130) at (0.5233157352566016,0.7013669049646052) {};
    \coordinate (A131) at (0.2918432996570445,0.494547153656359) {};
    \coordinate (A132) at (0.6343884172372523,0.353051695280704) {};
    \draw (A56) -- (A80) -- (A40) -- (A92) -- cycle;
    \draw (A57) -- (A81) -- (A41) -- (A93) -- cycle;
    \draw (A54) -- (A78) -- (A38) -- (A90) -- cycle;
    \draw (A53) -- (A83) -- (A37) -- (A99) -- cycle;
    \draw (A52) -- (A82) -- (A37) -- (A83) -- cycle;
    \draw (A63) -- (A130) -- (A42) -- (A84) -- cycle;
    \draw (A62) -- (A129) -- (A50) -- (A82) -- cycle;
    \draw (A64) -- (A132) -- (A52) -- (A83) -- cycle;
    \draw (A13) -- (A88) -- (A12) -- (A2) -- cycle;
    \draw (A5) -- (A89) -- (A36) -- (A1) -- cycle;
    \draw (A21) -- (A86) -- (A20) -- (A3) -- cycle;
    \draw (A29) -- (A87) -- (A28) -- (A4) -- cycle;
    \draw (A98) -- (A63) -- (A118) -- (A51) -- cycle;
    \draw (A101) -- (A64) -- (A120) -- (A45) -- cycle;
    \draw (A99) -- (A65) -- (A119) -- (A53) -- cycle;
    \draw (A51) -- (A99) -- (A37) -- (A98) -- cycle;
    \draw (A50) -- (A98) -- (A37) -- (A82) -- cycle;
    \draw (A131) -- (A44) -- (A122) -- (A85) -- cycle;
    \draw (A33) -- (A67) -- (A102) -- (A32) -- cycle;
    \draw (A24) -- (A25) -- (A104) -- (A66) -- cycle;
    \draw (A16) -- (A17) -- (A106) -- (A68) -- cycle;
    \draw (A8) -- (A9) -- (A107) -- (A69) -- cycle;
    \draw (A103) -- (A54) -- (A90) -- (A70) -- cycle;
    \draw (A102) -- (A55) -- (A91) -- (A71) -- cycle;
    \draw (A109) -- (A56) -- (A92) -- (A72) -- cycle;
    \draw (A108) -- (A57) -- (A93) -- (A73) -- cycle;
    \draw (A39) -- (A84) -- (A60) -- (A96) -- cycle;
    \draw (A41) -- (A85) -- (A61) -- (A97) -- cycle;
    \draw (A24) -- (A66) -- (A125) -- (A103) -- cycle;
    \draw (A16) -- (A68) -- (A126) -- (A109) -- cycle;
    \draw (A8) -- (A69) -- (A128) -- (A108) -- cycle;
    \draw (A33) -- (A105) -- (A122) -- (A67) -- cycle;
    \draw (A91) -- (A79) -- (A118) -- (A39) -- cycle;
    \draw (A119) -- (A81) -- (A120) -- (A53) -- cycle;
    \draw (A91) -- (A55) -- (A127) -- (A79) -- cycle;
    \draw (A106) -- (A74) -- (A94) -- (A58) -- cycle;
    \draw (A107) -- (A75) -- (A95) -- (A59) -- cycle;
    \draw (A104) -- (A76) -- (A96) -- (A60) -- cycle;
    \draw (A105) -- (A77) -- (A97) -- (A61) -- cycle;
    \draw (A125) -- (A42) -- (A130) -- (A78) -- cycle;
    \draw (A127) -- (A44) -- (A131) -- (A79) -- cycle;
    \draw (A62) -- (A80) -- (A126) -- (A43) -- cycle;
    \draw (A120) -- (A81) -- (A128) -- (A45) -- cycle;
    \draw (A50) -- (A130) -- (A63) -- (A98) -- cycle;
    \draw (A51) -- (A131) -- (A65) -- (A99) -- cycle;
    \draw (A118) -- (A79) -- (A131) -- (A51) -- cycle;
    \draw (A40) -- (A80) -- (A52) -- (A132) -- cycle;
    \draw (A125) -- (A66) -- (A121) -- (A42) -- cycle;
    \draw (A127) -- (A67) -- (A122) -- (A44) -- cycle;
    \draw (A126) -- (A68) -- (A123) -- (A43) -- cycle;
    \draw (A128) -- (A69) -- (A124) -- (A45) -- cycle;
    \draw (A90) -- (A46) -- (A110) -- (A70) -- cycle;
    \draw (A91) -- (A47) -- (A111) -- (A71) -- cycle;
    \draw (A92) -- (A48) -- (A116) -- (A72) -- cycle;
    \draw (A93) -- (A49) -- (A117) -- (A73) -- cycle;
    \draw (A94) -- (A74) -- (A112) -- (A46) -- cycle;
    \draw (A96) -- (A76) -- (A113) -- (A47) -- cycle;
    \draw (A95) -- (A75) -- (A114) -- (A48) -- cycle;
    \draw (A97) -- (A77) -- (A115) -- (A49) -- cycle;
    \draw (A130) -- (A50) -- (A129) -- (A78) -- cycle;
    \draw (A78) -- (A129) -- (A100) -- (A38) -- cycle;
    \draw (A43) -- (A100) -- (A129) -- (A62) -- cycle;
    \draw (A40) -- (A132) -- (A64) -- (A101) -- cycle;
    \draw (A42) -- (A121) -- (A60) -- (A84) -- cycle;
    \draw (A76) -- (A104) -- (A25) -- (A26) -- cycle;
    \draw (A77) -- (A105) -- (A33) -- (A34) -- cycle;
    \draw (A74) -- (A106) -- (A17) -- (A18) -- cycle;
    \draw (A75) -- (A107) -- (A9) -- (A10) -- cycle;
    \draw (A71) -- (A31) -- (A32) -- (A102) -- cycle;
    \draw (A70) -- (A23) -- (A24) -- (A103) -- cycle;
    \draw (A73) -- (A7) -- (A8) -- (A108) -- cycle;
    \draw (A72) -- (A15) -- (A16) -- (A109) -- cycle;
    \draw (A43) -- (A123) -- (A58) -- (A100) -- cycle;
    \draw (A45) -- (A124) -- (A59) -- (A101) -- cycle;
    \draw (A80) -- (A62) -- (A82) -- (A52) -- cycle;
    \draw (A64) -- (A83) -- (A53) -- (A120) -- cycle;
    \draw (A38) -- (A100) -- (A58) -- (A94) -- cycle;
    \draw (A40) -- (A101) -- (A59) -- (A95) -- cycle;
    \draw (A23) -- (A70) -- (A110) -- (A22) -- cycle;
    \draw (A31) -- (A71) -- (A111) -- (A30) -- cycle;
    \draw (A15) -- (A72) -- (A116) -- (A14) -- cycle;
    \draw (A7) -- (A73) -- (A117) -- (A6) -- cycle;
    \draw (A18) -- (A19) -- (A112) -- (A74) -- cycle;
    \draw (A10) -- (A11) -- (A114) -- (A75) -- cycle;
    \draw (A26) -- (A27) -- (A113) -- (A76) -- cycle;
    \draw (A34) -- (A35) -- (A115) -- (A77) -- cycle;
    \draw (A39) -- (A118) -- (A63) -- (A84) -- cycle;
    \draw (A131) -- (A85) -- (A119) -- (A65) -- cycle;
    \draw (A81) -- (A119) -- (A85) -- (A41) -- cycle;
    \draw (A46) -- (A90) -- (A38) -- (A94) -- cycle;
    \draw (A47) -- (A91) -- (A39) -- (A96) -- cycle;
    \draw (A48) -- (A92) -- (A40) -- (A95) -- cycle;
    \draw (A49) -- (A93) -- (A41) -- (A97) -- cycle;
    \draw (A22) -- (A110) -- (A86) -- (A21) -- cycle;
    \draw (A14) -- (A116) -- (A88) -- (A13) -- cycle;
    \draw (A30) -- (A111) -- (A87) -- (A29) -- cycle;
    \draw (A6) -- (A117) -- (A89) -- (A5) -- cycle;
    \draw (A19) -- (A20) -- (A86) -- (A112) -- cycle;
    \draw (A11) -- (A12) -- (A88) -- (A114) -- cycle;
    \draw (A27) -- (A28) -- (A87) -- (A113) -- cycle;
    \draw (A35) -- (A36) -- (A89) -- (A115) -- cycle;
    \draw (A110) -- (A46) -- (A112) -- (A86) -- cycle;
    \draw (A111) -- (A47) -- (A113) -- (A87) -- cycle;
    \draw (A116) -- (A48) -- (A114) -- (A88) -- cycle;
    \draw (A117) -- (A49) -- (A115) -- (A89) -- cycle;
    \draw (A78) -- (A54) -- (A103) -- (A125) -- cycle;
    \draw (A67) -- (A127) -- (A55) -- (A102) -- cycle;
    \draw (A80) -- (A56) -- (A109) -- (A126) -- cycle;
    \draw (A81) -- (A57) -- (A108) -- (A128) -- cycle;
    \draw (A68) -- (A106) -- (A58) -- (A123) -- cycle;
    \draw (A69) -- (A107) -- (A59) -- (A124) -- cycle;
    \draw (A66) -- (A104) -- (A60) -- (A121) -- cycle;
    \draw (A85) -- (A122) -- (A105) -- (A61) -- cycle;
\end{tikzpicture}
      \\[0.4cm]
    \end{tabular}
\begin{tikzpicture}[scale=5,line width=0.8,color=blue!70]
    \coordinate (A1) at (0,0) {};
    \coordinate (A2) at (1,0) {};
    \coordinate (A3) at (1,1) {};
    \coordinate (A4) at (0,1) {};
    \coordinate (A5) at (0.1111111111108444,0) {};
    \coordinate (A6) at (0.2222222222216888,0) {};
    \coordinate (A7) at (0.3333333333326072,0) {};
    \coordinate (A8) at (0.4444444444435502,0) {};
    \coordinate (A9) at (0.5555555555546474,0) {};
    \coordinate (A10) at (0.6666666666658989,0) {};
    \coordinate (A11) at (0.7777777777771888,0) {};
    \coordinate (A12) at (0.8888888888885944,0) {};
    \coordinate (A13) at (1,0.8888888888890432) {};
    \coordinate (A14) at (1,0.7777777777780862) {};
    \coordinate (A15) at (1,0.6666666666673606) {};
    \coordinate (A16) at (1,0.5555555555567121) {};
    \coordinate (A17) at (1,0.4444444444457551) {};
    \coordinate (A18) at (1,0.3333333333344898) {};
    \coordinate (A19) at (1,0.2222222222231475) {};
    \coordinate (A20) at (1,0.1111111111115738) {};
    \coordinate (A21) at (0.8888888888890432,1) {};
    \coordinate (A22) at (0.7777777777780862,1) {};
    \coordinate (A23) at (0.6666666666673606,1) {};
    \coordinate (A24) at (0.5555555555567121,1) {};
    \coordinate (A25) at (0.4444444444457551,1) {};
    \coordinate (A26) at (0.3333333333344898,1) {};
    \coordinate (A27) at (0.2222222222231475,1) {};
    \coordinate (A28) at (0.1111111111115738,1) {};
    \coordinate (A29) at (0,0.8888888888890432) {};
    \coordinate (A30) at (0,0.7777777777780862) {};
    \coordinate (A31) at (0,0.6666666666673606) {};
    \coordinate (A32) at (0,0.5555555555567121) {};
    \coordinate (A33) at (0,0.4444444444457551) {};
    \coordinate (A34) at (0,0.3333333333344898) {};
    \coordinate (A35) at (0,0.2222222222231475) {};
    \coordinate (A36) at (0,0.1111111111115738) {};
    \coordinate (A37) at (0.4896968103158592,0.3727320363267672) {};
    \coordinate (A38) at (0.7814174888284939,0.2923707184890695) {};
    \coordinate (A39) at (0.2817473973531531,0.7684662513060894) {};
    \coordinate (A40) at (0.2048629024577838,0.2373795989370628) {};
    \coordinate (A41) at (0.6799809136530539,0.7071724182029255) {};
    \coordinate (A42) at (0.08485877000902493,0.8400561445870619) {};
    \coordinate (A43) at (0.9086481653156846,0.8446829174292438) {};
    \coordinate (A44) at (0.7398203597901577,0.08089655423535448) {};
    \coordinate (A45) at (0.8337712705119823,0.9205681180682782) {};
    \coordinate (A46) at (0.2285666179133003,0.4164543097155909) {};
    \coordinate (A47) at (0.7848936673765179,0.5342384230085764) {};
    \coordinate (A48) at (0.4817607219246043,0.7369650480829728) {};
    \coordinate (A49) at (0.5074735280923236,0.1041295748403279) {};
    \coordinate (A50) at (0.3741446157024846,0.5690956118196183) {};
    \coordinate (A51) at (0.1521525594112199,0.589851255970592) {};
    \coordinate (A52) at (0.5799772749379728,0.5502027953588847) {};
    \coordinate (A53) at (0.5393141938802573,0.8415676688785179) {};
    \coordinate (A54) at (0.1370621030263838,0.3418648614082736) {};
    \coordinate (A55) at (0.6744828549367706,0.3496675109439035) {};
    \coordinate (A56) at (0.3087554387305871,0.2136725965208915) {};
    \coordinate (A57) at (0.3833936452112924,0.9027629765868167) {};
    \coordinate (A58) at (0.8852153119288965,0.7404142479088892) {};
    \coordinate (A59) at (0.09927008393052654,0.7200404949914917) {};
    \coordinate (A60) at (0.2695086296796442,0.8823071535339428) {};
    \coordinate (A61) at (0.9227565994328794,0.358564844115848) {};
    \coordinate (A62) at (0.093632191895219,0.2185984559367971) {};
    \coordinate (A63) at (0.7180011977241172,0.9072137582536914) {};
    \coordinate (A64) at (0.1544112388784324,0.9127514033035978) {};
    \coordinate (A65) at (0.1733565179802812,0.1121912365134053) {};
    \coordinate (A66) at (0.2846065314033214,0.1049190526471203) {};
    \coordinate (A67) at (0.8024430396566005,0.825956606531551) {};
    \coordinate (A68) at (0.181684938271093,0.8090562891790953) {};
    \coordinate (A69) at (0.7952852178770345,0.179069138838037) {};
    \coordinate (A70) at (0.1417769572345392,0.4577194466281078) {};
    \coordinate (A71) at (0.4974776181496037,0.2331104683922483) {};
    \coordinate (A72) at (0.3535059755636742,0.4004523633060279) {};
    \coordinate (A73) at (0.7304927539894233,0.6230261295803099) {};
    \coordinate (A74) at (0.5919083287089618,0.2905435063162533) {};
    \coordinate (A75) at (0.2609801644080395,0.5888212004589696) {};
    \coordinate (A76) at (0.4242751648647285,0.4764173200645572) {};
    \coordinate (A77) at (0.9106048888214398,0.5443440030621306) {};
    \coordinate (A78) at (0.6307026993851398,0.6298811783313787) {};
    \coordinate (A79) at (0.8479692322957689,0.6379934959593609) {};
    \coordinate (A80) at (0.5298935736300789,0.6433168547376779) {};
    \coordinate (A81) at (0.6786304494820828,0.5435067832174573) {};
    \coordinate (A82) at (0.3818053876132408,0.2780938254121307) {};
    \coordinate (A83) at (0.4792113007111641,0.5593110921598804) {};
    \coordinate (A84) at (0.3067647436583859,0.4914247064419113) {};
    \coordinate (A85) at (0.07128234270066712,0.509920105516576) {};
    \coordinate (A86) at (0.3813321221977031,0.7352879725144476) {};
    \coordinate (A87) at (0.2180482085250969,0.5088541838430343) {};
    \coordinate (A88) at (0.7714338106896199,0.7259974944446341) {};
    \coordinate (A89) at (0.5944572676282027,0.7568190139775648) {};
    \coordinate (A90) at (0.5964793155985781,0.1812766877710464) {};
    \coordinate (A91) at (0.4289902577593507,0.6497121935235235) {};
    \coordinate (A92) at (0.3212513445327513,0.6736832528291727) {};
    \coordinate (A93) at (0.3925915166967531,0.1345146837536958) {};
    \coordinate (A94) at (0.5290770919869177,0.4697570002262929) {};
    \coordinate (A95) at (0.6978322146081674,0.8023578180535357) {};
    \coordinate (A96) at (0.6249364516023197,0.4626000317441766) {};
    \coordinate (A97) at (0.3501249778011909,0.8175352079927128) {};
    \coordinate (A98) at (0.241505111217689,0.1688859925170703) {};
    \coordinate (A99) at (0.06454099720848271,0.6084068157405464) {};
    \coordinate (A100) at (0.8648787287589447,0.4486941039438495) {};
    \coordinate (A101) at (0.7467900430843953,0.4295495651585958) {};
    \coordinate (A102) at (0.2690930708841617,0.3146529258833295) {};
    \coordinate (A103) at (0.890310235991261,0.2713180802183151) {};
    \coordinate (A104) at (0.4386264911861728,0.8328538362076228) {};
    \coordinate (A105) at (0.5780813387348107,0.3941789737888583) {};
    \coordinate (A106) at (0.07002428059231805,0.4174564382666405) {};
    \coordinate (A107) at (0.6856537281867894,0.1490479555322426) {};
    \coordinate (A108) at (0.4937264422913281,0.918509289452297) {};
    \coordinate (A109) at (0.2161810813177972,0.6916531241225684) {};
    \coordinate (A110) at (0.6301258394617045,0.8483659028408269) {};
    \coordinate (A111) at (0.8418303059421968,0.3498388510222085) {};
    \coordinate (A112) at (0.8703791834558856,0.07641489395576725) {};
    \coordinate (A113) at (0.07007622399980623,0.9283392873559405) {};
    \coordinate (A114) at (0.9262616649433421,0.9308279848773131) {};
    \coordinate (A115) at (0.6252748589813991,0.08589179539649523) {};
    \coordinate (A116) at (0.6874588287591292,0.2396863830933766) {};
    \coordinate (A117) at (0.6005649825969133,0.9192761032375555) {};
    \coordinate (A118) at (0.9111949274648362,0.1720270892693681) {};
    \coordinate (A119) at (0.06301663683105743,0.1106871709641539) {};
    \draw (A21) -- (A22) -- (A45) -- cycle;
    \draw (A43) -- (A14) -- (A13) -- cycle;
    \draw (A29) -- (A30) -- (A42) -- cycle;
    \draw (A25) -- (A26) -- (A57) -- cycle;
    \draw (A64) -- (A60) -- (A27) -- cycle;
    \draw (A40) -- (A62) -- (A65) -- cycle;
    \draw (A66) -- (A6) -- (A7) -- cycle;
    \draw (A45) -- (A63) -- (A67) -- cycle;
    \draw (A42) -- (A59) -- (A68) -- cycle;
    \draw (A39) -- (A60) -- (A68) -- cycle;
    \draw (A64) -- (A27) -- (A28) -- cycle;
    \draw (A64) -- (A42) -- (A68) -- cycle;
    \draw (A64) -- (A68) -- (A60) -- cycle;
    \draw (A5) -- (A6) -- (A65) -- cycle;
    \draw (A65) -- (A6) -- (A66) -- cycle;
    \draw (A44) -- (A10) -- (A11) -- cycle;
    \draw (A18) -- (A17) -- (A61) -- cycle;
    \draw (A8) -- (A9) -- (A49) -- cycle;
    \draw (A34) -- (A35) -- (A62) -- cycle;
    \draw (A62) -- (A54) -- (A34) -- cycle;
    \draw (A40) -- (A54) -- (A62) -- cycle;
    \draw (A26) -- (A27) -- (A60) -- cycle;
    \draw (A60) -- (A57) -- (A26) -- cycle;
    \draw (A58) -- (A15) -- (A14) -- cycle;
    \draw (A43) -- (A58) -- (A14) -- cycle;
    \draw (A23) -- (A63) -- (A22) -- cycle;
    \draw (A22) -- (A63) -- (A45) -- cycle;
    \draw (A67) -- (A58) -- (A43) -- cycle;
    \draw (A43) -- (A45) -- (A67) -- cycle;
    \draw (A30) -- (A31) -- (A59) -- cycle;
    \draw (A30) -- (A59) -- (A42) -- cycle;
    \draw (A54) -- (A46) -- (A70) -- cycle;
    \draw (A37) -- (A71) -- (A74) -- cycle;
    \draw (A76) -- (A72) -- (A37) -- cycle;
    \draw (A78) -- (A73) -- (A41) -- cycle;
    \draw (A79) -- (A73) -- (A47) -- cycle;
    \draw (A47) -- (A77) -- (A79) -- cycle;
    \draw (A52) -- (A78) -- (A80) -- cycle;
    \draw (A47) -- (A73) -- (A81) -- cycle;
    \draw (A82) -- (A71) -- (A37) -- cycle;
    \draw (A37) -- (A72) -- (A82) -- cycle;
    \draw (A50) -- (A76) -- (A83) -- cycle;
    \draw (A52) -- (A80) -- (A83) -- cycle;
    \draw (A46) -- (A72) -- (A84) -- cycle;
    \draw (A50) -- (A75) -- (A84) -- cycle;
    \draw (A85) -- (A32) -- (A33) -- cycle;
    \draw (A85) -- (A70) -- (A51) -- cycle;
    \draw (A87) -- (A70) -- (A46) -- cycle;
    \draw (A87) -- (A75) -- (A51) -- cycle;
    \draw (A51) -- (A70) -- (A87) -- cycle;
    \draw (A58) -- (A67) -- (A88) -- cycle;
    \draw (A41) -- (A73) -- (A88) -- cycle;
    \draw (A89) -- (A78) -- (A41) -- cycle;
    \draw (A90) -- (A71) -- (A49) -- cycle;
    \draw (A91) -- (A80) -- (A48) -- cycle;
    \draw (A48) -- (A86) -- (A91) -- cycle;
    \draw (A92) -- (A75) -- (A50) -- cycle;
    \draw (A92) -- (A86) -- (A39) -- cycle;
    \draw (A93) -- (A8) -- (A49) -- cycle;
    \draw (A94) -- (A76) -- (A37) -- cycle;
    \draw (A95) -- (A67) -- (A63) -- cycle;
    \draw (A96) -- (A81) -- (A52) -- cycle;
    \draw (A52) -- (A94) -- (A96) -- cycle;
    \draw (A97) -- (A60) -- (A39) -- cycle;
    \draw (A57) -- (A60) -- (A97) -- cycle;
    \draw (A39) -- (A86) -- (A97) -- cycle;
    \draw (A40) -- (A65) -- (A98) -- cycle;
    \draw (A66) -- (A56) -- (A98) -- cycle;
    \draw (A98) -- (A65) -- (A66) -- cycle;
    \draw (A98) -- (A56) -- (A40) -- cycle;
    \draw (A31) -- (A32) -- (A99) -- cycle;
    \draw (A51) -- (A59) -- (A99) -- cycle;
    \draw (A99) -- (A59) -- (A31) -- cycle;
    \draw (A100) -- (A77) -- (A47) -- cycle;
    \draw (A47) -- (A81) -- (A101) -- cycle;
    \draw (A102) -- (A54) -- (A40) -- cycle;
    \draw (A46) -- (A54) -- (A102) -- cycle;
    \draw (A48) -- (A53) -- (A104) -- cycle;
    \draw (A104) -- (A86) -- (A48) -- cycle;
    \draw (A37) -- (A74) -- (A105) -- cycle;
    \draw (A105) -- (A74) -- (A55) -- cycle;
    \draw (A105) -- (A94) -- (A37) -- cycle;
    \draw (A33) -- (A34) -- (A106) -- cycle;
    \draw (A54) -- (A70) -- (A106) -- cycle;
    \draw (A106) -- (A34) -- (A54) -- cycle;
    \draw (A108) -- (A24) -- (A25) -- cycle;
    \draw (A25) -- (A57) -- (A108) -- cycle;
    \draw (A39) -- (A68) -- (A109) -- cycle;
    \draw (A109) -- (A68) -- (A59) -- cycle;
    \draw (A12) -- (A2) -- (A112) -- cycle;
    \draw (A112) -- (A2) -- (A20) -- cycle;
    \draw (A113) -- (A4) -- (A29) -- cycle;
    \draw (A113) -- (A64) -- (A28) -- cycle;
    \draw (A28) -- (A4) -- (A113) -- cycle;
    \draw (A13) -- (A3) -- (A114) -- cycle;
    \draw (A114) -- (A3) -- (A21) -- cycle;
    \draw (A99) -- (A32) -- (A85) -- cycle;
    \draw (A51) -- (A99) -- (A85) -- cycle;
    \draw (A95) -- (A41) -- (A88) -- cycle;
    \draw (A67) -- (A95) -- (A88) -- cycle;
    \draw (A104) -- (A53) -- (A108) -- cycle;
    \draw (A104) -- (A108) -- (A57) -- cycle;
    \draw (A113) -- (A29) -- (A42) -- cycle;
    \draw (A64) -- (A113) -- (A42) -- cycle;
    \draw (A85) -- (A33) -- (A106) -- cycle;
    \draw (A85) -- (A106) -- (A70) -- cycle;
    \draw (A71) -- (A90) -- (A74) -- cycle;
    \draw (A76) -- (A50) -- (A84) -- cycle;
    \draw (A76) -- (A84) -- (A72) -- cycle;
    \draw (A78) -- (A52) -- (A81) -- cycle;
    \draw (A78) -- (A81) -- (A73) -- cycle;
    \draw (A79) -- (A58) -- (A88) -- cycle;
    \draw (A79) -- (A88) -- (A73) -- cycle;
    \draw (A87) -- (A46) -- (A84) -- cycle;
    \draw (A75) -- (A87) -- (A84) -- cycle;
    \draw (A94) -- (A52) -- (A83) -- cycle;
    \draw (A76) -- (A94) -- (A83) -- cycle;
    \draw (A89) -- (A48) -- (A80) -- cycle;
    \draw (A78) -- (A89) -- (A80) -- cycle;
    \draw (A91) -- (A50) -- (A83) -- cycle;
    \draw (A80) -- (A91) -- (A83) -- cycle;
    \draw (A96) -- (A55) -- (A101) -- cycle;
    \draw (A96) -- (A101) -- (A81) -- cycle;
    \draw (A92) -- (A50) -- (A91) -- cycle;
    \draw (A86) -- (A92) -- (A91) -- cycle;
    \draw (A104) -- (A57) -- (A97) -- cycle;
    \draw (A86) -- (A104) -- (A97) -- cycle;
    \draw (A111) -- (A61) -- (A100) -- cycle;
    \draw (A95) -- (A63) -- (A110) -- cycle;
    \draw (A105) -- (A55) -- (A96) -- cycle;
    \draw (A94) -- (A105) -- (A96) -- cycle;
    \draw (A114) -- (A43) -- (A13) -- cycle;
    \draw (A114) -- (A21) -- (A45) -- cycle;
    \draw (A45) -- (A43) -- (A114) -- cycle;
    \draw (A107) -- (A44) -- (A69) -- cycle;
    \draw (A16) -- (A77) -- (A17) -- cycle;
    \draw (A100) -- (A61) -- (A17) -- cycle;
    \draw (A77) -- (A100) -- (A17) -- cycle;
    \draw (A103) -- (A19) -- (A18) -- cycle;
    \draw (A69) -- (A103) -- (A38) -- cycle;
    \draw (A38) -- (A103) -- (A111) -- cycle;
    \draw (A18) -- (A61) -- (A103) -- cycle;
    \draw (A111) -- (A103) -- (A61) -- cycle;
    \draw (A102) -- (A72) -- (A46) -- cycle;
    \draw (A102) -- (A40) -- (A56) -- cycle;
    \draw (A102) -- (A56) -- (A82) -- cycle;
    \draw (A72) -- (A102) -- (A82) -- cycle;
    \draw (A8) -- (A93) -- (A7) -- cycle;
    \draw (A7) -- (A93) -- (A66) -- cycle;
    \draw (A49) -- (A71) -- (A93) -- cycle;
    \draw (A56) -- (A66) -- (A93) -- cycle;
    \draw (A82) -- (A56) -- (A93) -- cycle;
    \draw (A82) -- (A93) -- (A71) -- cycle;
    \draw (A101) -- (A55) -- (A38) -- cycle;
    \draw (A111) -- (A101) -- (A38) -- cycle;
    \draw (A47) -- (A101) -- (A100) -- cycle;
    \draw (A101) -- (A111) -- (A100) -- cycle;
    \draw (A92) -- (A39) -- (A109) -- cycle;
    \draw (A51) -- (A75) -- (A109) -- cycle;
    \draw (A109) -- (A59) -- (A51) -- cycle;
    \draw (A75) -- (A92) -- (A109) -- cycle;
    \draw (A15) -- (A77) -- (A16) -- cycle;
    \draw (A15) -- (A58) -- (A79) -- cycle;
    \draw (A77) -- (A15) -- (A79) -- cycle;
    \draw (A95) -- (A89) -- (A41) -- cycle;
    \draw (A53) -- (A48) -- (A89) -- cycle;
    \draw (A53) -- (A89) -- (A110) -- cycle;
    \draw (A95) -- (A110) -- (A89) -- cycle;
    \draw (A115) -- (A49) -- (A9) -- cycle;
    \draw (A90) -- (A49) -- (A115) -- cycle;
    \draw (A116) -- (A69) -- (A38) -- cycle;
    \draw (A107) -- (A69) -- (A116) -- cycle;
    \draw (A117) -- (A23) -- (A24) -- cycle;
    \draw (A117) -- (A53) -- (A110) -- cycle;
    \draw (A117) -- (A24) -- (A108) -- cycle;
    \draw (A53) -- (A117) -- (A108) -- cycle;
    \draw (A63) -- (A23) -- (A117) -- cycle;
    \draw (A63) -- (A117) -- (A110) -- cycle;
    \draw (A115) -- (A9) -- (A10) -- cycle;
    \draw (A115) -- (A10) -- (A44) -- cycle;
    \draw (A115) -- (A107) -- (A90) -- cycle;
    \draw (A107) -- (A115) -- (A44) -- cycle;
    \draw (A55) -- (A74) -- (A116) -- cycle;
    \draw (A116) -- (A38) -- (A55) -- cycle;
    \draw (A90) -- (A116) -- (A74) -- cycle;
    \draw (A107) -- (A116) -- (A90) -- cycle;
    \draw (A69) -- (A44) -- (A112) -- cycle;
    \draw (A11) -- (A112) -- (A44) -- cycle;
    \draw (A11) -- (A12) -- (A112) -- cycle;
    \draw (A103) -- (A69) -- (A118) -- cycle;
    \draw (A118) -- (A69) -- (A112) -- cycle;
    \draw (A20) -- (A19) -- (A118) -- cycle;
    \draw (A118) -- (A112) -- (A20) -- cycle;
    \draw (A118) -- (A19) -- (A103) -- cycle;
    \draw (A119) -- (A1) -- (A5) -- cycle;
    \draw (A36) -- (A1) -- (A119) -- cycle;
    \draw (A35) -- (A36) -- (A119) -- cycle;
    \draw (A65) -- (A62) -- (A119) -- cycle;
    \draw (A119) -- (A5) -- (A65) -- cycle;
    \draw (A35) -- (A119) -- (A62) -- cycle;
\end{tikzpicture}
  \end{center}
  \caption{\label{fig:Meshes} The meshes on which the conservation of
    divergence or curl are performed. Top left: Cartesian mesh, top right:
  unstructured quadrangular mesh, bottom: unstructured triangular mesh.}
\end{figure}
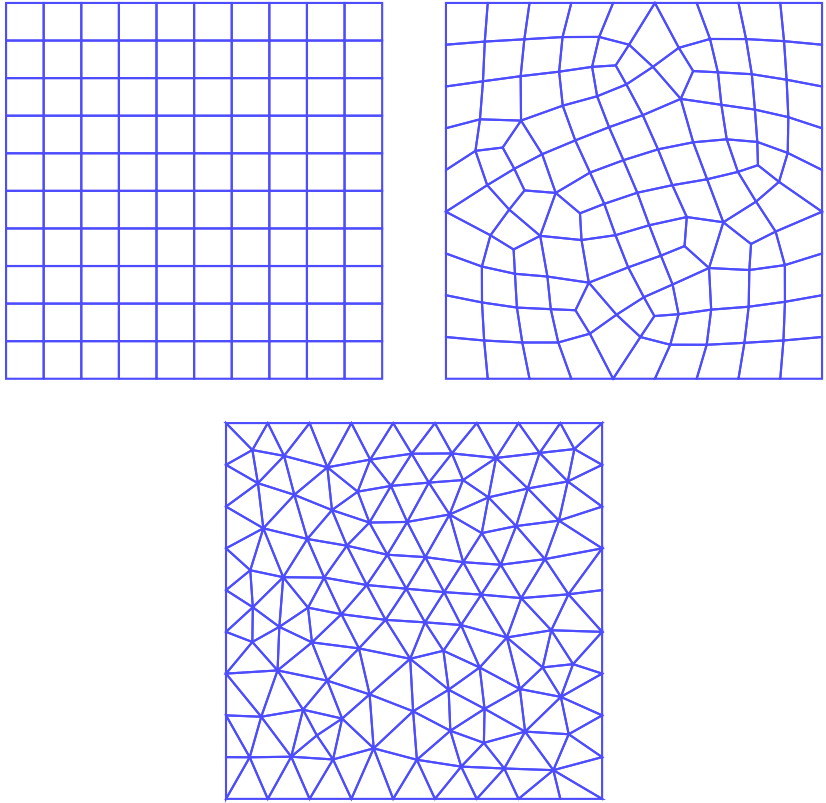

The convergence tests will be performed on Cartesian meshes with
uniform size of the mesh, with $h=0.1$,  $h=0.05$,  $h=0.025$ and
$h=0.0125$.
The convergence tests on triangular meshes will be performed on a set
of unstructured meshes with a number of cells equal to 268, 1036,
4186 and 16682, and minimal size (defined as the square root of the
surface of the smallest element) approximately equal to
$0.0844$, $0.0419$, $0.0223$ and $0.0109$.

For all the test cases, the time stepping is done with classical
Strong-Stability Preserving explicit time schemes \cite{gottlieb2009high}.
If $k$ is the polynomial order of the space discretization,
the order of the time
integration scheme is $k+1$. The CFL number is equal to $0.5$ for
$k=0$, $0.33$ for $k=1$ and $0.2$ for $k=2$.

Concerning the optimal order of convergence we may get, we make
the following remark
\begin{remark}[Optimal order of convergence]
  \label{rem:ConvergenceCartesian}
  It has been known for a while that the convergence order of the
  discontinuous Galerkin method for advection problems is of order
  $k+1/2$ \cite{johnson1986analysis,peterson1991note} in general. On Cartesian
  meshes, the order of convergence can be improved to $k+1$ by using
  tensor products of one dimensional basis \cite{lesaint1974finite}.
  Here, we are working with
  basis for vectors that are not tensor products of one dimensional basis.
  Therefore, the optimal order we can expect is only $k+1/2$, even
  though we will frequently observe $k+1$. 
\end{remark}

\subsection{Discrete conservation of the divergence: Maxwell system}

In this section, we are interested in the two-dimensional Maxwell system in
the vacuum,
which reads
\begin{equation}
  \label{eq:Maxwell}
  \left\{
  \begin{array}{l}
    \partial_t b + \nablaperp \cdot \be = 0 \\
    \partial_t \be + c^2  \nablaperp b = 0,
  \end{array}
  \right.
\end{equation}
where $\be$ is the two-dimensional electric field and $b$ is the
$z$-component of the magnetic field. 

\subsubsection{Conservation of the divergence of a stationary solution}
\label{subsubsec:ConservationDivergence}

For this first test case, the following initial condition is imposed
$$
\left\{
\begin{array}{r@{\, = \, }l}
  b(\bx) & 0\\
  \be_x (\bx) & \overline{x} \, \ex{-\overline{r}^2 / 2}\\
  \be_y (\bx) & \overline{y} \, \ex{-\overline{r}^2 / 2},
\end{array}
\right.
$$
with $\overline{x} = \dfrac{x-x_c}{r_0}$,
$\overline{y} = \dfrac{y-y_c}{r_0}$, $r^2 = (x-x_c)^2+(y-y_c)^2$, and
$\overline{r} = \dfrac{r}{r_0}$. This initial condition is such that
$\nablaperp \cdot \be = 0$, and so is clearly a stationary solution
of \eqref{eq:Maxwell}. This solution was built by considering the
potential $\ex{-\overline{r}^2/2}$
and by taking its gradient. The numerical parameters are $r_0 = 0.15$,
$x_c = y_c = 0.5$, and $c^2 = 1$. The computation is led until $t=3$ on the
different meshes represented in \autoref{fig:Meshes}, and The $L^2$
difference between $\nabla^\star \be$ and its initial value is computed
along the time.
Two configurations of finite element spaces for approximating
$\be$ are used:
\begin{itemize}
\item the finite element spaces $\Bcurl_k$,
\item the classical finite element spaces for discontinuous Galerkin
  methods obtained by tensorization of the scalar basis (note however that
  in the triangular case, these two approximation spaces match),
\end{itemize}
and with two different numerical flux:
\begin{itemize}
\item the Lax-Friedrich flux,
\item the Lax-Friedrich flux with purely tangential diffusion
  \eqref{eq:LaxFriedrichTangential}, which matches in the case of
    the two-dimensional Maxwell system \eqref{eq:Maxwell} with the Godunov'
    flux.
\end{itemize}
Numerical results are represented in
\autoref{fig:StationaryDivergenceConservation} for degree $k=0,1,2$, and
show that the only combination that preserves correctly the divergence is the
one with the finite element space $\Bcurl_k$ and with the Godunov'
numerical flux.
\begin{figure}
  \begin{center}
    \includegraphics[width=14.5cm]{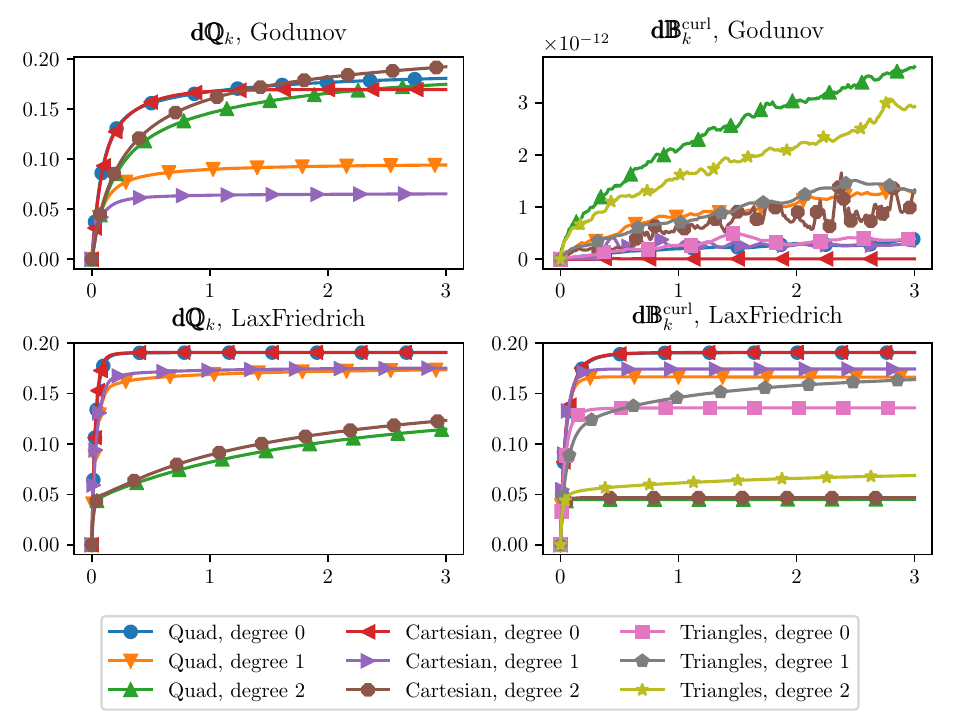}
  \end{center}
  \caption{\label{fig:StationaryDivergenceConservation}
    Plot of $\norme{\nabla^\star \be - \nabla^\star \be_0}_2$ with respect to
    time for different degree, approximation space,
    type of meshes and numerical flux. In the
    left column, Cartesian and unstructured quadrangular meshes are
    considered, with the $\bdQQ_k$ approximation space with the Godunov
    (top figure) and Lax-Friedrich (bottom figure) numerical flux.
    In the right column,
    Cartesian, unstructured quadrangular and triangular meshes are
    considered, with the $\Bcurl_k$ approximation space with the Godunov
    (top figure) and Lax-Friedrich (bottom figure) numerical flux.
    On triangular meshes, the space $\Bcurl_k$ and $\bdPP_k$ are the same, and
    the computations on these spaces are represented on the right column. 
    Note that
    the $y$ scaling of the top right figure ($\Bcurl_k$ with Godunov' scheme)
    is $10^{-12}$, whereas it is of the
    order of $1$ for the other plots. 
  }
\end{figure}

\subsubsection{Convergence test}
\label{subsubsec:ConvergenceDivergence}

The numerical test is taken from \cite[Section 5.1]{munz2000divergence}. 
The initial condition is
\begin{equation}
  \label{eq:exactDivergence}
\left\{
\begin{array}{r@{\, = \, }l}
  b (\bx) & \dfrac{\omega}{c^2}
  \cos \left( k_{\perp} \pi y \right)
  \sin \left( k_\parallel \pi x \right)
  \\
  \be_x  (\bx) & - k_{\perp} \pi \, \sin \left( k_{\perp} \pi y \right)
  \cos \left( k_\parallel \pi x \right)\\
  \be_y  (\bx) & k_{\parallel} \pi 
  \cos \left( k_{\perp} \pi y \right)
  \sin \left( k_\parallel \pi x \right)
  \\
\end{array}
\right.
\end{equation}
and the exact solution is
$$
\left\{
\begin{array}{r@{\, = \, }l}
  b  (\bx,t) & \dfrac{\omega}{c^2}
  \cos \left( k_{\perp} \pi y \right)
  \sin \left( k_\parallel \pi x - \omega t \right)
  \\
  \be_x  (\bx,t)& - k_{\perp} \pi \, \sin \left( k_{\perp} \pi y \right)
  \cos \left( k_\parallel \pi x - \omega t \right)\\
  \be_y  (\bx,t) & k_{\parallel} \pi 
  \cos \left( k_{\perp} \pi y \right)
  \sin \left( k_\parallel \pi x - \omega t \right)
  \\
\end{array}
\right.
$$
The longitudinal and transverse wave numbers $k_\parallel$ and $k_\perp$ are
data of the test case. The frequency $\omega$ is such that
$$k_\parallel ^2 + k_\perp ^2 = \dfrac{\omega^2}{\pi^2 c^2}.$$
The numerical parameters are $c=1$, and 
$k_\parallel = k_\perp = 2$.

A second test case is performed, in which a stationary non divergence free
solution is added to the initial solution \eqref{eq:exactDivergence}
(it is similar to the stationary solution of
\autoref{subsubsec:ConservationDivergence}, but regular),
equal to
\begin{equation}
  \label{eq:DivergenceAddVortex}
  \left\{
  \begin{array}{r@{\, = \, }l}
    b(\bx) & 0\\
    \be_x (\bx) & 2 K_0 \, \alpha \, \overline{x} \, \dfrac{\ex{-\alpha/(1-\overline{r}^2)}}{(1-\overline{r}^2)^2}\\
    \be_y (\bx) & 2 K_0 \, \alpha \, \overline{y} \, \dfrac{\ex{-\alpha/(1-\overline{r}^2)}}{(1-\overline{r}^2)^2},
  \end{array}
  \right.
\end{equation}
if $r < r_0$, and $0$ otherwise. The numerical parameters are
$K_0 = 100$, $r_0 = 0.35$, $x_c = y_c = 0.5$, $\alpha=4$. 

The convergence curves obtained on the two test cases
on each variable is plotted in \autoref{fig:DivergenceConvergenceQuad}. 
The errors obtained and rate of convergence for Cartesian meshes on the
variable $\be_x$ are gathered in 
\autoref{tab:DivergenceConvergenceQuadSineCos} for the case with
the initial condition \eqref{eq:exactDivergence}
and in \autoref{tab:DivergenceConvergenceQuadSineCosPlusVortex} where
\eqref{eq:DivergenceAddVortex} was added to 
the initial condition \eqref{eq:exactDivergence}. 
With the single initial condition \eqref{eq:exactDivergence}, the
order of convergence obtained is not very sensitive to the change of
approximation basis or of the change of numerical flux. This is probably due
to the fact that \eqref{eq:exactDivergence} is already divergence free, and
so the numerical divergence is already low when
computing the initial condition, 
and so the benefit of preserving exactly the divergence by using the
space $\Bcurl_k$ with Godunov flux is low.

The same plots are performed on the series of triangular meshes, shown in
\autoref{fig:DivergenceConvergenceTriangle}, and the errors obtained
for the variable $\be_x$ are shown in
\autoref{tab:DivergenceConvergenceTriangleSineCos}
and
\autoref{tab:DivergenceConvergenceTriangleSineCosPlusVortex}.
The same observation as for Cartesian meshes hold. 

\begin{figure}
  \begin{center}
    \begin{tabular}{r@{\qquad}l}
      \includegraphics[width=6cm]{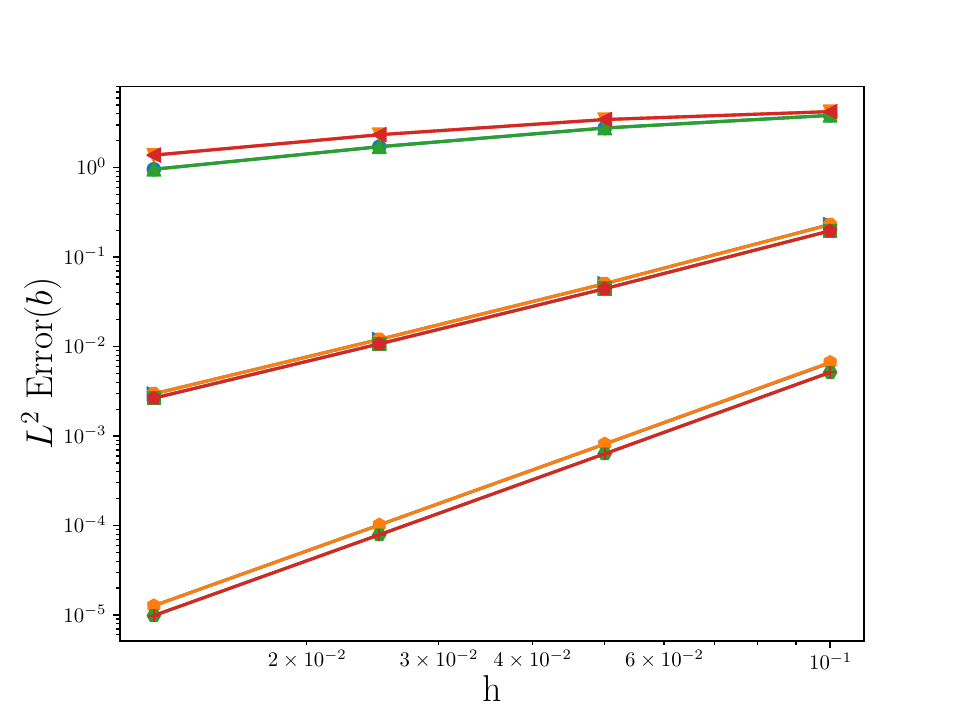}
      &
      \includegraphics[width=6cm]{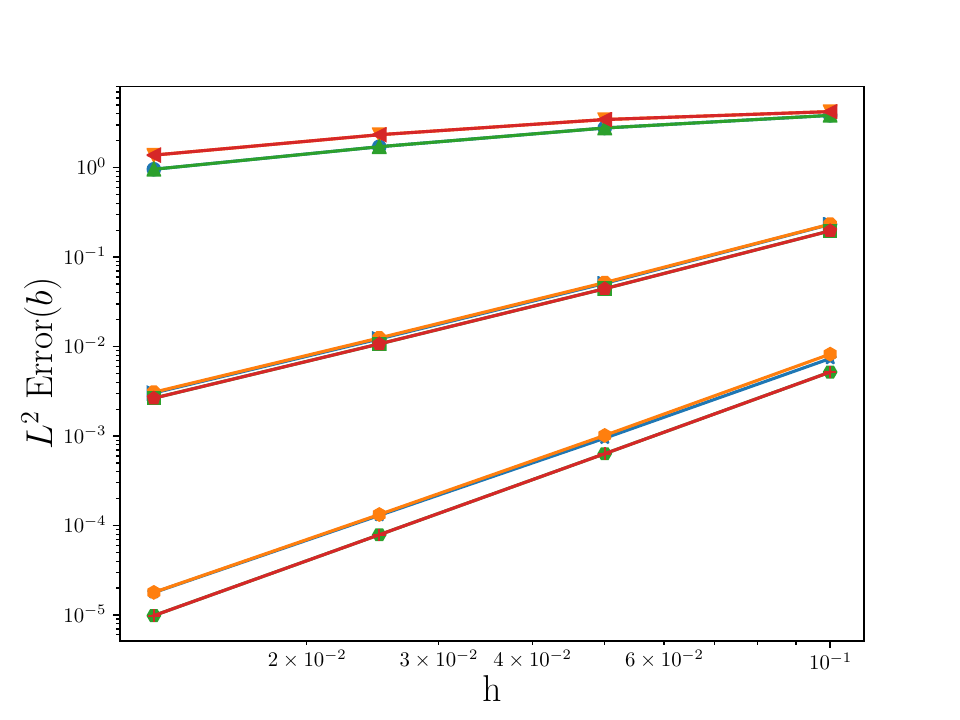} \\
      \includegraphics[width=6cm]{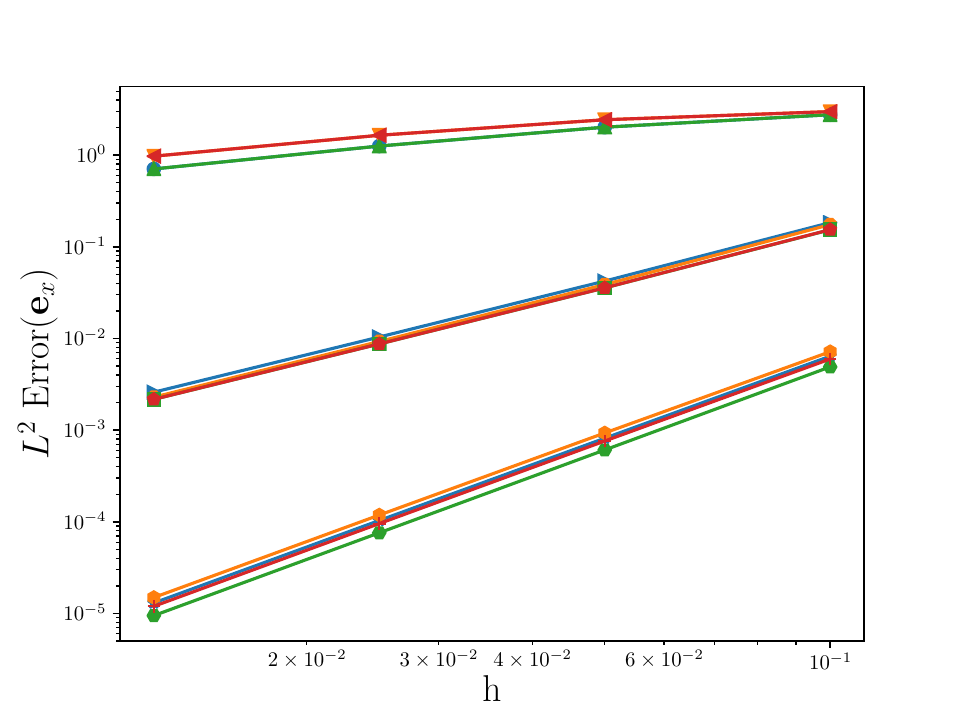}
      &
      \includegraphics[width=6cm]{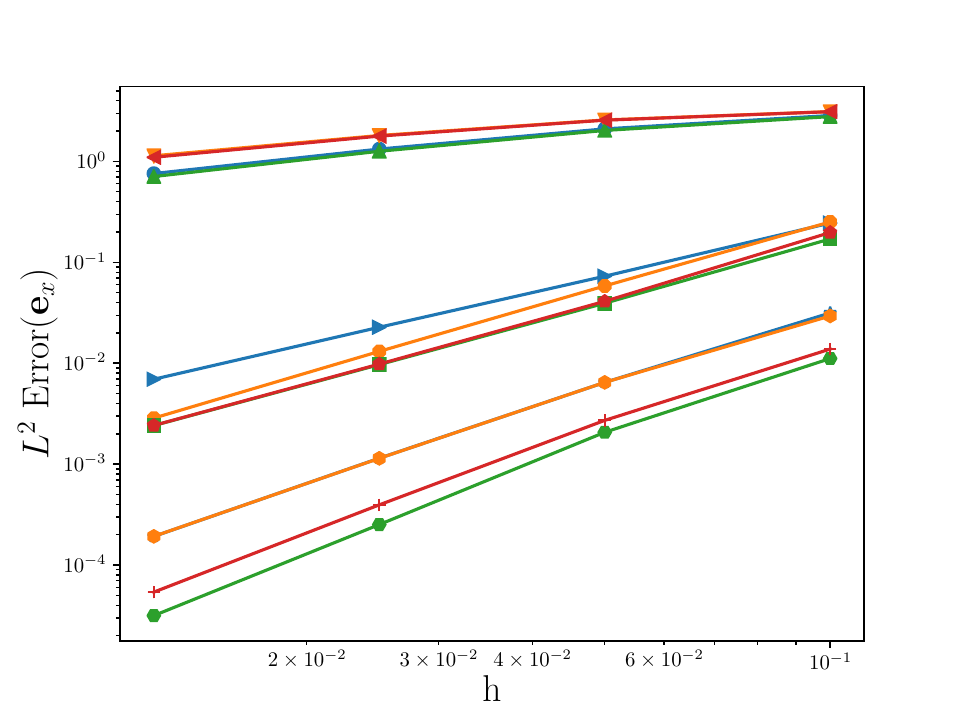} \\
      \includegraphics[width=6cm]{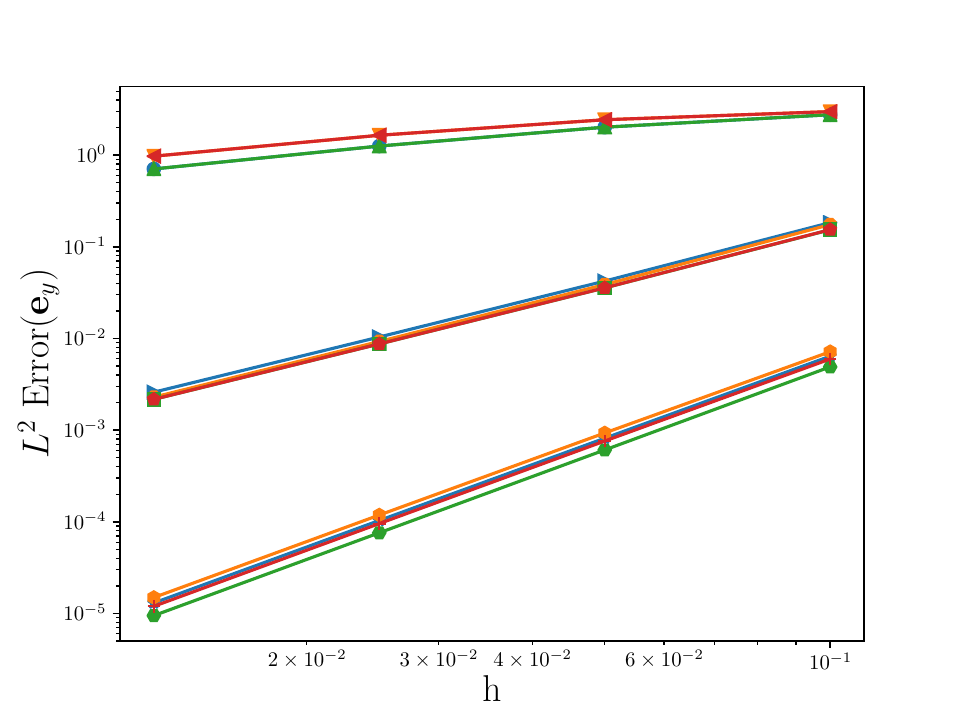}
      &
      \includegraphics[width=6cm]{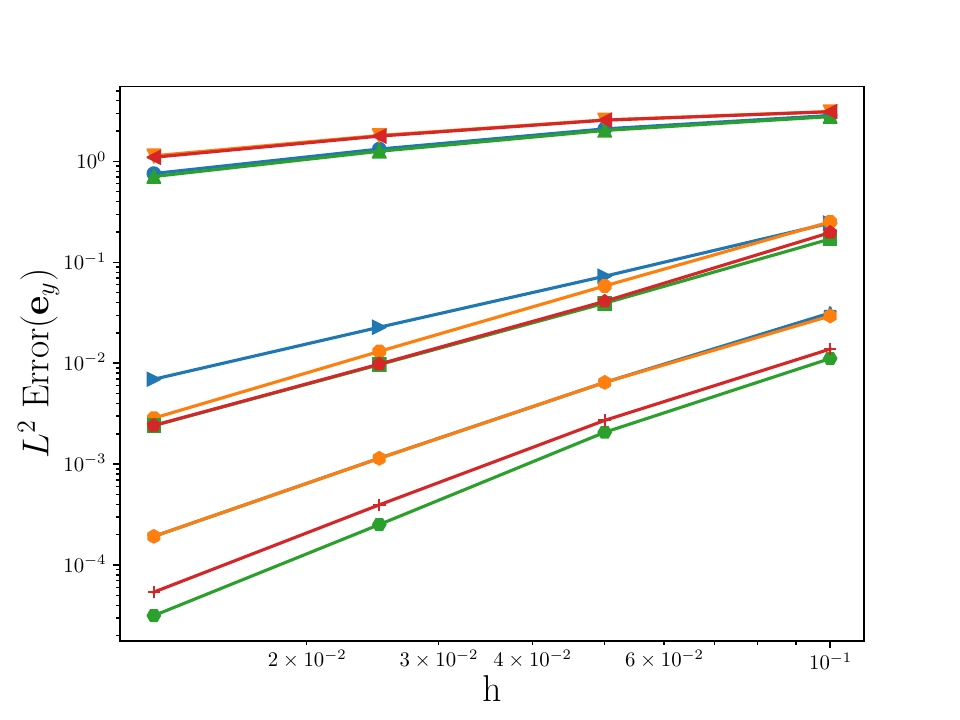} \\
      \multicolumn{2}{c}{
        \includegraphics[width=8cm]{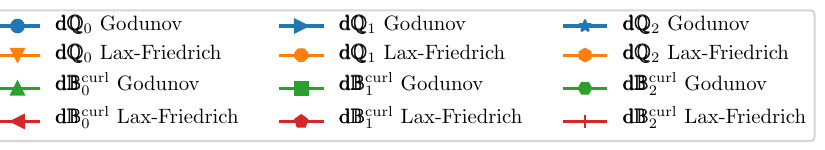}
      }
      \\
    \end{tabular}
  \end{center}
  \caption{\label{fig:DivergenceConvergenceQuad} Error obtained
    on the test case described in
    \autoref{subsubsec:ConvergenceDivergence} 
    with initial condition \eqref{eq:exactDivergence} on the left, and
    with the initial condition \eqref{eq:DivergenceAddVortex} added to
    \eqref{eq:exactDivergence} on the right, on a series of
    Cartesian meshes. For each of the test cases,
    the error obtained for the variables $b$ (top row),
    $\be_x$ (middle row) and $\be_y$ (bottom row)
    is shown for different
    approximation spaces and for the Lax-Friedrich and Godunov flux. 
  }
\end{figure}

\begin{table}
  \begin{center}
\begin{tabular}{||c||c|c||c|c||c|c||c|c||}\multicolumn{1}{c}{}& \multicolumn{4}{c}{Godunov}& \multicolumn{4}{c}{Lax-Friedrich}\\ \hline\multicolumn{1}{||c||}{}& \multicolumn{2}{|c||}{$\bdQQ_0$}& \multicolumn{2}{|c||}{$\Bcurl_0$}& \multicolumn{2}{|c||}{$\bdQQ_0$}& \multicolumn{2}{|c||}{$\Bcurl_0$}\\ \hline$h$  & Error & rate  & Error & rate  & Error & rate  & Error & rate \\ \hline0.1 & 2.76e+00 &  & 2.76e+00 &  & 2.99e+00 &  & 2.99e+00 &  \\ 0.05 & 2.02e+00  & 0.45 & 2.02e+00  & 0.45 & 2.44e+00  & 0.30 & 2.44e+00  & 0.30\\ 0.025 & 1.26e+00  & 0.68 & 1.26e+00  & 0.68 & 1.65e+00  & 0.56 & 1.65e+00  & 0.56\\ 0.0125 & 7.08e-01  & 0.83 & 7.08e-01  & 0.83 & 9.76e-01  & 0.76 & 9.76e-01  & 0.76\\ \hline \multicolumn{1}{c}{}& \multicolumn{4}{c}{Godunov}& \multicolumn{4}{c}{Lax-Friedrich}\\ \hline\multicolumn{1}{||c||}{}& \multicolumn{2}{|c||}{$\bdQQ_1$}& \multicolumn{2}{|c||}{$\Bcurl_1$}& \multicolumn{2}{|c||}{$\bdQQ_1$}& \multicolumn{2}{|c||}{$\Bcurl_1$}\\ \hline$h$  & Error & rate  & Error & rate  & Error & rate  & Error & rate \\ \hline0.1 & 1.84e-01 &  & 1.54e-01 &  & 1.75e-01 &  & 1.54e-01 &  \\ 0.05 & 4.23e-02  & 2.12 & 3.57e-02  & 2.11 & 3.88e-02  & 2.17 & 3.57e-02  & 2.11\\ 0.025 & 1.04e-02  & 2.03 & 8.72e-03  & 2.03 & 9.28e-03  & 2.06 & 8.72e-03  & 2.03\\ 0.0125 & 2.60e-03  & 2.00 & 2.17e-03  & 2.01 & 2.29e-03  & 2.02 & 2.17e-03  & 2.01\\ \hline \multicolumn{1}{c}{}& \multicolumn{4}{c}{Godunov}& \multicolumn{4}{c}{Lax-Friedrich}\\ \hline\multicolumn{1}{||c||}{}& \multicolumn{2}{|c||}{$\bdQQ_2$}& \multicolumn{2}{|c||}{$\Bcurl_2$}& \multicolumn{2}{|c||}{$\bdQQ_2$}& \multicolumn{2}{|c||}{$\Bcurl_2$}\\ \hline$h$  & Error & rate  & Error & rate  & Error & rate  & Error & rate \\ \hline0.1 & 6.41e-03 &  & 4.90e-03 &  & 7.17e-03 &  & 5.98e-03 &  \\ 0.05 & 8.16e-04  & 2.97 & 6.08e-04  & 3.01 & 9.32e-04  & 2.94 & 7.61e-04  & 2.97\\ 0.025 & 1.03e-04  & 2.98 & 7.59e-05  & 3.00 & 1.19e-04  & 2.97 & 9.57e-05  & 2.99\\ 0.0125 & 1.30e-05  & 2.99 & 9.48e-06  & 3.00 & 1.49e-05  & 2.99 & 1.20e-05  & 3.00\\ \hline \end{tabular}

  \end{center}
  \caption{\label{tab:DivergenceConvergenceQuadSineCos} Errors and
    convergence rates obtained on the variable $\be_x$
    with the test case described in \autoref{subsubsec:ConvergenceDivergence}
    with initial condition
    \eqref{eq:exactDivergence}, on a series of quadrangular meshes.
    Results show a low benefit in using the Godunov flux, namely in
    exactly preserving the divergence.}
\end{table}

\begin{table}
  \begin{center}
\begin{tabular}{||c||c|c||c|c||c|c||c|c||}\multicolumn{1}{c}{}& \multicolumn{4}{c}{Godunov}& \multicolumn{4}{c}{Lax-Friedrich}\\ \hline\multicolumn{1}{||c||}{}& \multicolumn{2}{|c||}{$\bdQQ_0$}& \multicolumn{2}{|c||}{$\Bcurl_0$}& \multicolumn{2}{|c||}{$\bdQQ_0$}& \multicolumn{2}{|c||}{$\Bcurl_0$}\\ \hline$h$  & Error & rate  & Error & rate  & Error & rate  & Error & rate \\ \hline0.1 & 2.85e+00 &  & 2.78e+00 &  & 3.12e+00 &  & 3.11e+00 &  \\ 0.05 & 2.10e+00  & 0.44 & 2.03e+00  & 0.46 & 2.58e+00  & 0.27 & 2.56e+00  & 0.28\\ 0.025 & 1.33e+00  & 0.66 & 1.26e+00  & 0.69 & 1.81e+00  & 0.51 & 1.78e+00  & 0.52\\ 0.0125 & 7.59e-01  & 0.81 & 7.09e-01  & 0.83 & 1.14e+00  & 0.67 & 1.10e+00  & 0.70\\ \hline \multicolumn{1}{c}{}& \multicolumn{4}{c}{Godunov}& \multicolumn{4}{c}{Lax-Friedrich}\\ \hline\multicolumn{1}{||c||}{}& \multicolumn{2}{|c||}{$\bdQQ_1$}& \multicolumn{2}{|c||}{$\Bcurl_1$}& \multicolumn{2}{|c||}{$\bdQQ_1$}& \multicolumn{2}{|c||}{$\Bcurl_1$}\\ \hline$h$  & Error & rate  & Error & rate  & Error & rate  & Error & rate \\ \hline0.1 & 2.45e-01 &  & 1.71e-01 &  & 2.52e-01 &  & 1.98e-01 &  \\ 0.05 & 7.29e-02  & 1.75 & 3.94e-02  & 2.12 & 5.84e-02  & 2.11 & 4.11e-02  & 2.26\\ 0.025 & 2.28e-02  & 1.68 & 9.70e-03  & 2.02 & 1.31e-02  & 2.15 & 9.83e-03  & 2.07\\ 0.0125 & 6.96e-03  & 1.71 & 2.42e-03  & 2.00 & 2.87e-03  & 2.20 & 2.43e-03  & 2.02\\ \hline \multicolumn{1}{c}{}& \multicolumn{4}{c}{Godunov}& \multicolumn{4}{c}{Lax-Friedrich}\\ \hline\multicolumn{1}{||c||}{}& \multicolumn{2}{|c||}{$\bdQQ_2$}& \multicolumn{2}{|c||}{$\Bcurl_2$}& \multicolumn{2}{|c||}{$\bdQQ_2$}& \multicolumn{2}{|c||}{$\Bcurl_2$}\\ \hline$h$  & Error & rate  & Error & rate  & Error & rate  & Error & rate \\ \hline0.1 & 3.15e-02 &  & 1.12e-02 &  & 2.94e-02 &  & 1.38e-02 &  \\ 0.05 & 6.47e-03  & 2.29 & 2.07e-03  & 2.43 & 6.48e-03  & 2.18 & 2.72e-03  & 2.35\\ 0.025 & 1.15e-03  & 2.50 & 2.53e-04  & 3.04 & 1.14e-03  & 2.50 & 3.96e-04  & 2.78\\ 0.0125 & 1.93e-04  & 2.57 & 3.17e-05  & 3.00 & 1.93e-04  & 2.57 & 5.41e-05  & 2.87\\ \hline \end{tabular}

  \end{center}
  \caption{\label{tab:DivergenceConvergenceQuadSineCosPlusVortex} Errors and
    convergence rates obtained on the variable $\be_x$
    with the test case described in \autoref{subsubsec:ConvergenceDivergence}
    with initial condition
    \eqref{eq:exactDivergence}, on a series of Cartesian meshes.
    Results show a high benefit in using
    the space $\Bcurl_k$ with the Godunov flux, which is the only one
    to always reach the optimal order.}
\end{table}

\begin{figure}
  \begin{center}
    \begin{tabular}{r@{\qquad}l}
      \includegraphics[width=6cm]{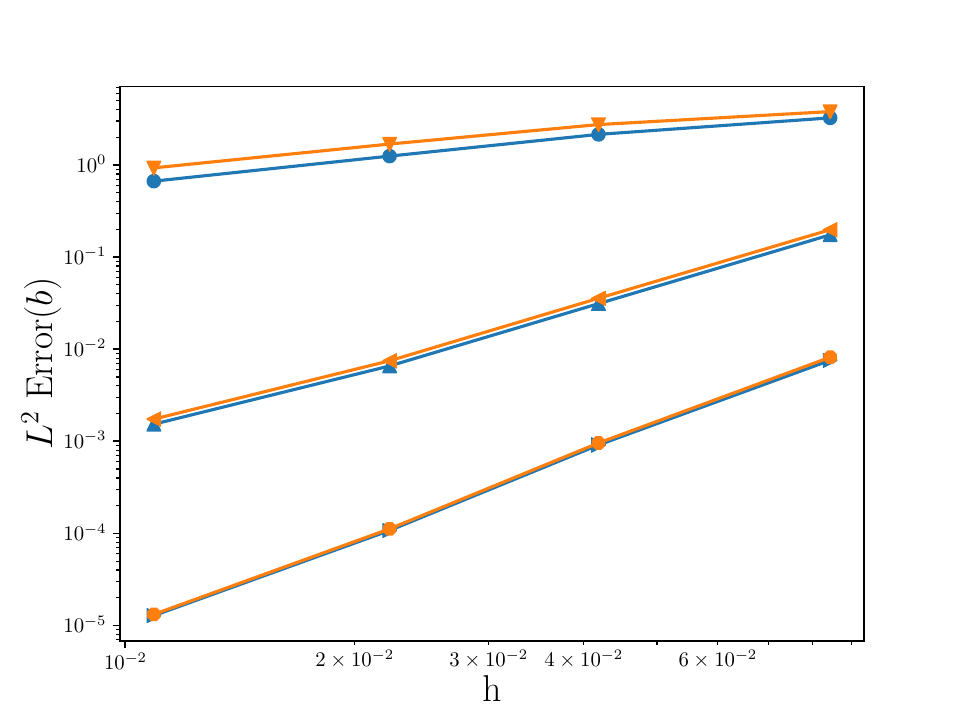}
      &
      \includegraphics[width=6cm]{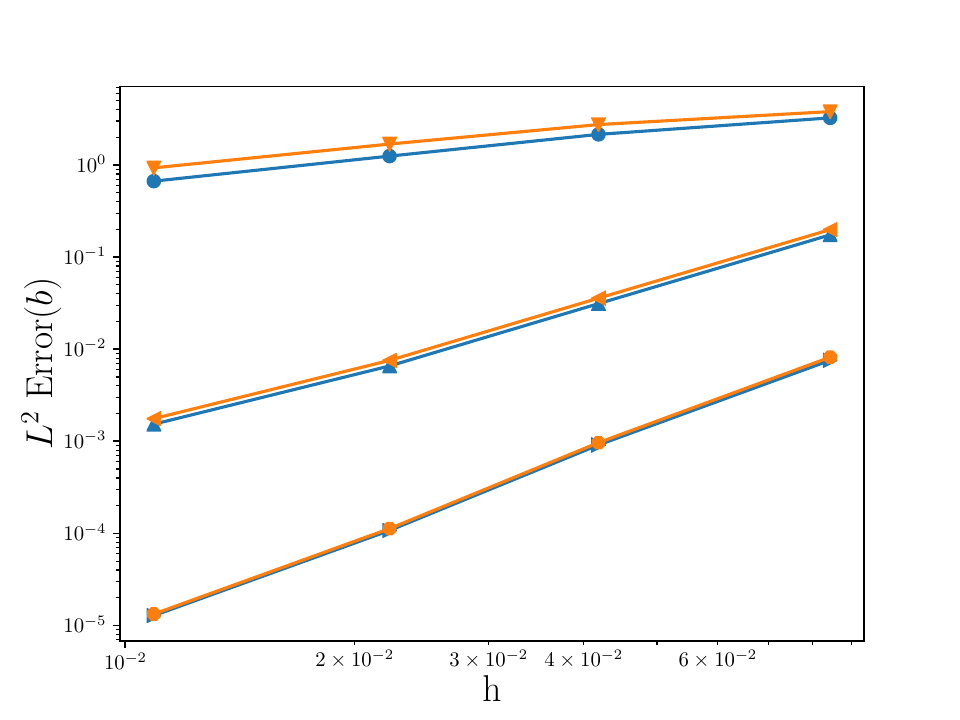} \\
      \includegraphics[width=6cm]{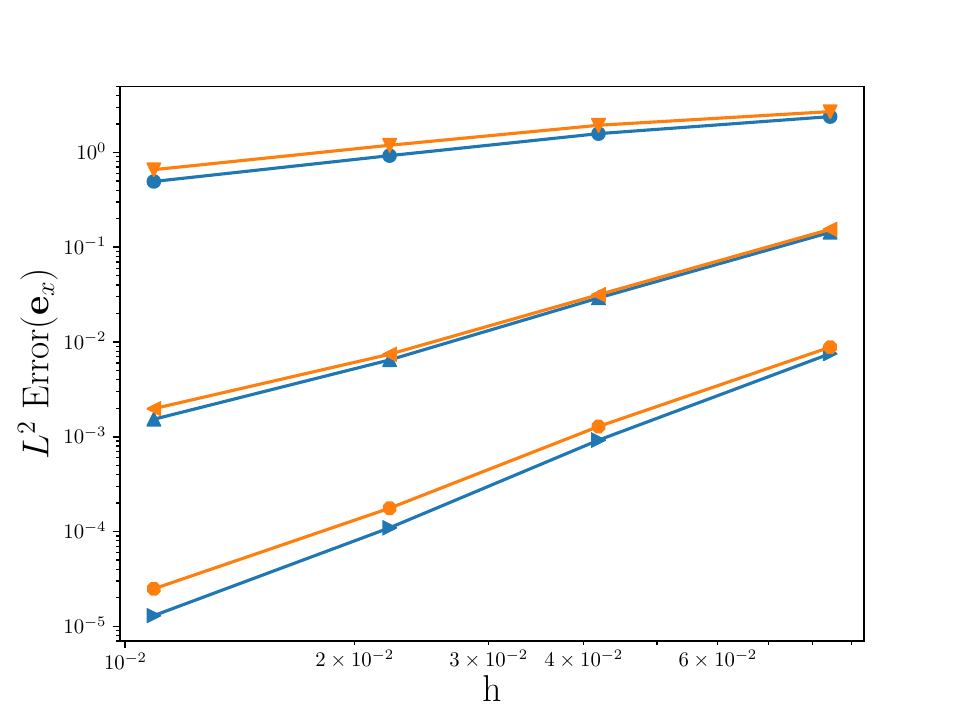}
      &
      \includegraphics[width=6cm]{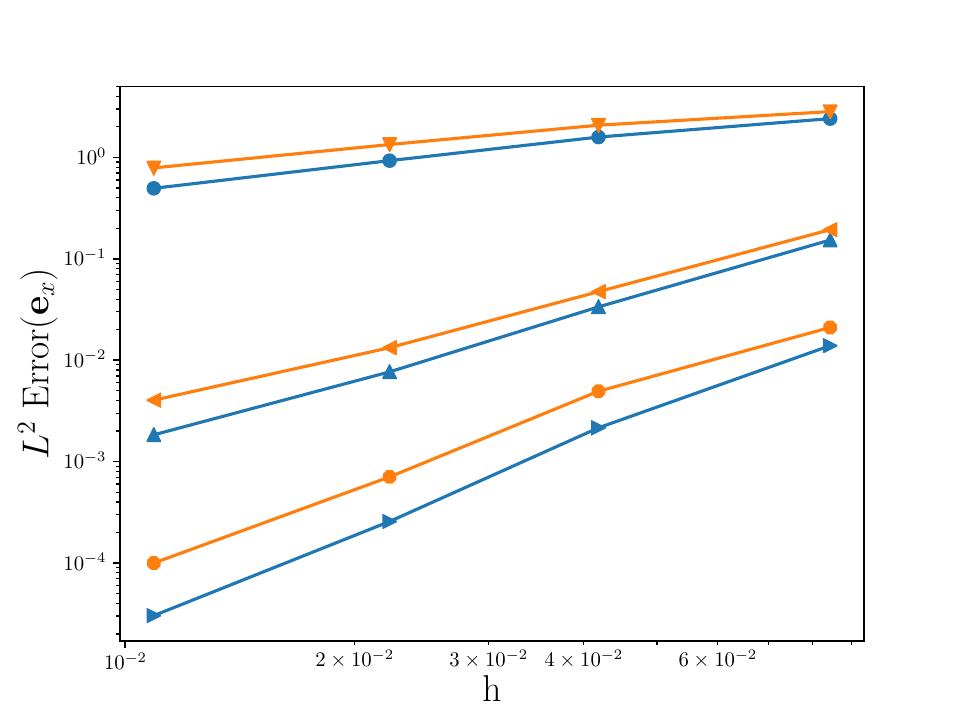} \\
      \includegraphics[width=6cm]{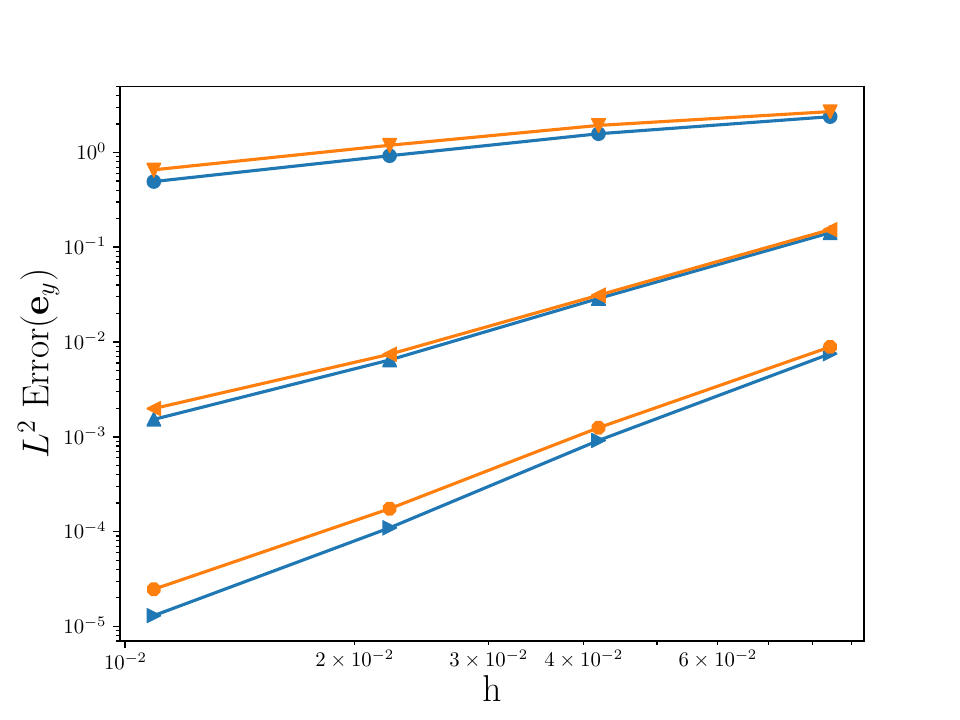}
      &
      \includegraphics[width=6cm]{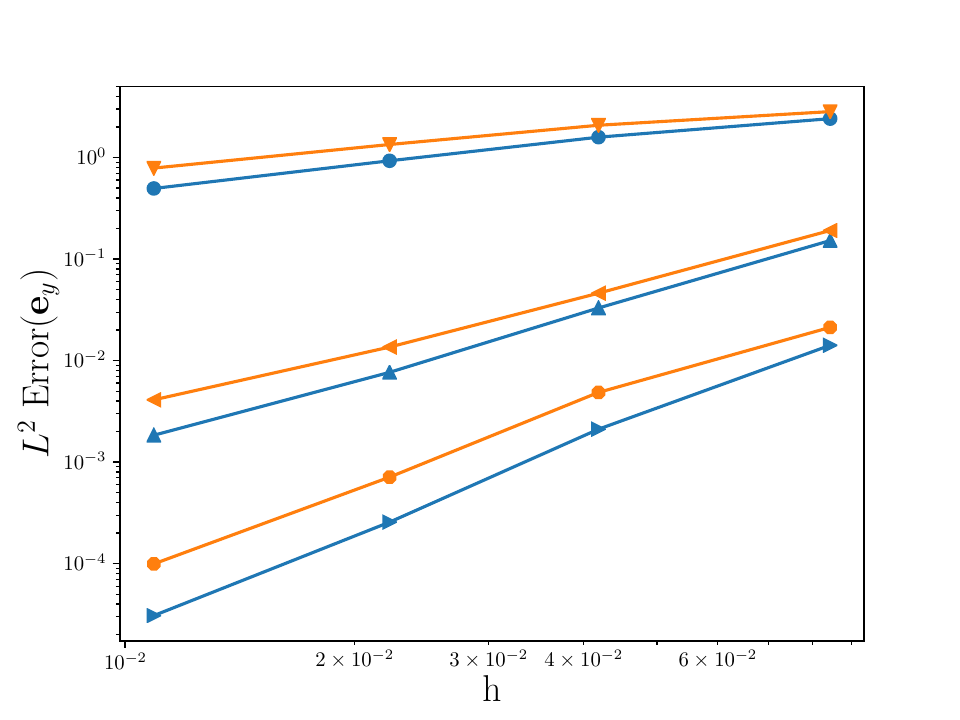} \\
      \multicolumn{2}{c}{
        \includegraphics[width=8cm]{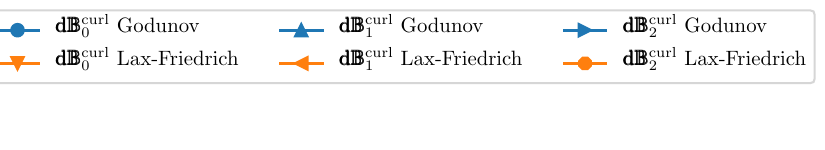}
      }
      \\
    \end{tabular}
  \end{center}
  \caption{\label{fig:DivergenceConvergenceTriangle} Error obtained
    on the test case described in
    \autoref{subsubsec:ConvergenceDivergence} 
    with initial condition \eqref{eq:exactDivergence} on the left, and
    with the initial condition \eqref{eq:DivergenceAddVortex} added to
    \eqref{eq:exactDivergence} on the right, on a series of
    unstructured triangular meshes. For each of the test cases,
    the error obtained for the variables $b$ (top row),
    $\be_x$ (middle row) and $\be_y$ (bottom row)
    is shown for different degrees, and for
    the Lax-Friedrich and Godunov numerical flux. 
  }
\end{figure}

\begin{table}
  \begin{center}
\begin{tabular}{||c||c|c||c|c||}
\hline 
\multicolumn{1}{||c||}{}
& \multicolumn{2}{|c||}{Godunov}
& \multicolumn{2}{|c||}{Lax-Friedrich}
\\ \hline
$h$ 
 & Error & rate 
 & Error & rate 
\\ \hline
0.08439823 & 2.39e+00 & 
 & 2.70e+00 & 
 \\ 
0.04187874 & 1.58e+00  & 0.59 & 1.93e+00  & 0.47\\ 
0.02226197 & 9.25e-01  & 0.85 & 1.19e+00  & 0.77\\ 
0.01091023 & 4.94e-01  & 0.88 & 6.58e-01  & 0.83\\ \hline 
\multicolumn{1}{||c||}{}
& \multicolumn{2}{|c||}{Godunov}
& \multicolumn{2}{|c||}{Lax-Friedrich}
\\ \hline
$h$ 
 & Error & rate 
 & Error & rate 
\\ \hline
0.08439823 & 1.43e-01 & 
 & 1.54e-01 & 
 \\ 
0.04187874 & 2.92e-02  & 2.27 & 3.15e-02  & 2.27\\ 
0.02226197 & 6.47e-03  & 2.38 & 7.43e-03  & 2.29\\ 
0.01091023 & 1.53e-03  & 2.02 & 1.98e-03  & 1.85\\ \hline 
\multicolumn{1}{||c||}{}
& \multicolumn{2}{|c||}{Godunov}
& \multicolumn{2}{|c||}{Lax-Friedrich}
\\ \hline
$h$ 
 & Error & rate 
 & Error & rate 
\\ \hline
0.08439823 & 7.53e-03 & 
 & 8.83e-03 & 
 \\ 
0.04187874 & 9.22e-04  & 3.00 & 1.28e-03  & 2.75\\ 
0.02226197 & 1.10e-04  & 3.37 & 1.76e-04  & 3.14\\ 
0.01091023 & 1.30e-05  & 2.99 & 2.49e-05  & 2.75\\ \hline 
\end{tabular}

  \end{center}
  \caption{\label{tab:DivergenceConvergenceTriangleSineCos} Errors and
    convergence rates obtained on the variable $\be_x$
    with the test case described in \autoref{subsubsec:ConvergenceDivergence}
    with initial condition
    \eqref{eq:exactDivergence}, on the series of triangular meshes.
    Results show a low benefit in using the Godunov flux, namely in
    exactly preserving the divergence.}
\end{table}

\begin{table}
  \begin{center}
\begin{tabular}{||c||c|c||c|c||}
\hline 
\multicolumn{1}{||c||}{}
& \multicolumn{2}{|c||}{Godunov}
& \multicolumn{2}{|c||}{Lax-Friedrich}
\\ \hline
$h$ 
 & Error & rate 
 & Error & rate 
\\ \hline
0.08439823 & 2.41e+00 & 
 & 2.83e+00 & 
 \\ 
0.04187874 & 1.59e+00  & 0.60 & 2.08e+00  & 0.44\\ 
0.02226197 & 9.28e-01  & 0.85 & 1.34e+00  & 0.70\\ 
0.01091023 & 4.95e-01  & 0.88 & 7.88e-01  & 0.74\\ \hline 
\multicolumn{1}{||c||}{}
& \multicolumn{2}{|c||}{Godunov}
& \multicolumn{2}{|c||}{Lax-Friedrich}
\\ \hline
$h$ 
 & Error & rate 
 & Error & rate 
\\ \hline
0.08439823 & 1.53e-01 & 
 & 1.94e-01 & 
 \\ 
0.04187874 & 3.36e-02  & 2.17 & 4.74e-02  & 2.01\\ 
0.02226197 & 7.67e-03  & 2.34 & 1.33e-02  & 2.01\\ 
0.01091023 & 1.84e-03  & 2.00 & 4.04e-03  & 1.67\\ \hline 
\multicolumn{1}{||c||}{}
& \multicolumn{2}{|c||}{Godunov}
& \multicolumn{2}{|c||}{Lax-Friedrich}
\\ \hline
$h$ 
 & Error & rate 
 & Error & rate 
\\ \hline
0.08439823 & 1.39e-02 & 
 & 2.10e-02 & 
 \\ 
0.04187874 & 2.15e-03  & 2.66 & 4.94e-03  & 2.07\\ 
0.02226197 & 2.57e-04  & 3.37 & 7.07e-04  & 3.08\\ 
0.01091023 & 3.03e-05  & 3.00 & 9.99e-05  & 2.74\\ \hline 
\end{tabular}

  \end{center}
  \caption{\label{tab:DivergenceConvergenceTriangleSineCosPlusVortex} Errors and
    convergence rates obtained on the variable $\be_x$
    with the test case described in \autoref{subsubsec:ConvergenceDivergence}
    with initial condition
    \eqref{eq:exactDivergence}, on the series of triangular meshes.
    Results show a higher benefit than in
    \autoref{tab:DivergenceConvergenceTriangleSineCos} 
    in using the Godunov flux for preserving exactly the divergence.}
\end{table}

\subsection{Discrete conservation of a curl: Wave system}

In this section, the wave system
\begin{equation}
  \label{eq:Waves}
  \left\{
  \begin{array}{l}
    \partial_t p + \nabla \cdot \bu = 0 \\
    \partial_t \bu + c^2 \nabla p = 0,
  \end{array}
  \right.
\end{equation}
which couples the pressure $p$ and the velocity $\bu$ is considered. The
wave velocity $c$ is a parameter of the system, and will be always equal
to $1$ in the numerical applications. 

\subsubsection{Conservation of the curl of a stationary solution}
\label{subsubsec:ConservationCurl}

For this first test case, the following initial condition is imposed
$$
\left\{
\begin{array}{r@{\, = \, }l}
  p(\bx) & 0\\
  \bu_x (\bx) & - \overline{y} \, \ex{-\overline{r}^2 / 2}\\
  \bu_y (\bx) & \overline{x} \, \ex{-\overline{r}^2 / 2},
\end{array}
\right.
$$
with $\overline{x} = \dfrac{x-x_c}{r_0}$,
$\overline{y} = \dfrac{y-y_c}{r_0}$, $r^2 = (x-x_c)^2+(y-y_c)^2$, and
$\overline{r} = \dfrac{r}{r_0}$. This initial condition is such that
$\nabla \cdot \bu = 0$, and so is clearly a stationary solution
of \eqref{eq:Waves}. This solution was built by considering the
potential $\ex{-\overline{r}^2/2}$
and by taking its rotated gradient $\nablaperp$.
The numerical parameters are $r_0 = 0.15$, and 
$x_c = y_c = 0.5$. The computation is led until $t=3$ on the
different meshes represented in \autoref{fig:Meshes}, and The $L^2$
difference between $\left( \nablaperp \right)^\star \bu$
and its initial value is computed along the time.
Two configurations of finite element spaces for approximating
$\bu$ are used:
\begin{itemize}
\item the finite element spaces $\Bdiv_k$,
\item the classical finite element spaces for discontinuous Galerkin
  methods obtained by tensorization of the scalar basis (note however that
  in the triangular case, these two approximation spaces match),
\end{itemize}
and with two different numerical flux:
\begin{itemize}
\item the Lax-Friedrich flux,
\item the Lax-Friedrich flux with purely normal diffusion
  \eqref{eq:LaxFriedrichNormal}, which matches in our case with
  the Godunov' flux.
\end{itemize}
Numerical results are represented in
\autoref{fig:StationaryCurlConservation} for degree $k=0,1,2$, and
show that the only combination that preserves correctly the curl is the
one with the finite element space $\Bdiv_k$ and with the Godunov'
numerical flux.

\begin{figure}
  \begin{center}
    \includegraphics[width=14.5cm]{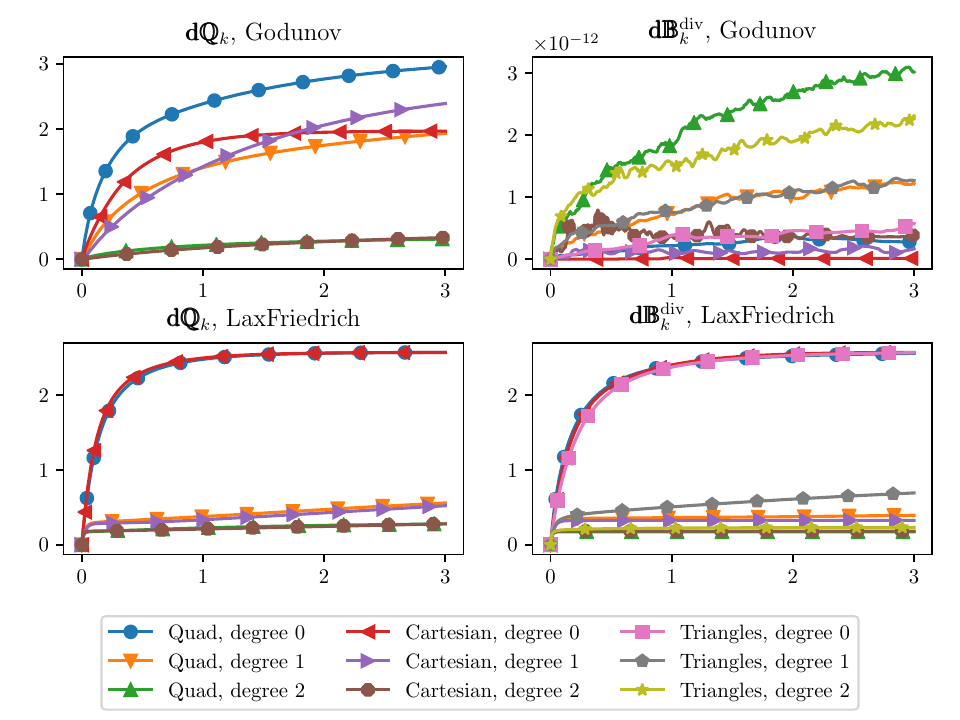}
  \end{center}
  \caption{ \label{fig:StationaryCurlConservation}
    Plot of $\norme{\left( \nablaperp \right)^\star \bu
      - \left( \nablaperp \right) ^\star \bu_0}_2$ with respect to
    time for different degree, approximation space,
    type of meshes and numerical flux. In the
    left column, Cartesian and unstructured quadrangular meshes are
    considered, with the $\bdQQ_k$ approximation space with the Godunov
    (top figure) and Lax-Friedrich (bottom figure) numerical flux.
    In the right column,
    Cartesian, unstructured quadrangular and triangular meshes are
    considered, with the $\Bdiv_k$ approximation space with the Godunov
    (top figure) and Lax-Friedrich (bottom figure) numerical flux.
    On triangular meshes, the space $\Bdiv_k$ and $\bdPP_k$ are the same, and
    the computations on these spaces are represented only
    on the right column. Note that
    the $y$ scaling of the top right figure ($\Bdiv_k$ with Godunov' scheme)
    is $10^{-12}$, whereas it is of the order of $1$ for the other plots. 
  }
\end{figure}

\subsubsection{Convergence test}
\label{subsubsec:ConvergenceCurl}

The numerical test is adapted from the one for the divergence convergence
test case of \autoref{subsubsec:ConvergenceDivergence}. 
The initial condition is
\begin{equation}
  \label{eq:exactCurl}
  \left\{
  \begin{array}{r@{\, = \, }l}
    p (\bx) & \dfrac{\omega}{c^2}
    \sin \left( k_{\perp} \pi y \right)
    \cos \left( k_\parallel \pi x \right)
    \\
    \bu_x  (\bx) &  k_{\parallel} \pi \, \cos \left( k_{\perp} \pi y \right)
    \sin \left( k_\parallel \pi x \right)\\
    \bu_y  (\bx) & k_{\perp} \pi 
    \sin \left( k_{\perp} \pi y \right)
    \cos \left( k_\parallel \pi x \right)
    \\
  \end{array}
  \right.
\end{equation}
and the exact solution is
$$
\left\{
\begin{array}{r@{\, = \, }l}
  p (\bx) & \dfrac{\omega}{c^2}
  \sin \left( k_{\perp} \pi y - \omega t \right)
  \cos \left( k_\parallel \pi x \right)
  \\
  \bu_x  (\bx) &  k_{\parallel} \pi \, \cos \left( k_{\perp} \pi y - \omega t \right)
  \sin \left( k_\parallel \pi x \right)\\
  \bu_y  (\bx) & k_{\perp} \pi 
  \sin \left( k_{\perp} \pi y - \omega t \right)
  \cos \left( k_\parallel \pi x \right).
  \\
\end{array}
\right.
$$
The longitudinal and transverse wave numbers $k_\parallel$ and $k_\perp$ are
data of the test case. The frequency $\omega$ is such that
$$k_\parallel ^2 + k_\perp ^2 = \dfrac{\omega^2}{\pi^2 c^2}.$$
The numerical parameters are $c=1$, and 
$k_\parallel = k_\perp = 2$.

A second test case is performed, in which a stationary non curl free
solution is added to the initial solution \eqref{eq:exactCurl}
(it is similar to the stationary solution of
\autoref{subsubsec:ConservationCurl}, but regular),
equal to
\begin{equation}
  \label{eq:CurlAddVortex}
  \left\{
  \begin{array}{r@{\, = \, }l}
    b(\bx) & 0\\
    \bu_x (\bx) & - 2 K_0 \, \alpha \, \overline{y} \, \dfrac{\ex{-\alpha/(1-\overline{r}^2)}}{(1-\overline{r}^2)^2}\\
    \bu_y (\bx) & 2 K_0 \, \alpha \, \overline{x} \, \dfrac{\ex{-\alpha/(1-\overline{r}^2)}}{(1-\overline{r}^2)^2},
  \end{array}
  \right.
\end{equation}
if $r < r_0$, and $0$ otherwise. The numerical parameters are
$K_0 = 100$, $r_0 = 0.35$, $x_c = y_c = 0.5$, $\alpha=4$. 
The convergence curve for this test case with the two initial conditions
are represented in \autoref{fig:CurlConvergenceQuad} for
Cartesian meshes, and the matching  table
with the convergence errors and rates are summarized in 
\autoref{tab:CurlConvergenceQuadSineCos} and
\autoref{tab:CurlConvergenceQuadSineCosPlusVortex}.
For triangular meshes, the convergence curves are plotted in 
\autoref{fig:CurlConvergenceTriangle}, and the errors and convergence rates
are summarized in the \autoref{tab:CurlConvergenceTriangleSineCos}
and \autoref{tab:CurlConvergenceTriangleSineCosPlusVortex}.
Same comments as for the convergence test case of the previous section
hold. 

\begin{figure}
  \begin{center}
    \begin{tabular}{r@{\qquad}l}
      \includegraphics[width=6cm]{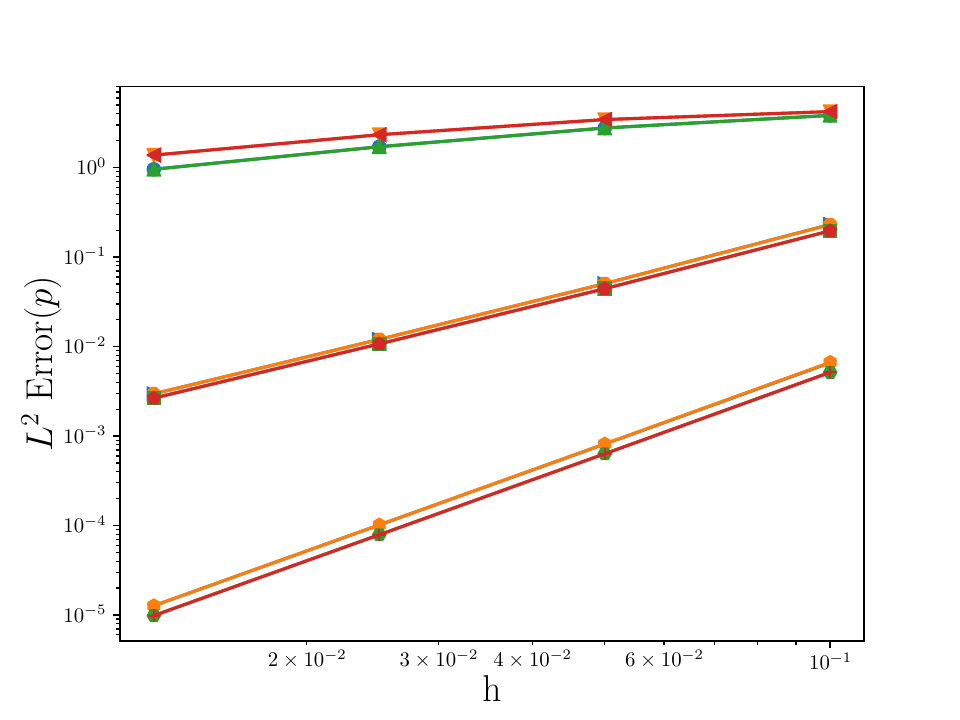}
      &
      \includegraphics[width=6cm]{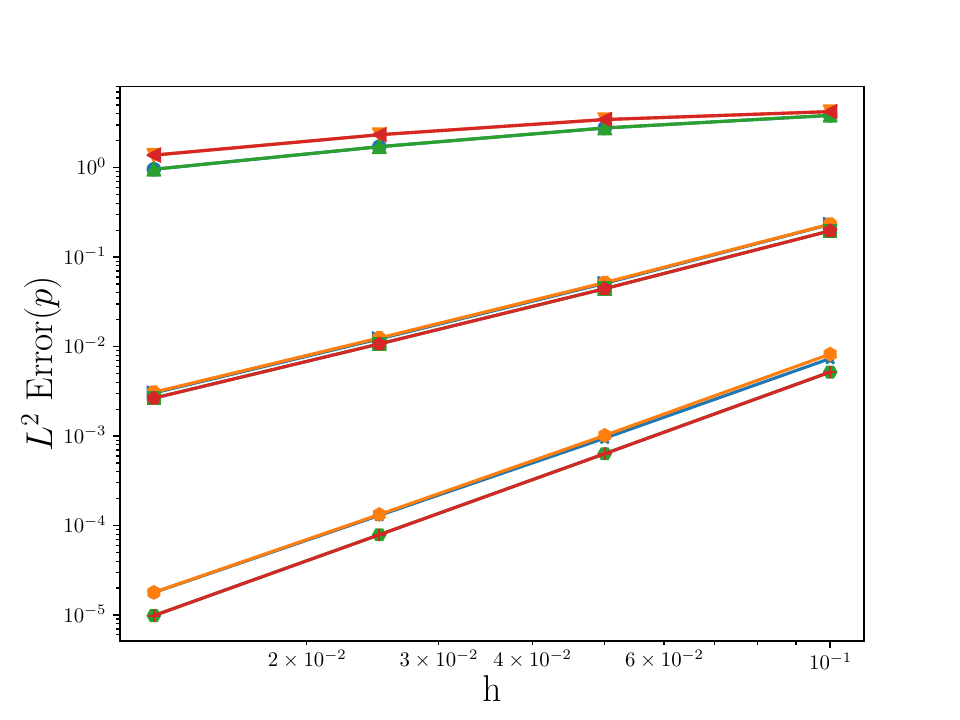} \\
      \includegraphics[width=6cm]{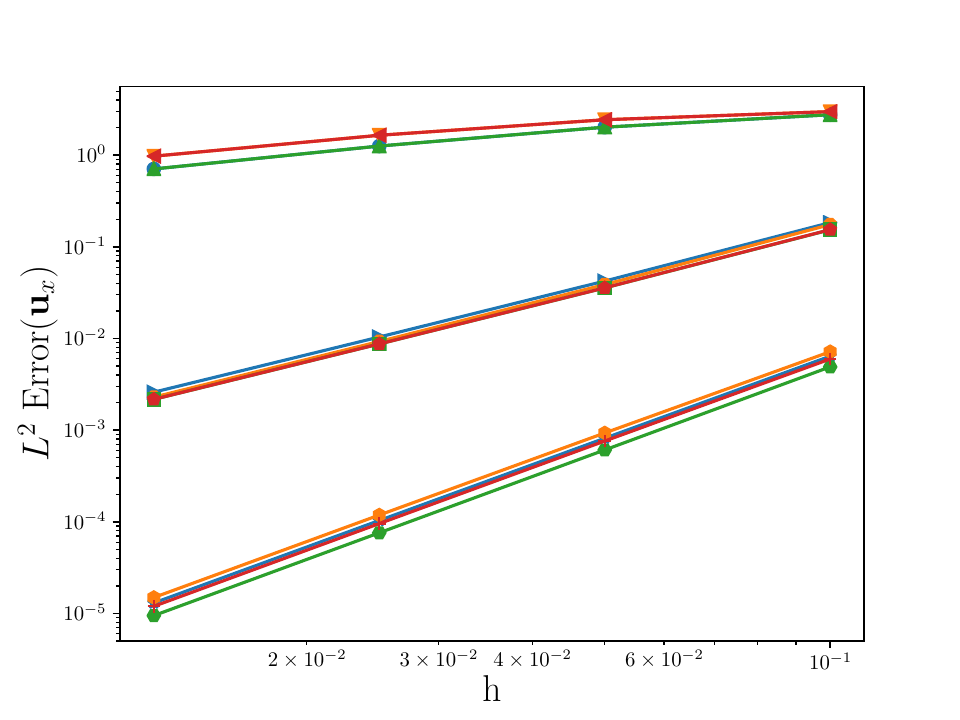}
      &
      \includegraphics[width=6cm]{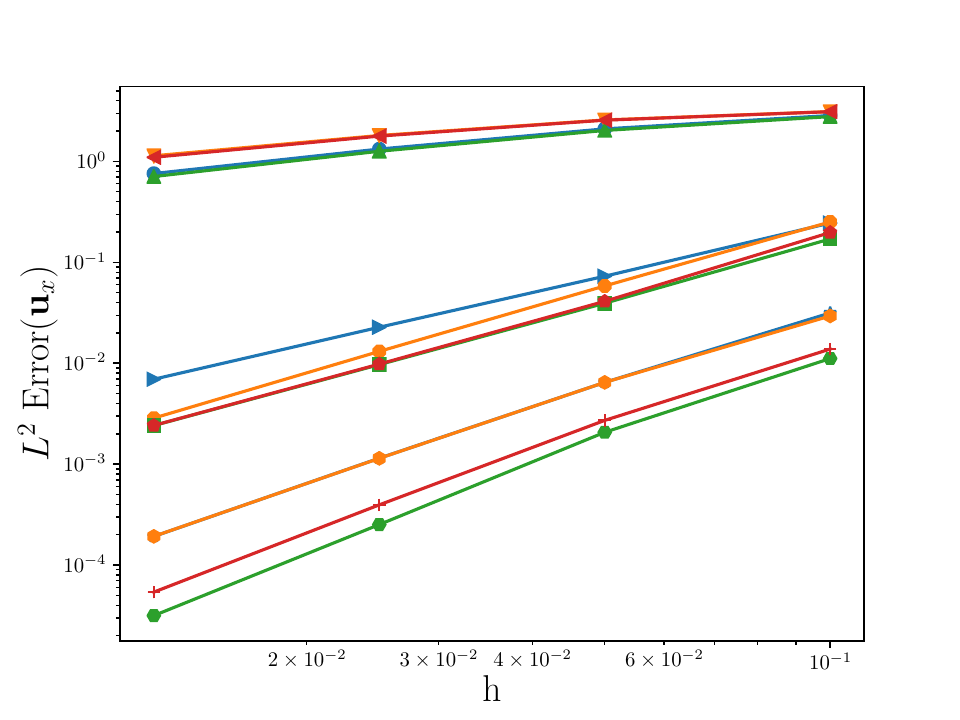} \\
      \includegraphics[width=6cm]{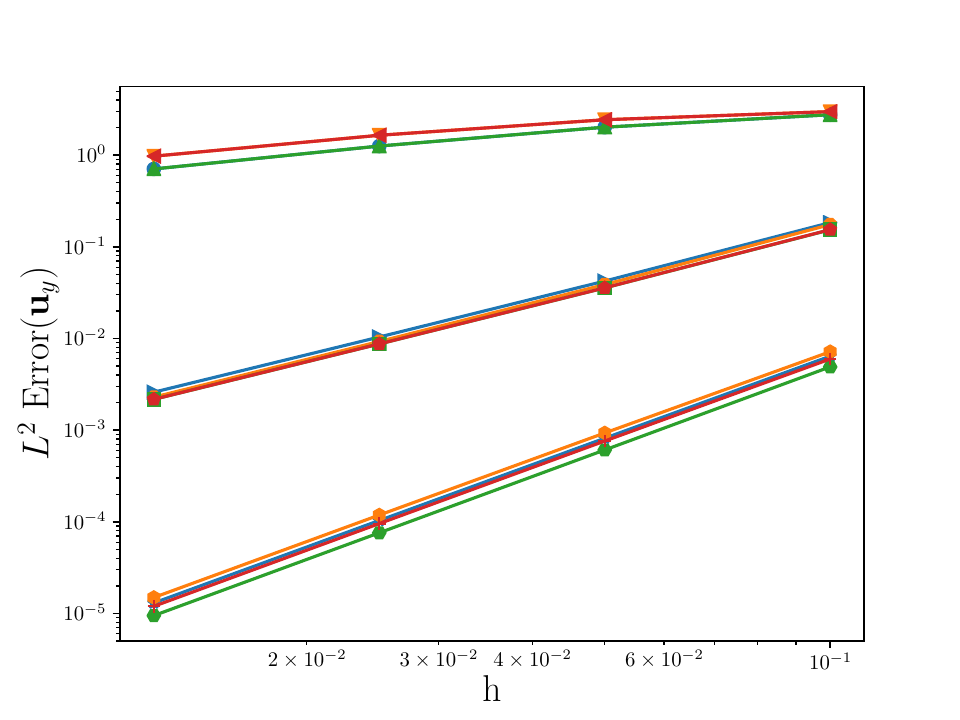}
      &
      \includegraphics[width=6cm]{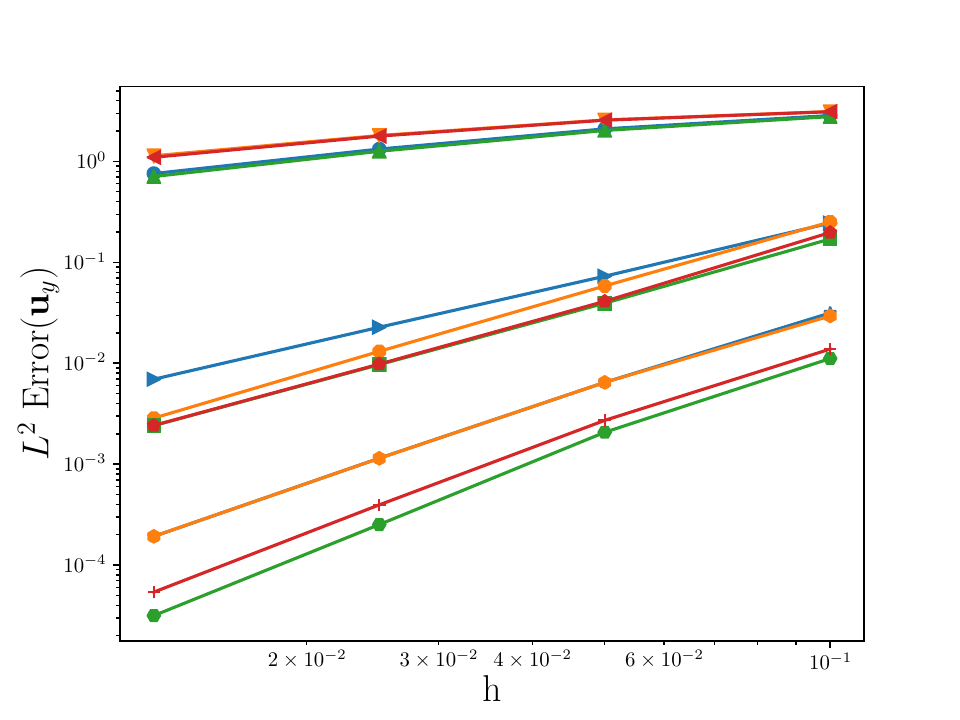} \\
      \multicolumn{2}{c}{
        \includegraphics[width=8cm]{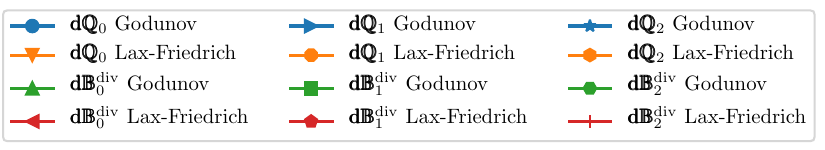}
      }
      \\
    \end{tabular}
  \end{center}
  \caption{\label{fig:CurlConvergenceQuad} Error obtained
    on the test case described in
    \autoref{subsubsec:ConvergenceCurl} 
    with initial condition \eqref{eq:exactCurl} on the left, and
    with the initial condition \eqref{eq:CurlAddVortex} added to
    \eqref{eq:exactCurl} on the right, on a series of
    Cartesian meshes. For each of the test cases,
    the error obtained for the variables $p$ (top row),
    $\bu_x$ (middle row) and $\bu_y$ (bottom row)
    is shown for different
    approximation spaces and for the Lax-Friedrich and Godunov flux. 
  }
\end{figure}

\begin{table}
  \begin{center}
\begin{tabular}{||c||c|c||c|c||c|c||c|c||}\multicolumn{1}{c}{}& \multicolumn{4}{c}{Godunov}& \multicolumn{4}{c}{Lax-Friedrich}\\ \hline\multicolumn{1}{||c||}{}& \multicolumn{2}{|c||}{$\bdQQ_0$}& \multicolumn{2}{|c||}{$\Bdiv_0$}& \multicolumn{2}{|c||}{$\bdQQ_0$}& \multicolumn{2}{|c||}{$\Bdiv_0$}\\ \hline$h$  & Error & rate  & Error & rate  & Error & rate  & Error & rate \\ \hline0.1 & 2.76e+00 &  & 2.76e+00 &  & 2.99e+00 &  & 2.99e+00 &  \\ 0.05 & 2.02e+00  & 0.45 & 2.02e+00  & 0.45 & 2.44e+00  & 0.30 & 2.44e+00  & 0.30\\ 0.025 & 1.26e+00  & 0.68 & 1.26e+00  & 0.68 & 1.65e+00  & 0.56 & 1.65e+00  & 0.56\\ 0.0125 & 7.08e-01  & 0.83 & 7.08e-01  & 0.83 & 9.76e-01  & 0.76 & 9.76e-01  & 0.76\\ \hline \multicolumn{1}{c}{}& \multicolumn{4}{c}{Godunov}& \multicolumn{4}{c}{Lax-Friedrich}\\ \hline\multicolumn{1}{||c||}{}& \multicolumn{2}{|c||}{$\bdQQ_1$}& \multicolumn{2}{|c||}{$\Bdiv_1$}& \multicolumn{2}{|c||}{$\bdQQ_1$}& \multicolumn{2}{|c||}{$\Bdiv_1$}\\ \hline$h$  & Error & rate  & Error & rate  & Error & rate  & Error & rate \\ \hline0.1 & 1.84e-01 &  & 1.54e-01 &  & 1.75e-01 &  & 1.54e-01 &  \\ 0.05 & 4.23e-02  & 2.12 & 3.57e-02  & 2.11 & 3.88e-02  & 2.17 & 3.57e-02  & 2.11\\ 0.025 & 1.04e-02  & 2.03 & 8.72e-03  & 2.03 & 9.28e-03  & 2.06 & 8.72e-03  & 2.03\\ 0.0125 & 2.60e-03  & 2.00 & 2.17e-03  & 2.01 & 2.29e-03  & 2.02 & 2.17e-03  & 2.01\\ \hline \multicolumn{1}{c}{}& \multicolumn{4}{c}{Godunov}& \multicolumn{4}{c}{Lax-Friedrich}\\ \hline\multicolumn{1}{||c||}{}& \multicolumn{2}{|c||}{$\bdQQ_2$}& \multicolumn{2}{|c||}{$\Bdiv_2$}& \multicolumn{2}{|c||}{$\bdQQ_2$}& \multicolumn{2}{|c||}{$\Bdiv_2$}\\ \hline$h$  & Error & rate  & Error & rate  & Error & rate  & Error & rate \\ \hline0.1 & 6.41e-03 &  & 4.90e-03 &  & 7.17e-03 &  & 5.98e-03 &  \\ 0.05 & 8.16e-04  & 2.97 & 6.08e-04  & 3.01 & 9.32e-04  & 2.94 & 7.61e-04  & 2.97\\ 0.025 & 1.03e-04  & 2.98 & 7.59e-05  & 3.00 & 1.19e-04  & 2.97 & 9.57e-05  & 2.99\\ 0.0125 & 1.30e-05  & 2.99 & 9.48e-06  & 3.00 & 1.49e-05  & 2.99 & 1.20e-05  & 3.00\\ \hline \end{tabular}

  \end{center}
  \caption{\label{tab:CurlConvergenceQuadSineCos} Errors and
    convergence rates obtained on the variable $\bu_x$
    with the test case described in \autoref{subsubsec:ConvergenceCurl}
    with initial condition
    \eqref{eq:exactCurl}, on a series of triangular meshes.
    Results show a low benefit in using the Godunov flux, namely in
    exactly preserving the divergence.}
\end{table}

\begin{table}
  \begin{center}
\begin{tabular}{||c||c|c||c|c||c|c||c|c||}\multicolumn{1}{c}{}& \multicolumn{4}{c}{Godunov}& \multicolumn{4}{c}{Lax-Friedrich}\\ \hline\multicolumn{1}{||c||}{}& \multicolumn{2}{|c||}{$\bdQQ_0$}& \multicolumn{2}{|c||}{$\Bdiv_0$}& \multicolumn{2}{|c||}{$\bdQQ_0$}& \multicolumn{2}{|c||}{$\Bdiv_0$}\\ \hline$h$  & Error & rate  & Error & rate  & Error & rate  & Error & rate \\ \hline0.1 & 2.85e+00 &  & 2.78e+00 &  & 3.12e+00 &  & 3.11e+00 &  \\ 0.05 & 2.10e+00  & 0.44 & 2.03e+00  & 0.46 & 2.58e+00  & 0.27 & 2.56e+00  & 0.28\\ 0.025 & 1.33e+00  & 0.66 & 1.26e+00  & 0.69 & 1.81e+00  & 0.51 & 1.78e+00  & 0.52\\ 0.0125 & 7.59e-01  & 0.81 & 7.09e-01  & 0.83 & 1.14e+00  & 0.67 & 1.10e+00  & 0.70\\ \hline \multicolumn{1}{c}{}& \multicolumn{4}{c}{Godunov}& \multicolumn{4}{c}{Lax-Friedrich}\\ \hline\multicolumn{1}{||c||}{}& \multicolumn{2}{|c||}{$\bdQQ_1$}& \multicolumn{2}{|c||}{$\Bdiv_1$}& \multicolumn{2}{|c||}{$\bdQQ_1$}& \multicolumn{2}{|c||}{$\Bdiv_1$}\\ \hline$h$  & Error & rate  & Error & rate  & Error & rate  & Error & rate \\ \hline0.1 & 2.45e-01 &  & 1.71e-01 &  & 2.52e-01 &  & 1.98e-01 &  \\ 0.05 & 7.29e-02  & 1.75 & 3.94e-02  & 2.12 & 5.84e-02  & 2.11 & 4.11e-02  & 2.26\\ 0.025 & 2.28e-02  & 1.68 & 9.70e-03  & 2.02 & 1.31e-02  & 2.15 & 9.83e-03  & 2.07\\ 0.0125 & 6.96e-03  & 1.71 & 2.42e-03  & 2.00 & 2.87e-03  & 2.20 & 2.43e-03  & 2.02\\ \hline \multicolumn{1}{c}{}& \multicolumn{4}{c}{Godunov}& \multicolumn{4}{c}{Lax-Friedrich}\\ \hline\multicolumn{1}{||c||}{}& \multicolumn{2}{|c||}{$\bdQQ_2$}& \multicolumn{2}{|c||}{$\Bdiv_2$}& \multicolumn{2}{|c||}{$\bdQQ_2$}& \multicolumn{2}{|c||}{$\Bdiv_2$}\\ \hline$h$  & Error & rate  & Error & rate  & Error & rate  & Error & rate \\ \hline0.1 & 3.15e-02 &  & 1.12e-02 &  & 2.94e-02 &  & 1.38e-02 &  \\ 0.05 & 6.47e-03  & 2.29 & 2.07e-03  & 2.43 & 6.48e-03  & 2.18 & 2.72e-03  & 2.35\\ 0.025 & 1.15e-03  & 2.50 & 2.53e-04  & 3.04 & 1.14e-03  & 2.50 & 3.96e-04  & 2.78\\ 0.0125 & 1.93e-04  & 2.57 & 3.17e-05  & 3.00 & 1.93e-04  & 2.57 & 5.41e-05  & 2.87\\ \hline \end{tabular}

  \end{center}
  \caption{\label{tab:CurlConvergenceQuadSineCosPlusVortex} Errors and
    convergence rates obtained on the variable $\bu_x$
    with the test case described in \autoref{subsubsec:ConvergenceCurl}
    with initial condition
    \eqref{eq:exactCurl}, on a series of Cartesian meshes.
    Results show a high benefit in using
    the space $\Bcurl_k$ with the Godunov flux, which is the only one
    to always reach the optimal order.}
\end{table}

\begin{figure}
  \begin{center}
    \begin{tabular}{r@{\qquad}l}
      \includegraphics[width=6cm]{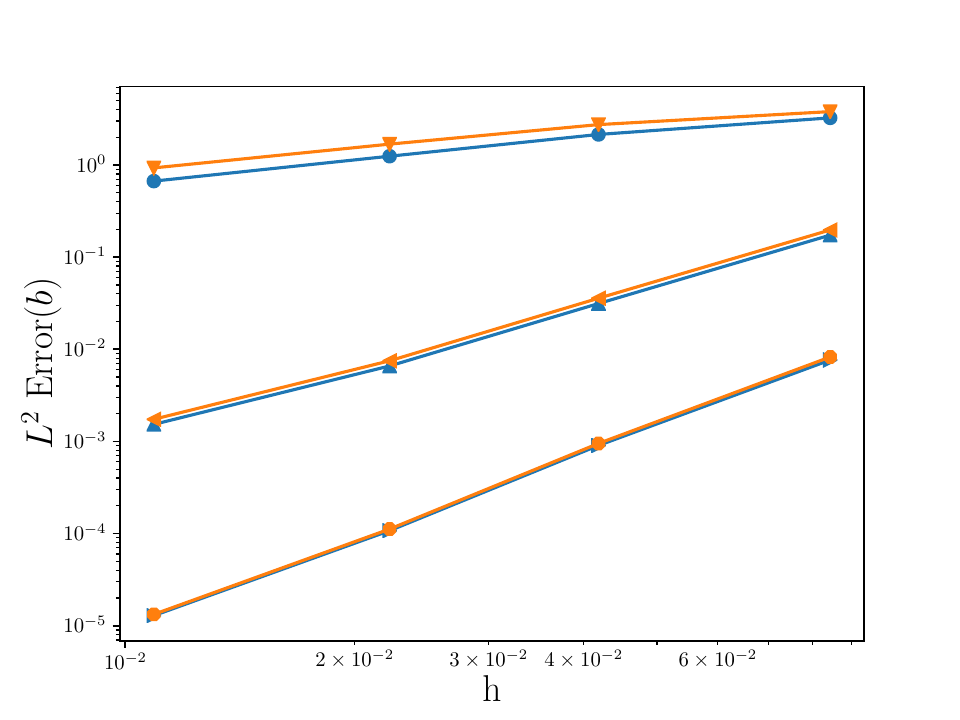}
      &
      \includegraphics[width=6cm]{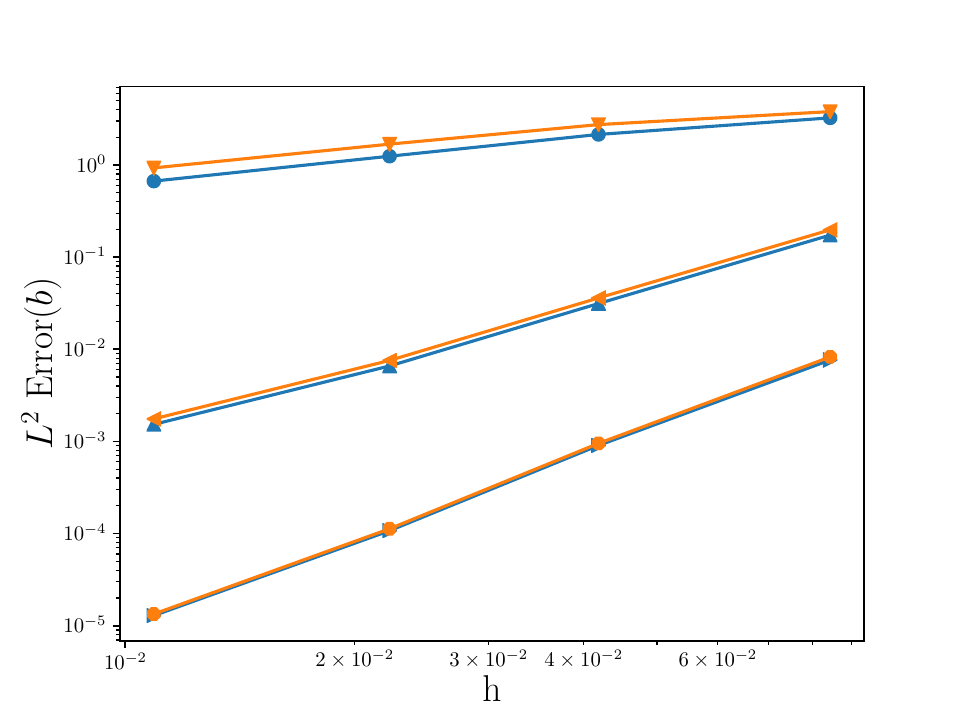} \\
      \includegraphics[width=6cm]{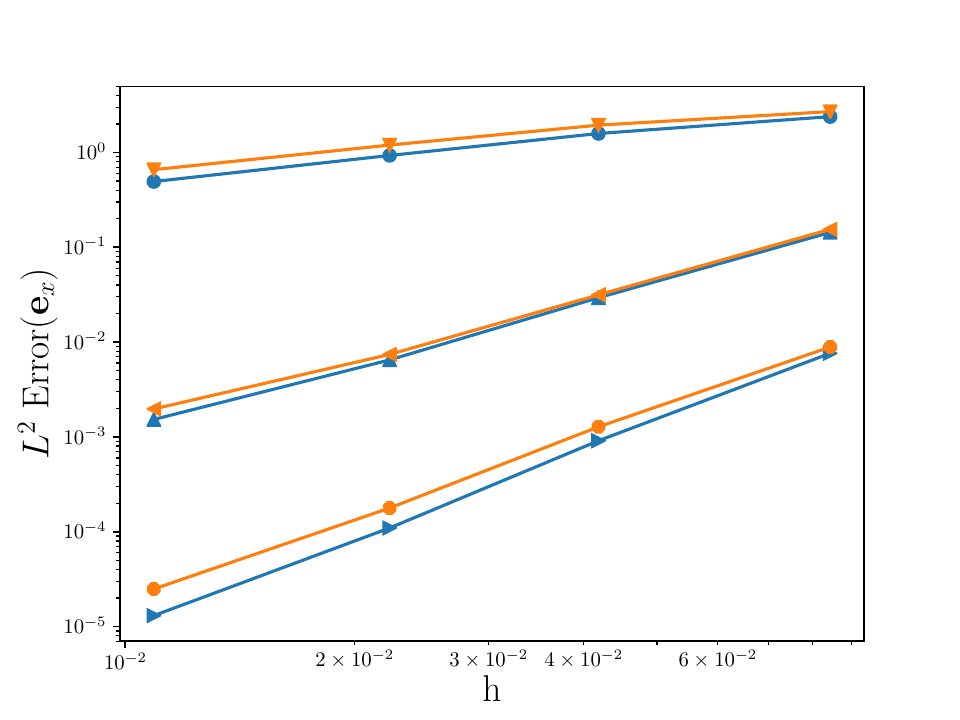}
      &
      \includegraphics[width=6cm]{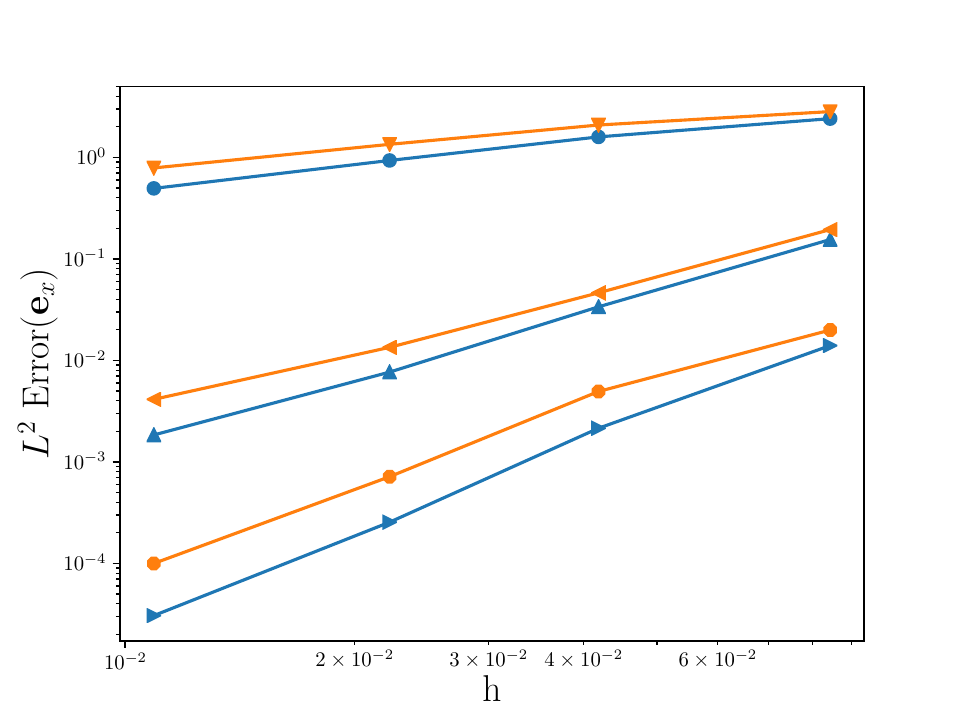} \\
      \includegraphics[width=6cm]{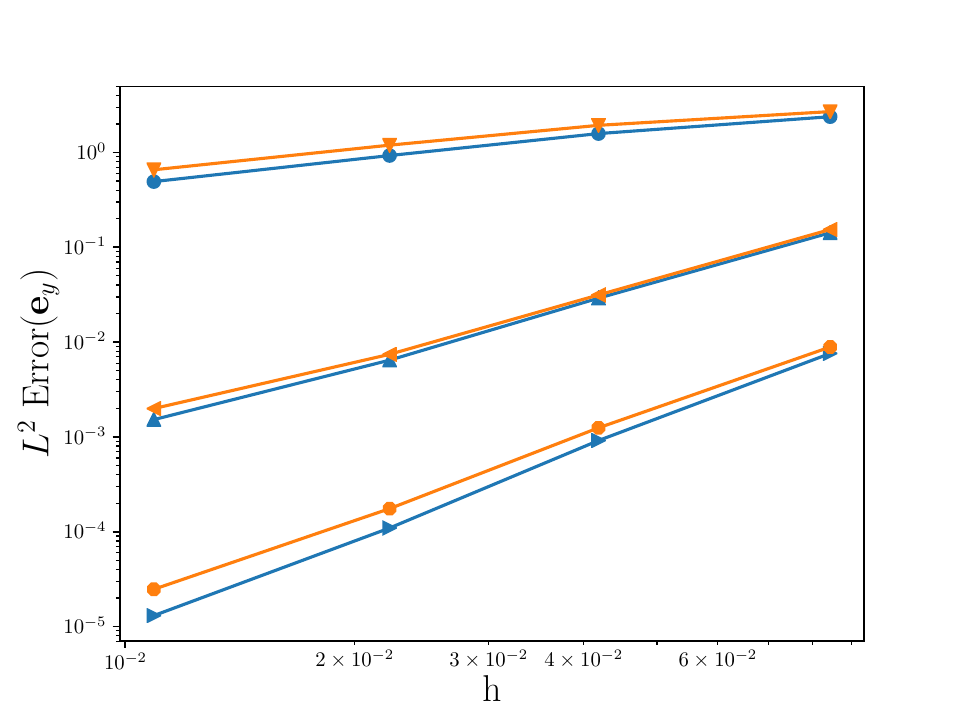}
      &
      \includegraphics[width=6cm]{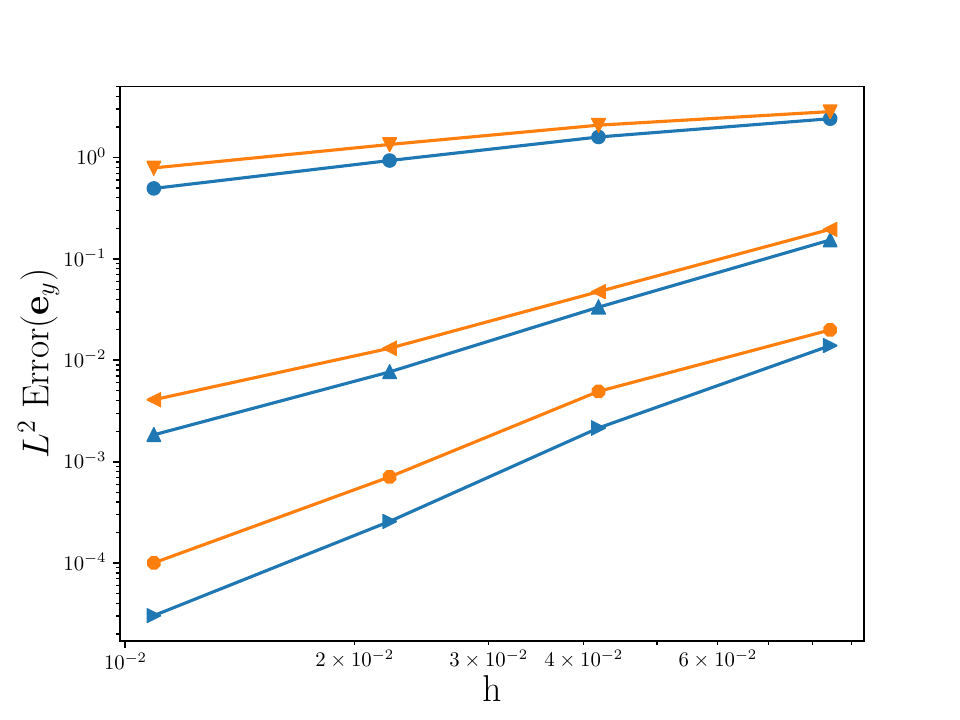} \\
      \multicolumn{2}{c}{
        \includegraphics[width=8cm]{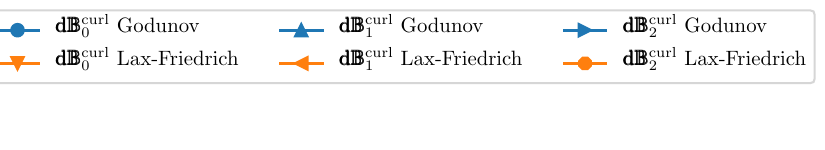}
      }
      \\
    \end{tabular}
  \end{center}
  \caption{\label{fig:CurlConvergenceTriangle} Error obtained
    on the test case described in
    \autoref{subsubsec:ConvergenceCurl} 
    with initial condition \eqref{eq:exactCurl} on the left, and
    with the initial condition \eqref{eq:CurlAddVortex} added to
    \eqref{eq:exactCurl} on the right, on a series of
    unstructured triangular meshes. For each of the test cases,
    the error obtained for the variables $p$ (top row),
    $\bu_x$ (middle row) and $\bu_y$ (bottom row)
    is shown for different degrees, and for
    the Lax-Friedrich and Godunov numerical flux. 
  }
\end{figure}

\begin{table}
  \begin{center}
\begin{tabular}{||c||c|c||c|c||}
\hline 
\multicolumn{1}{||c||}{}
& \multicolumn{2}{|c||}{Godunov}
& \multicolumn{2}{|c||}{Lax-Friedrich}
\\ \hline
$h$ 
 & Error & rate 
 & Error & rate 
\\ \hline
0.08439823 & 2.39e+00 & 
 & 2.70e+00 & 
 \\ 
0.04187874 & 1.58e+00  & 0.59 & 1.94e+00  & 0.47\\ 
0.02226197 & 9.29e-01  & 0.84 & 1.20e+00  & 0.77\\ 
0.01091023 & 4.93e-01  & 0.89 & 6.58e-01  & 0.84\\ \hline 
\multicolumn{1}{||c||}{}
& \multicolumn{2}{|c||}{Godunov}
& \multicolumn{2}{|c||}{Lax-Friedrich}
\\ \hline
$h$ 
 & Error & rate 
 & Error & rate 
\\ \hline
0.08439823 & 1.43e-01 & 
 & 1.55e-01 & 
 \\ 
0.04187874 & 2.94e-02  & 2.26 & 3.15e-02  & 2.27\\ 
0.02226197 & 6.48e-03  & 2.39 & 7.41e-03  & 2.29\\ 
0.01091023 & 1.53e-03  & 2.03 & 1.97e-03  & 1.86\\ \hline 
\multicolumn{1}{||c||}{}
& \multicolumn{2}{|c||}{Godunov}
& \multicolumn{2}{|c||}{Lax-Friedrich}
\\ \hline
$h$ 
 & Error & rate 
 & Error & rate 
\\ \hline
0.08439823 & 7.62e-03 & 
 & 8.87e-03 & 
 \\ 
0.04187874 & 9.12e-04  & 3.03 & 1.28e-03  & 2.76\\ 
0.02226197 & 1.09e-04  & 3.35 & 1.78e-04  & 3.12\\ 
0.01091023 & 1.30e-05  & 2.98 & 2.49e-05  & 2.76\\ \hline 
\end{tabular}

  \end{center}
  \caption{\label{tab:CurlConvergenceTriangleSineCos} Errors and
    convergence rates obtained on the variable $\bu_x$
    with the test case described in \autoref{subsubsec:ConvergenceCurl}
    with initial condition
    \eqref{eq:exactCurl}, on the series of triangular meshes.
    Results show a low benefit in using the Godunov flux, namely in
    exactly preserving the divergence.}
\end{table}

\begin{table}
  \begin{center}
\begin{tabular}{||c||c|c||c|c||}
\hline 
\multicolumn{1}{||c||}{}
& \multicolumn{2}{|c||}{Godunov}
& \multicolumn{2}{|c||}{Lax-Friedrich}
\\ \hline
$h$ 
 & Error & rate 
 & Error & rate 
\\ \hline
0.08439823 & 2.41e+00 & 
 & 2.83e+00 & 
 \\ 
0.04187874 & 1.59e+00  & 0.59 & 2.08e+00  & 0.43\\ 
0.02226197 & 9.33e-01  & 0.85 & 1.34e+00  & 0.69\\ 
0.01091023 & 4.95e-01  & 0.89 & 7.87e-01  & 0.75\\ \hline 
\multicolumn{1}{||c||}{}
& \multicolumn{2}{|c||}{Godunov}
& \multicolumn{2}{|c||}{Lax-Friedrich}
\\ \hline
$h$ 
 & Error & rate 
 & Error & rate 
\\ \hline
0.08439823 & 1.55e-01 & 
 & 1.94e-01 & 
 \\ 
0.04187874 & 3.37e-02  & 2.18 & 4.63e-02  & 2.05\\ 
0.02226197 & 7.66e-03  & 2.34 & 1.34e-02  & 1.96\\ 
0.01091023 & 1.84e-03  & 2.00 & 4.13e-03  & 1.65\\ \hline 
\multicolumn{1}{||c||}{}
& \multicolumn{2}{|c||}{Godunov}
& \multicolumn{2}{|c||}{Lax-Friedrich}
\\ \hline
$h$ 
 & Error & rate 
 & Error & rate 
\\ \hline
0.08439823 & 1.40e-02 & 
 & 1.99e-02 & 
 \\ 
0.04187874 & 2.14e-03  & 2.68 & 4.94e-03  & 1.99\\ 
0.02226197 & 2.54e-04  & 3.37 & 7.14e-04  & 3.06\\ 
0.01091023 & 3.06e-05  & 2.97 & 9.98e-05  & 2.76\\ \hline 
\end{tabular}

  \end{center}
  \caption{\label{tab:CurlConvergenceTriangleSineCosPlusVortex} Errors and
    convergence rates obtained on the variable $\bu_x$
    with the test case described in \autoref{subsubsec:ConvergenceCurl}
    with initial condition
    \eqref{eq:exactCurl}, on the series of triangular meshes.
    Results show a higher benefit than in
    \autoref{tab:CurlConvergenceTriangleSineCos} 
    in using the Godunov flux for preserving exactly the divergence.}
\end{table}

\subsection{Induction equation}

In this section, we are interested in the system \eqref{eq:InductionEquation}.
The test case is taken from \cite[Section 4.3]{veiga2021arbitrary}
\emph{Rotating discontinuous magnetic field loop}, 
still, it was modified in order to ensure that the magnetic loop is regular
because we wish to do a convergence study.
The computational domain is still $[0,1]^2$, and the 
vector field $\bv$ is an orthoradial velocity with respect to the
center of the computational domain $(0.5, 0.5)$:
$\bv = - \be_{\theta}$.

The initial condition is defined by
\begin{equation}
  \label{eq:InductionRegularInit}
  \bu^0 (\bx) = 
  \left\{
  \begin{array}{l}
    \text{if } \rbar < r_0 \quad 
    \left\{
    \begin{array}{r@{\, = \, }l}
      \bu_x &  -2 K_0  \alpha  \ybar  \,
      \dfrac{\ex{ - \alpha/(1-\rbar^2) }}{(1-\rbar)^2} \\
      \bu_y &  \hphantom{-} 2 K_0  \alpha  \xbar \, 
      \dfrac{\ex{ - \alpha/(1-\rbar^2) }}{(1-\rbar)^2} \\
    \end{array}
    \right.
    \\
    \text{if } \rbar \geq r_0 \quad 
    0 ,
  \end{array}
  \right.
\end{equation}
where $r^2 = (x-x_c)^2 + (y-y_c)^2$, $\rbar = r/r_0$,
$\xbar = (x-x_c)/r_0$,
$\ybar = (y-y_c)/r_0$, and $K_0$, $\alpha$, $x_c$, $y_c$ and $r_0$ are numerical
parameters of the test case. The numerical parameters are
$K_0 = 2$, $\alpha = 4$, $x_c = 0.5$, $y_c = 0.75$, and $r_0 = 0.125$.
This solution is such that $\bu^0 = \nablaperp f^0$, where $f^0$
is the regular function
$$f^0(\bx) :=
\left\{
\begin{array}{l@{\qquad }l}
  - K_0 \rbar_0    \ex{ - \alpha/(1-\rbar^2)} & \text{if} \ \rbar < \rbar_0 \\
  0 & \text{otherwise}. \\
\end{array}
\right.
$$
The initial solution is a vortex rotating around the point $(x_c, y_c)$.
The field $\bv$ induces a rotation of the solution around
$(0.5,0.5)$, which means that the exact solution is a vortex rotating
around its center, and the center of the vortex rotates around 
$(0.5,0.5)$ with angular velocity equal to $1$, namely
$$\bu (\bx,t) = R(-t) \bu^0 \left( R(t) \bx \right), $$
where $R(t)$ is the matrix of rotation around $(0.5,0.5)$ of angle
$t$:
$$
R(t) :=
\left(
\begin{array}{cc}
  \cos t & - \sin t \\
  \sin t & \cos t \\
\end{array}
\right).
$$
The computations are led with the approximation space $\Bcurl_k$, and
with the Godunov' flux. 
The initial condition is computed with the method explained in
\autoref{rem:DivergenceInit}.

\subsubsection{Conservation of the divergence free field}

We first check that the divergence of $\bu$ is preserved, equal to $0$.
For this,
the test case is run until time $t=\pi$ on the different meshes
shown in \autoref{fig:Meshes}. The initial
$L^2$ norm of $\nabla^\star \bu^0$ is summarized in \autoref{tab:InitDivergence}.
Results show that the method described in \autoref{rem:DivergenceInit}
ensures the zero divergence for the initial condition. 

\begin{table}
  \begin{center}
    \begin{tabular}{||c||c|c|c||}
      \hline
      degree & 0 & 1 & 2 \\ \hline
      Triangular & $0$ & $2.9602e-14$ & $9.43413e-14$ \\ \hline
      Cartesian  & $0$ & $0$  & $0$ \\ \hline
      Unstructured quad & $0$ & $4.26275e-14$ & $1.69086e-13$ \\ \hline
    \end{tabular}
  \end{center}
  \caption{ \label{tab:InitDivergence} $L^2$ norm of the initial divergence when
    the initial condition is
    computed as in \autoref{rem:DivergenceInit} for the different meshes
  of \autoref{fig:Meshes} and for degree $0$, $1$ and $2$.}
\end{table}
The difference with respect to the initial divergence
is plotted in \autoref{fig:InductionDivergenceConservation}, and show
preservation up to round-off errors of the divergence free field. 

\begin{figure}
  \begin{center}
  \begin{tabular}{c}
    \includegraphics[width=12cm]{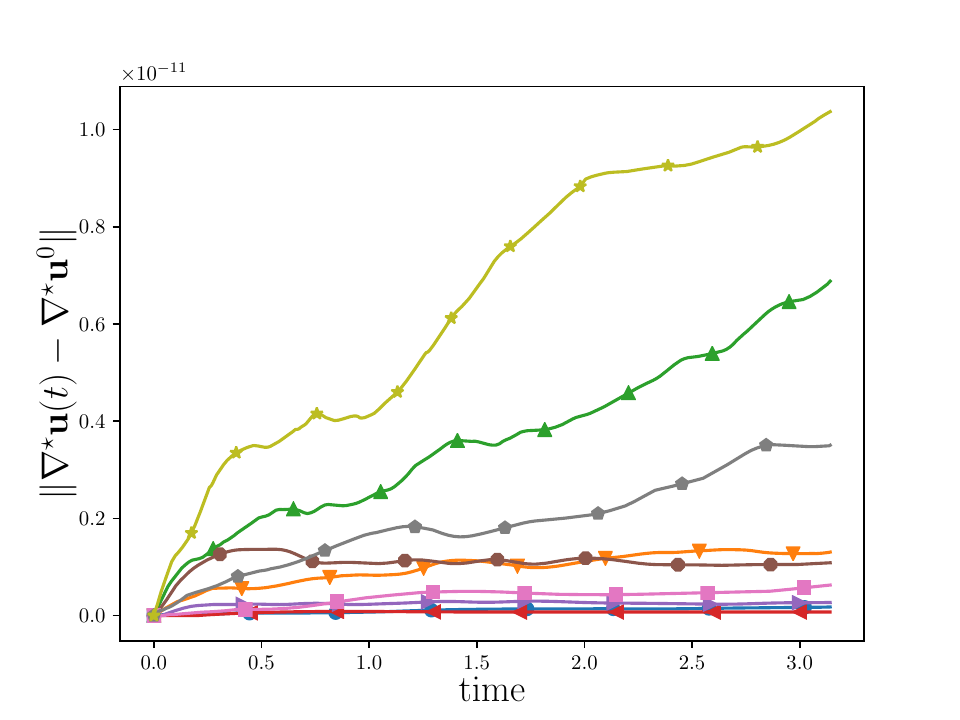} \\
    \includegraphics[width=9cm]{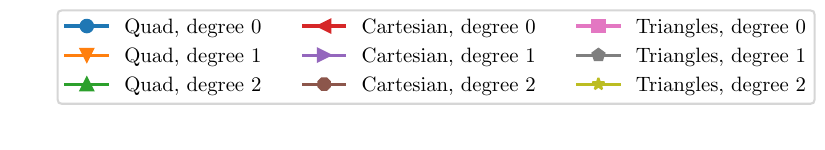} \\
  \end{tabular}
  \end{center}
  \caption{\label{fig:InductionDivergenceConservation} Time evolution of
    $\norme{\nabla^\star \bu - \nabla^\star \bu^0}_2$ with respect to the
    time for a coarse Cartesian, unstructured quadrangular and
    a triangular mesh (all shown in \autoref{fig:Meshes}). Note that the
    scale is $10^{-11}$, so that the figure shows exact (up to round-off errors)
    preservation of the divergence.}
\end{figure}

\subsubsection{Convergence test}
\label{subsubsec:InductionConvergence}

We wish now to perform a convergence test on the series of Cartesian
and unstructured triangular meshes described in the beginning of the section. 
In this test, the computation is led until $t=0.5$.
The convergence curves for the two
variables $\bu_x$ and $\bu_y$ are
plotted in \autoref{fig:ErrorInduction}, whereas the convergence rate
computations are shown in \autoref{tab:ErrorInduction}. Both the table and
the figure represent a convergence close of the optimal order of
convergence. Note however that in agreement with
  \autoref{rem:ConvergenceCartesian}, the order
of accuracy for the variable $\bu_x$, degree 1 on Cartesian mesh is $3/2$.

\begin{figure}
  \begin{center}
    \begin{tabular}{c@{\qquad}c}
      \includegraphics[width=6cm]{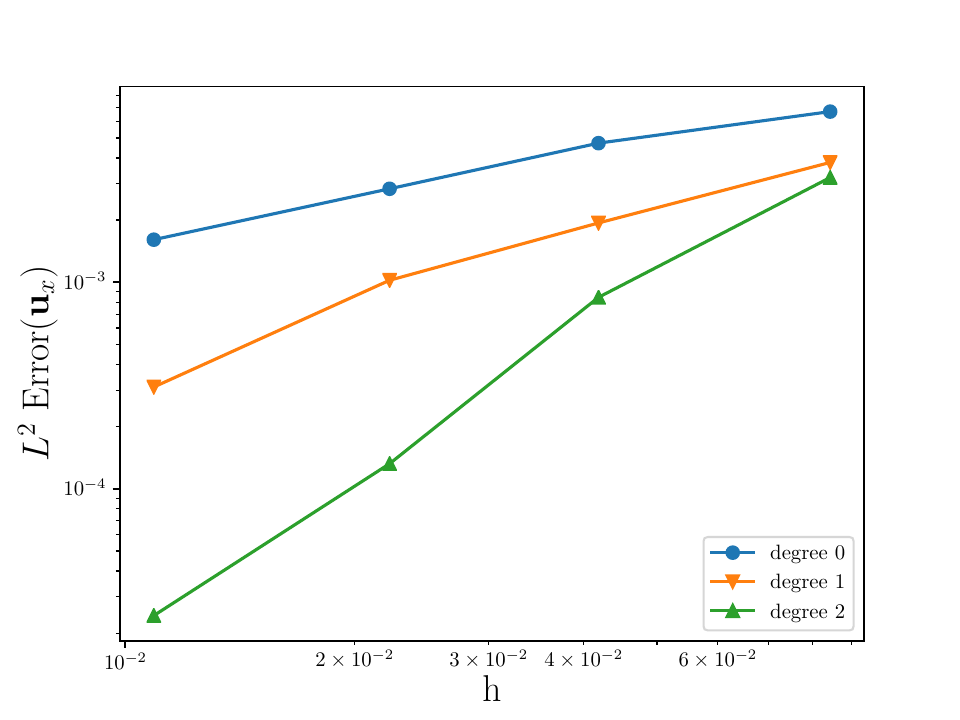}
      &
      \includegraphics[width=6cm]{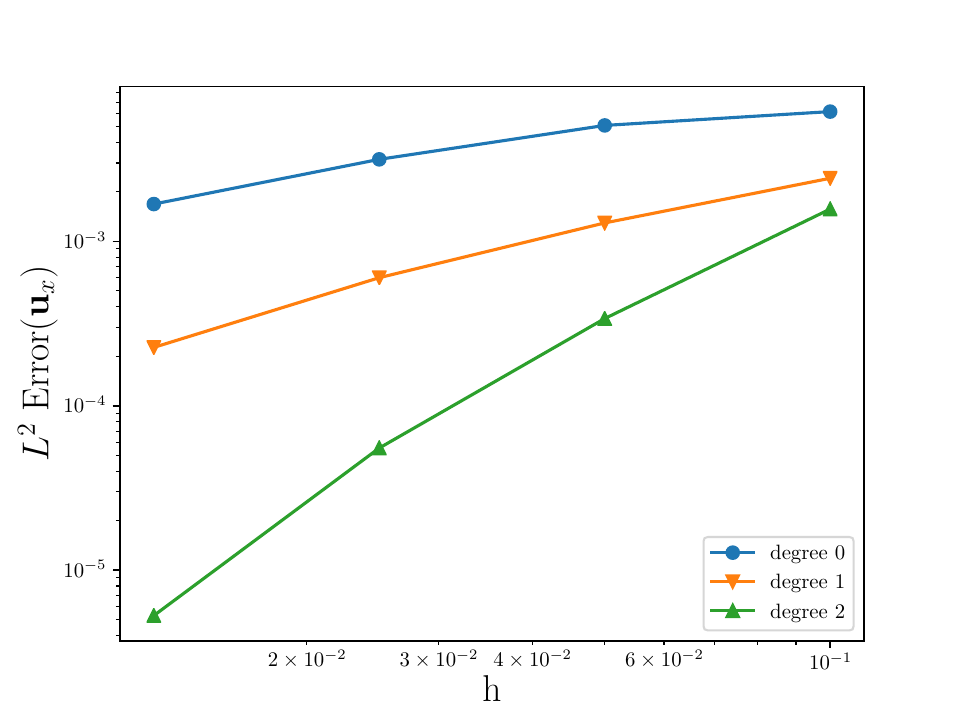} \\
      \includegraphics[width=6cm]{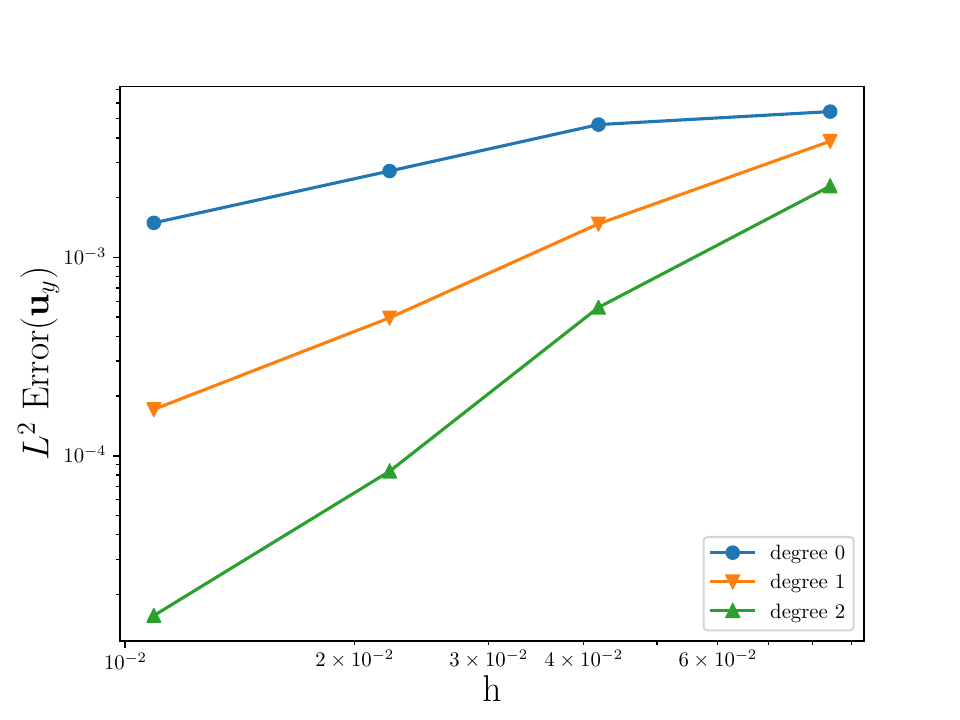}
      &
      \includegraphics[width=6cm]{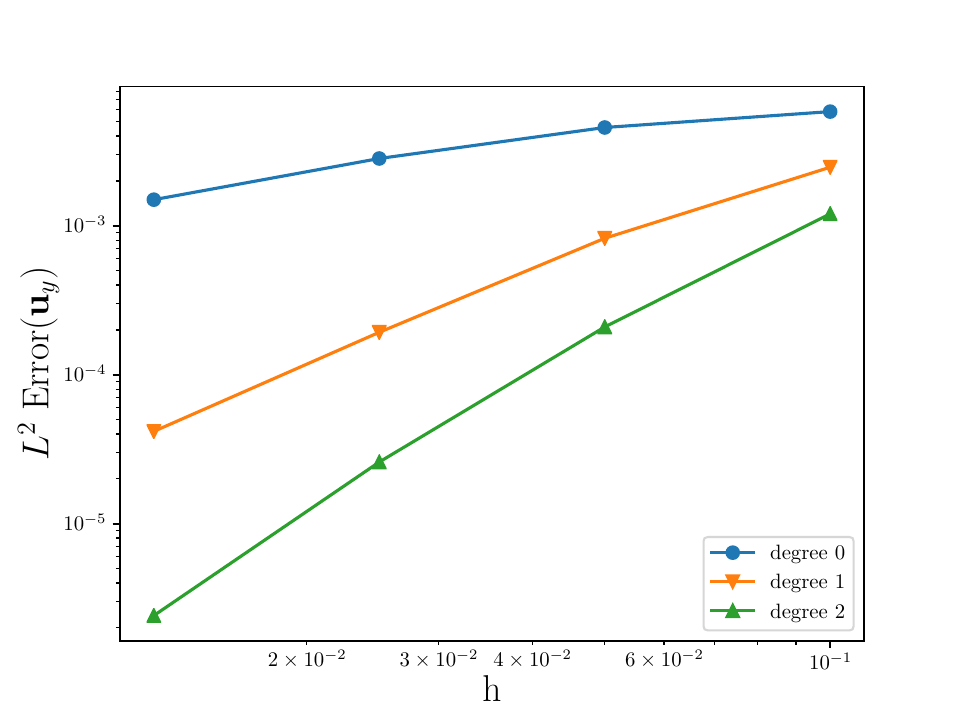} \\
    \end{tabular}
  \end{center}
  \caption{ \label{fig:ErrorInduction}
    Error curves for the test case described in
    \autoref{subsubsec:InductionConvergence}.
    The top figures are the error on $\bu_x$, whereas the bottom
    figures are the error on $\bu_y$. The left column matches with triangular
    meshes whereas the right column is for Cartesian meshes. 
  }
\end{figure}

\begin{table}
\begin{tabular}{c||c||c|c|c|c||c||c|c|c|c||}
\hline 
\multicolumn{1}{c}{degree 0}
& \multicolumn{5}{c}{Cartesian}
& \multicolumn{5}{c}{Triangle}
\\ \hline
\multicolumn{1}{c||}{} 
 & \multicolumn{1}{|c||}{}
& \multicolumn{2}{|c|}{$\bu_x$}
& \multicolumn{2}{|c||}{$\bu_y$}
 & \multicolumn{1}{|c||}{}
& \multicolumn{2}{|c|}{$\bu_x$}
& \multicolumn{2}{|c||}{$\bu_y$}
\\ \cline{2-11}
 & $h$  & Error & rate  & Error & rate  & $h$  & Error & rate  & Error & rate \\ \cline{2-11} 
 & 1.00e-01 & 6.15e-03 &  & 5.84e-03 &  & 8.44e-02 & 6.71e-03 &  & 5.43e-03 &  \\ 
 & 5.00e-02 & 5.07e-03 & 0.28 & 4.57e-03 & 0.35 & 4.19e-02 & 4.72e-03 & 0.50 & 4.67e-03 & 0.22\\ 
 & 2.50e-02 & 3.15e-03 & 0.69 & 2.83e-03 & 0.69 & 2.23e-02 & 2.84e-03 & 0.81 & 2.72e-03 & 0.85\\ 
 & 1.25e-02 & 1.69e-03 & 0.90 & 1.50e-03 & 0.92 & 1.09e-02 & 1.61e-03 & 0.80 & 1.49e-03 & 0.84\\ 
 & 6.25e-03 & 8.72e-04 & 0.95 & 7.79e-04 & 0.94 & 5.60e-03 & 9.44e-04 & 0.80 & 7.50e-04 & 1.03\\ \hline 
\multicolumn{1}{c}{degree 1}
& \multicolumn{5}{c}{Cartesian}
& \multicolumn{5}{c}{Triangle}
\\ \hline
\multicolumn{1}{c||}{} 
 & \multicolumn{1}{|c||}{}
& \multicolumn{2}{|c|}{$\bu_x$}
& \multicolumn{2}{|c||}{$\bu_y$}
 & \multicolumn{1}{|c||}{}
& \multicolumn{2}{|c|}{$\bu_x$}
& \multicolumn{2}{|c||}{$\bu_y$}
\\ \cline{2-11}
 & $h$  & Error & rate  & Error & rate  & $h$  & Error & rate  & Error & rate \\ \cline{2-11} 
 & 1.00e-01 & 2.42e-03 &  & 2.47e-03 &  & 8.44e-02 & 3.80e-03 &  & 3.85e-03 &  \\ 
 & 5.00e-02 & 1.29e-03 & 0.90 & 8.24e-04 & 1.58 & 4.19e-02 & 1.93e-03 & 0.97 & 1.48e-03 & 1.36\\ 
 & 2.50e-02 & 6.01e-04 & 1.11 & 1.93e-04 & 2.10 & 2.23e-02 & 1.02e-03 & 1.01 & 4.95e-04 & 1.73\\ 
 & 1.25e-02 & 2.26e-04 & 1.41 & 4.17e-05 & 2.21 & 1.09e-02 & 3.11e-04 & 1.66 & 1.71e-04 & 1.49\\ 
 & 6.25e-03 & 7.89e-05 & 1.52 & 9.64e-06 & 2.11 & 5.60e-03 & 1.17e-04 & 1.46 & 6.24e-05 & 1.51\\ \hline 
\multicolumn{1}{c}{degree 2}
& \multicolumn{5}{c}{Cartesian}
& \multicolumn{5}{c}{Triangle}
\\ \hline
\multicolumn{1}{c||}{} 
 & \multicolumn{1}{|c||}{}
& \multicolumn{2}{|c|}{$\bu_x$}
& \multicolumn{2}{|c||}{$\bu_y$}
 & \multicolumn{1}{|c||}{}
& \multicolumn{2}{|c|}{$\bu_x$}
& \multicolumn{2}{|c||}{$\bu_y$}
\\ \cline{2-11}
 & $h$  & Error & rate  & Error & rate  & $h$  & Error & rate  & Error & rate \\ \cline{2-11} 
 & 1.00e-01 & 1.57e-03 &  & 1.20e-03 &  & 8.44e-02 & 3.15e-03 &  & 2.28e-03 &  \\ 
 & 5.00e-02 & 3.39e-04 & 2.21 & 2.09e-04 & 2.52 & 4.19e-02 & 8.42e-04 & 1.88 & 5.59e-04 & 2.00\\ 
 & 2.50e-02 & 5.53e-05 & 2.62 & 2.59e-05 & 3.01 & 2.23e-02 & 1.32e-04 & 2.93 & 8.35e-05 & 3.01\\ 
 & 1.25e-02 & 5.29e-06 & 3.39 & 2.41e-06 & 3.42 & 1.09e-02 & 2.42e-05 & 2.38 & 1.55e-05 & 2.36\\ 
 & 6.25e-03 & 4.91e-07 & 3.43 & 2.69e-07 & 3.17 & 5.60e-03 & 3.49e-06 & 2.90 & 2.06e-06 & 3.02\\ \hline 
\end{tabular}

  \caption{\label{tab:ErrorInduction} Errors obtained for the test case of the
    Rotating regular magnetic loop on triangular and quadrangular
    meshes. The error obtained is close of the optimal rate of convergence. }
\end{table}

\subsubsection{Stability test}

In \autoref{prop:EquationOnDivergence}, the divergence free preservation is
obtained for the numerical flux \eqref{eq:LaxFriedrichTangential}, which has 
purely tangential diffusion. Numerical flux \eqref{eq:LaxFriedrichTangential}
is less diffusive than the Lax-Friedrich numerical flux, which has
diffusion in all the directions. The aim of this test is to assess that the
tangential diffusion is sufficient for ensuring $L^2$ stability.

For this, we do the same test as in \cite[Section 4.2]{veiga2021arbitrary},
\emph{Discontinuous magnetic field loop}.
The initial condition is
$$
\bu^0 (\bx) = 
\left\{
\begin{array}{l}
  \text{if } \rbar < r_0 \quad 
  \left\{
  \begin{array}{r@{\, = \, }l}
    \bu_x &  - K_0 \ybar\\
    \bu_y &  \hphantom{-} K_0 \xbar\\
  \end{array}
  \right.
  \\
  \text{if } \rbar \geq r_0 \quad 
  0 ,
\end{array}
\right.
$$
with the same notations as in \eqref{eq:InductionRegularInit}.
The numerical parameters are $x_c = y_c = 0.5$, $K_0 = 0.01$. 
Taking a discontinuous initial condition is well suited with the stability
study, because a discontinuous solution contains a lot of Fourier modes, and
these Fourier modes are decreasing slowly.
Following \cite{veiga2021arbitrary}, the velocity vector field is
$\bv = (1,1)$ for allowing a long time integration until
$t=2$. The boundary conditions are periodic.
The CFL number is equal to $0.5$ for
$k=0$, $0.33$ for $k=1$ and $0.2$ for $k=2$, as for the classical
discontinuous Galerkin method. 
The stability is assessed by computing the
norm of the magnetic energy $\bu \cdot \bu / 2$ at each time step,
and by checking that it is decreasing. This magnetic energy is adimensioned by
its value at time $0$. 
Results obtained on the triangular
mesh of \autoref{fig:Meshes} are shown in \autoref{fig:InductionStability} for
different numerical flux and different initialization:
\begin{itemize}
\item If the divergence free initialization is done as in
  \autoref{rem:DivergenceInit}, and the flux \eqref{eq:LaxFriedrichTangential}
  is used, the scheme
  is stable. This can be explained by the fact that the divergence free
  component of $\bu$ is preserved, while the flux
  \eqref{eq:LaxFriedrichTangential} is sufficient for
  stabilizing the evolution of the curl free component of $\bu$.
  At the fourth revision step of this article, one reviewer claimed
    that the numerical scheme was unstable in long time, and so, on the
    bottom right of this Figure, results for a final time of $20$ are also
  depicted, and show stability too.
\item If the initial condition is computed by projection, without the
  process described in \autoref{rem:DivergenceInit}, and the flux
  \eqref{eq:LaxFriedrichTangential} is
  used, then the scheme is not stable in general, in agreement with
  \autoref{prop:EquationOnDivergence}: both the divergence free and curl free
  component of $\bu$ evolve, and the flux \eqref{eq:LaxFriedrichTangential}
  is not sufficient for
  stabilizing the evolution of the two components.
\item If the divergence free initialization is done as in
  \autoref{rem:DivergenceInit}, and the Lax-Friedrich flux is used, the
  scheme is stable, still, the divergence of $\bu$ is not kept to $0$.
\end{itemize}
\begin{figure}
  \begin{center}
    \begin{tabular}{c@{\quad}c}
      \includegraphics[width=0.45\textwidth]{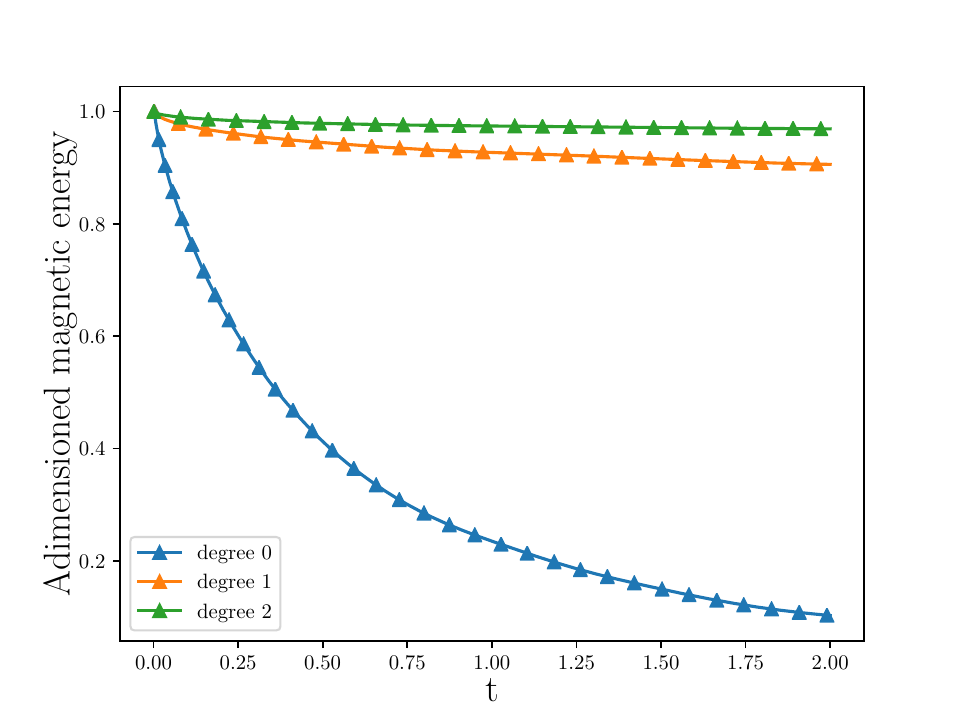}
      &
      \includegraphics[width=0.45\textwidth]{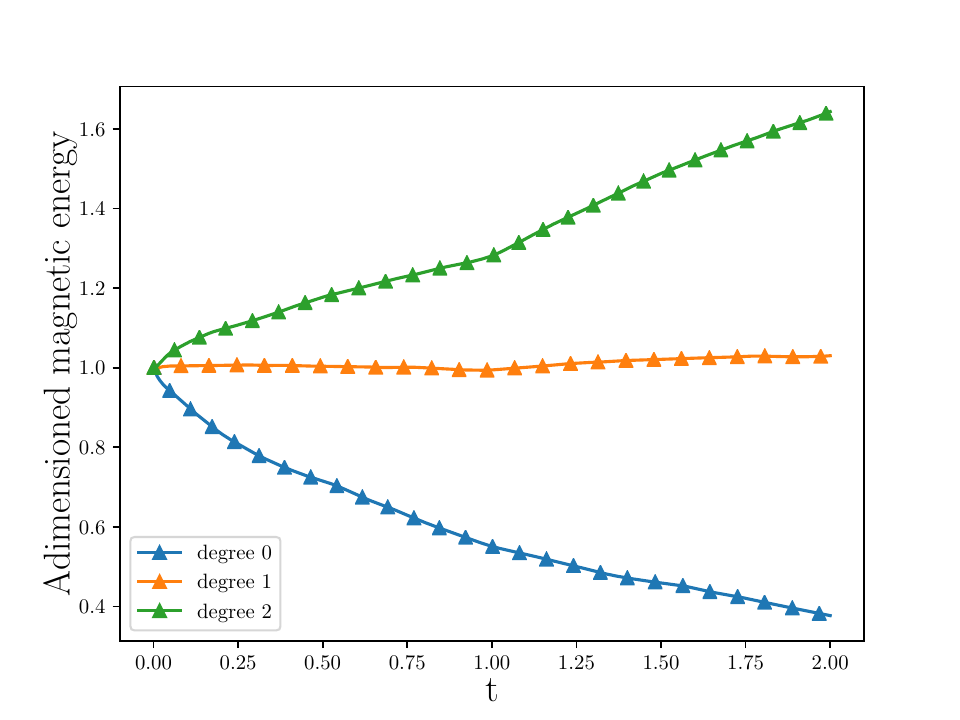} \\
      Divergence free initialization,
      &
      No divergence free initialization, \\
       Lax-Friedrich flux. &  flux \eqref{eq:LaxFriedrichTangential}. \\
       \includegraphics[width=0.45\textwidth]{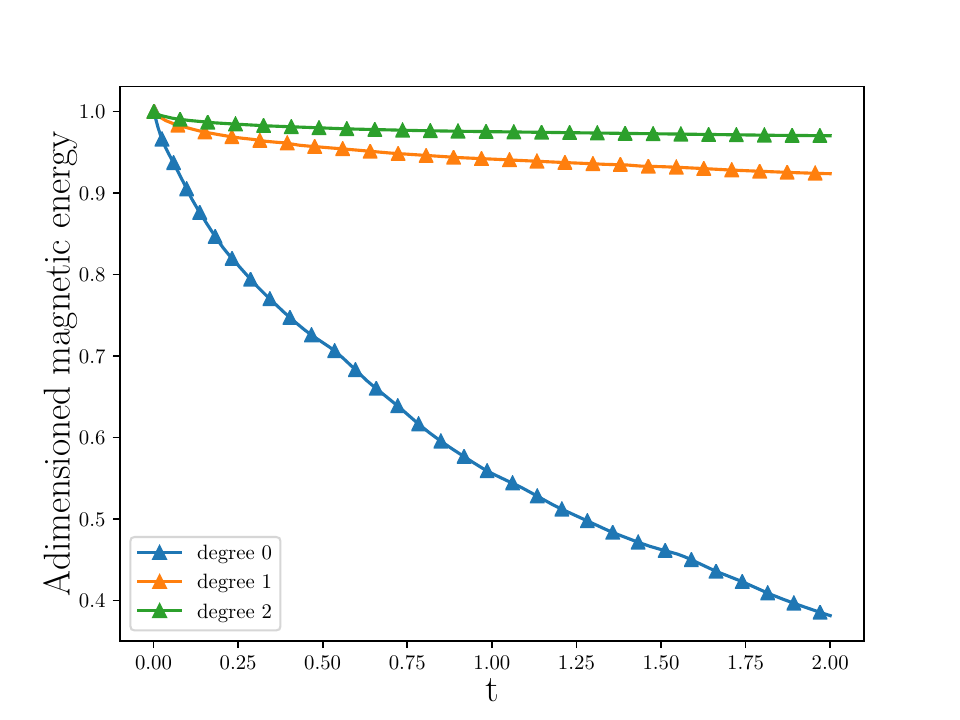}
       &
       \includegraphics[width=0.45\textwidth]{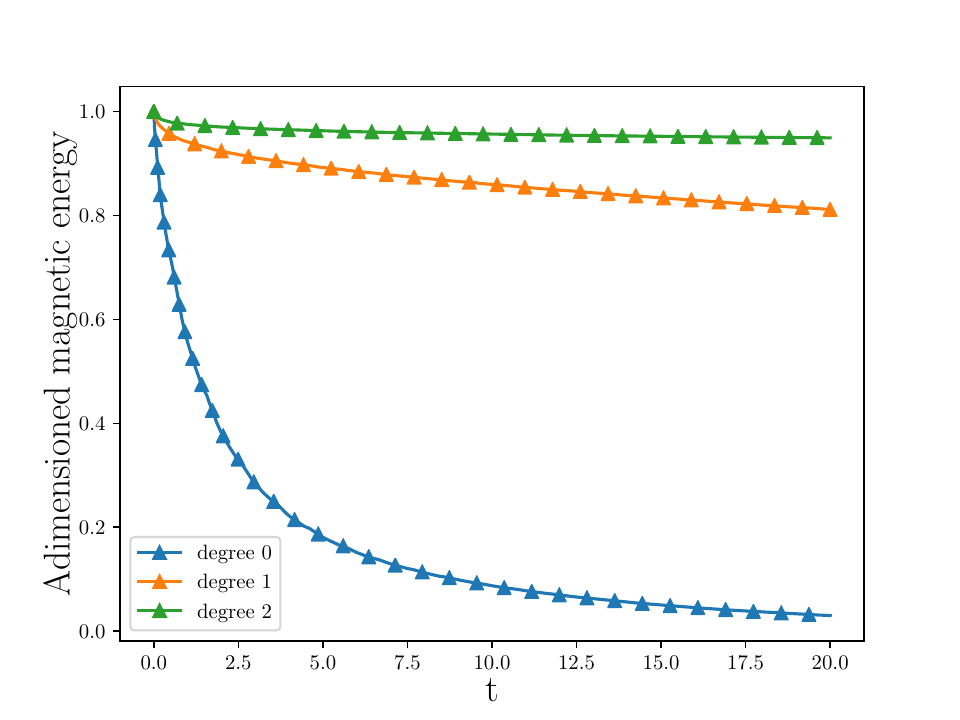}
       \\
      Divergence free initialization,  & Divergence free initialization, \\
      flux \eqref{eq:LaxFriedrichTangential}, $t_{final}=2$.
      &
      flux \eqref{eq:LaxFriedrichTangential}, $t_{final}=20$.
      \\
    \end{tabular}
    \caption{\label{fig:InductionStability}
      Adimensioned magnetic energy
      obtained for the test case \emph{Discontinuous magnetic field loop}.
      On the top left picture, results obtained with a divergence free
      initialization and the Lax-Friedrich flux show a stable behaviour.
      On the top right picture, results
      obtained with a non divergence free initialization and the
      flux \eqref{eq:LaxFriedrichTangential} show an unstable
      behaviour except for the degree 0 scheme. On the bottom picture,
      results obtained with a divergence free initialization and the
      flux \eqref{eq:LaxFriedrichTangential} show a stable behaviour,
      for both final time $2$ on the left and final time
      $20$ on the right.
    }
  \end{center}
\end{figure}
  In all our numerical test, no instability was observed, still, 
  the thorough analysis of the numerical scheme
  remains to be theoretically performed. 
  We foresee also an instability
  that may come from the accumulation of round-off errors: 
  \begin{remark}[Impact of accumulation of round-off errors]
    As seen in \autoref{fig:InductionDivergenceConservation}, the norm of
    the divergence is low, but seems to be increasing.
    We infer that it may come from the numerical scheme
    \eqref{eq:AdvectionCG} ensured by the divergence.
    The numerical scheme
    \eqref{eq:AdvectionCG} preserves an initially zero solution,
    but it does not include any diffusion, and so no mechanism for damping the
    round-off errors around $0$. For very long time integration we foresee that the
    accumulation of round-off errors may lead to a non divergence free
    vector, which may itself lead to instability of the numerical
    scheme, similar to the one observed on top right of the
    \autoref{fig:InductionStability}. This potential instability could be
    cured by performing a divergence cleaning at some time steps,
    similar to redistanciation in level-set methods, or to
    re-orthogonalization in Krylov methods. 
  \end{remark}

\section{Conclusion}
\label{sec:conclusion}

In this article, the discrete preservation of differential constraints
was investigated for the discontinuous Galerkin methods. The method
developed within this article relies on
a previously proposed framework for deriving appropriate approximation
spaces for vectors, which fits in a discrete de-Rham context
\cite{PerrierExact2024}. 
For a continuous equation on a vector that preserves either the
curl or the divergence, we were able to prove that the classical
discontinuous Galerkin method preserves also this constraint at the discrete
level under the following assumptions:
\begin{itemize}
\item Use the correct approximation space for vector unknowns
  (for straight triangular meshes, the approximation space is the
    classical one)
\item Use a numerical flux which diffusion is either parallel (for preserving
  a curl) or orthogonal (for preserving a divergence) to the normal of the
  faces. 
\end{itemize}
Notably, the global curl or the global divergence that are preserved are defined in an adjoint
sense. This use of the adjoint
de-Rham complex is in fact very well suited with the fact that our
numerical method is of Galerkin' type. Indeed, taking a strong
derivative (divergence or curl) of the solution of a numerical scheme
seems to be
a tedious task in general. On the contrary, taking the adjoint of the
exterior derivative is a very simple operation for a Galerkin based
numerical scheme: it just consists in testing the weak formulation with
a gradient or a curl of a function. For the induction equation, it
was even possible to derive a discrete transport equation in the space
$\bbA_{k+1}$ for the adjoint divergence, starting
from the discontinuous Galerkin method in $\Bcurl_k$ for the
unknown vector.
Last, the clear discrete de-Rham framework that was developed in this article
allows to initialize easily a divergence free or a curl free field from the
knowledge of its potential. 

Numerical tests were performed on three linear systems: the two dimensional
Maxwell system for the conservation of the divergence, the two dimensional
wave system for the conservation of the curl, and the two dimensional
induction equation. Numerical
tests confirmed the theoretical results, and show also that the
theoretical results seem to be sharp, in the sense that if one of the 
hypothesis
of \autoref{prop:divConservation} and \autoref{prop:curlConservation}
is not ensured (wrong approximation space or
diffusion not in the right direction), the conservation of the curl or
the divergence is jeopardized.
Last, the implementation was done  with a special family of finite elements
on quads (and also on triangles), but theoretical results hold for
other finite elements proposed in \cite{PerrierExact2024}.

Note that the numerical tests were performed only on linear systems, but
\autoref{prop:curlConservation} and \autoref{prop:divConservation}
hold also for nonlinear problems. Still, these propositions ensure
only geometrical properties on the solution, but do not
say anything on the stability of the schemes. As our scheme is fully
collocated,
we hope to avoid the difficulty of multidimensional Riemann
problems \cite{fu2018globally,chandrashekar2020constraint,balsara2021curl}
raised by the staggering of the magnetic field.
Still, the extension to nonlinear systems  will lead to
the adaptation of the necessary slope
or flux limiting to our framework, which is an opened question for the
moment. 
Concerning divergence free problems, efficiency of the method could
be improved by considering the same numerical method, but with a trimmed
basis consisting of the cell divergence free basis as in
\cite{cockburn2004locally}. Note however that for conforming with the
rest of the method, 
the basis should then be divergence free \emph{in the adjoint sense}.
Another aspect of the nonlinear problems is the bound preserving
property for the MHD system, which was recently resolved
in \cite{wu2023geometric}, and which is strongly linked with the
divergence free condition \cite{wu2018positivity}. It would be interesting
to check whether our method, which is divergence free preserving, but in
the adjoint sense, fits in the theory developed in
\cite{wu2018positivity,wu2023geometric} for the bound preservation.
Concerning curl free methods, applications to the hyperbolic version
of the Hamilton-Jacobi method \cite{lepsky2000analysis}
could be considered in the future, and as this is a problem with a curl
free field, a trimmed adjoint curl free basis function
could be considered as in \cite{li2005reinterpretation} (but in our case,
still in the adjoint sense).

Last, it is also important to see that
\autoref{prop:curlConservation} and \autoref{prop:divConservation} hold also
in dimension 3, provided the right approximation spaces were derived.
We are currently investigating the extension
of the results of this article to dimension 3 and to
nonlinear systems of conservation law.


\begin{thebibliography}{10}

\bibitem{arnold2018finite}
Douglas~Norman Arnold.
\newblock {\em Finite element exterior calculus}.
\newblock SIAM, 2018.

\bibitem{arnold2006finite}
Douglas~Norman Arnold, Richard~Steven Falk, and Ragnar Winther.
\newblock Finite element exterior calculus, homological techniques, and
  applications.
\newblock {\em Acta numerica}, 15:1--155, 2006.

\bibitem{arnold2010finite}
Douglas~Norman Arnold, Richard~Steven Falk, and Ragnar Winther.
\newblock Finite element exterior calculus: from {H}odge theory to numerical
  stability.
\newblock {\em Bulletin of the American mathematical society}, 47(2):281--354,
  2010.

\bibitem{arnold2014periodic}
Douglas~Norman Arnold and Anders~Bernhard Logg.
\newblock Periodic table of the finite elements.
\newblock {\em SIAM News}, 47(9):212, 2014.

\bibitem{balsara2001divergence}
Dinshaw~S. Balsara.
\newblock Divergence-free adaptive mesh refinement for magnetohydrodynamics.
\newblock {\em Journal of Computational Physics}, 174(2):614--648, 2001.

\bibitem{balsara2004second}
Dinshaw~S. Balsara.
\newblock Second-order-accurate schemes for magnetohydrodynamics with
  divergence-free reconstruction.
\newblock {\em The Astrophysical Journal Supplement Series}, 151(1):149, 2004.

\bibitem{balsara2021curl}
Dinshaw~S. Balsara, Roger K{\"a}ppeli, Walter Boscheri, and Michael Dumbser.
\newblock Curl constraint-preserving reconstruction and the guidance it gives
  for mimetic scheme design.
\newblock {\em Communications on Applied Mathematics and Computation}, pages
  1--60, 2021.

\bibitem{balsara1999staggered}
Dinshaw~S. Balsara and Daniel~Shields Spicer.
\newblock A staggered mesh algorithm using high order {G}odunov fluxes to
  ensure solenoidal magnetic fields in magnetohydrodynamic simulations.
\newblock {\em Journal of Computational Physics}, 149(2):270--292, 1999.

\bibitem{barsukow2019stationarity}
Wasilij Barsukow.
\newblock Stationarity preserving schemes for multi-dimensional linear systems.
\newblock {\em Mathematics of Computation}, 88(318):1621--1645, 2019.

\bibitem{barsukow2021truly}
Wasilij Barsukow.
\newblock Truly multi-dimensional all-speed schemes for the {E}uler equations
  on {C}artesian grids.
\newblock {\em Journal of Computational Physics}, 435:110216, 2021.

\bibitem{beirao2013basic}
Lourenco Beir{\~a}o~da Veiga, Franco Brezzi, Andrea Cangiani, Gianmarco
  Manzini, Luisa~Donatella Marini, and Alessandro Russo.
\newblock Basic principles of virtual element methods.
\newblock {\em Mathematical Models and Methods in Applied Sciences},
  23(01):199--214, 2013.

\bibitem{bell1989second}
John~B. Bell, Phillip Colella, and Harland~M. Glaz.
\newblock A second-order projection method for the incompressible
  {N}avier-{S}tokes equations.
\newblock {\em Journal of computational physics}, 85(2):257--283, 1989.

\bibitem{bonelle2014compatible}
J{\'e}r{\^o}me Bonelle.
\newblock {\em Compatible Discrete Operator schemes on polyhedral meshes for
  elliptic and {S}tokes equations}.
\newblock PhD thesis, Universit{\'e} Paris-Est, 2014.

\bibitem{bonelle2015analysis}
J{\'e}r{\^o}me Bonelle and Alexandre Ern.
\newblock Analysis of compatible discrete operator schemes for the {S}tokes
  equations on polyhedral meshes.
\newblock {\em IMA Journal of numerical analysis}, 35(4):1672--1697, 2015.

\bibitem{boscheri2023unconventional}
Walter Boscheri, Rapha{\"e}l Loub{\`e}re, and Pierre-Henri Maire.
\newblock An unconventional divergence preserving finite-volume discretization
  of {L}agrangian ideal {MHD}.
\newblock {\em Communications on Applied Mathematics and Computation}, pages
  1--55, 2023.

\bibitem{bossavit1988whitney}
Alain Bossavit.
\newblock Whitney forms: A class of finite elements for three-dimensional
  computations in electromagnetism.
\newblock {\em IEE Proceedings A (Physical Science, Measurement and
  Instrumentation, Management and Education, Reviews)}, 135(8):493--500, 1988.

\bibitem{bossavit1998computational}
Alain Bossavit.
\newblock {\em Computational electromagnetism: variational formulations,
  complementarity, edge elements}.
\newblock Academic Press, 1998.

\bibitem{bossavit1998geometry}
Alain Bossavit.
\newblock On the geometry of electromagnetism (4): {M}axwell's house.
\newblock {\em AEM Journal of the Japan Society of Applied Electromagnetics and
  Mechanics}, 6(4):318--326, 1998.

\bibitem{brackbill1980effect}
Jeremiah~U. Brackbill and Daniel~C. Barnes.
\newblock The effect of nonzero $\nabla \cdot {B}$ on the numerical solution of
  the magnetohydrodynamic equations.
\newblock {\em Journal of Computational Physics}, 35(3):426--430, 1980.

\bibitem{brooks1982streamline}
Alexander~Nelson Brooks and Thomas Joseph~Robert Hughes.
\newblock Streamline upwind/petrov-galerkin formulations for convection
  dominated flows with particular emphasis on the incompressible navier-stokes
  equations.
\newblock {\em Computer methods in applied mechanics and engineering},
  32(1-3):199--259, 1982.

\bibitem{Cartan1967}
Henri Cartan.
\newblock {\em Differential forms}.
\newblock Hermann, 1967.

\bibitem{chandrashekar2020constraint}
Praveen Chandrashekar and Rakesh Kumar.
\newblock Constraint preserving discontinuous galerkin method for ideal
  compressible mhd on 2-d cartesian grids.
\newblock {\em Journal of Scientific Computing}, 84(2):39, 2020.

\bibitem{cockburn2004locally}
Bernardo Cockburn, Fengyan Li, and Chi-Wang Shu.
\newblock Locally divergence-free discontinuous galerkin methods for the
  maxwell equations.
\newblock {\em Journal of Computational Physics}, 194(2):588--610, 2004.

\bibitem{dedner2002hyperbolic}
Andreas Dedner, Friedemann Kemm, Dietmar Kr{\"o}ner, Claus-Dieter Munz, Thomas
  Schnitzer, and Matthias Wesenberg.
\newblock Hyperbolic divergence cleaning for the {MHD} equations.
\newblock {\em Journal of Computational Physics}, 175(2):645--673, 2002.

\bibitem{Dellacherie2010_cell_geometry}
St{\'e}phane Dellacherie, Pascal Omnes, and Felix Rieper.
\newblock The influence of cell geometry on the \mbox{G}odunov scheme applied
  to the linear wave equation.
\newblock {\em Journal of Computational Physics}, 229(14):5315--5338, 2010.

\bibitem{di2020hybrid}
Daniele~Antonio Di~Pietro and J{\'e}r{\^o}me Droniou.
\newblock {\em The Hybrid High-Order method for polytopal meshes}, volume~19.
\newblock Modeling, Simulation and Application, Springer, 2020.

\bibitem{di2016review}
Daniele~Antonio Di~Pietro, Alexandre Ern, and Simon Lemaire.
\newblock A review of hybrid high-order methods: formulations, computational
  aspects, comparison with other methods.
\newblock {\em Building bridges: connections and challenges in modern
  approaches to numerical partial differential equations}, pages 205--236,
  2016.

\bibitem{di2018introduction}
Daniele~Antonio Di~Pietro and Roberta Tittarelli.
\newblock An introduction to hybrid high-order methods.
\newblock In {\em Numerical Methods for PDEs: State of the Art Techniques},
  pages 75--128. Springer, 2018.

\bibitem{dumbser2020glm}
Michael Dumbser, Francesco Fambri, Elena Gaburro, and Anne Reinarz.
\newblock On {GLM} curl cleaning for a first order reduction of the {CCZ4}
  formulation of the {E}instein field equations.
\newblock {\em Journal of Computational Physics}, 404:109088, 2020.

\bibitem{evans1988simulation}
Charles~R. Evans and John~F. Hawley.
\newblock Simulation of magnetohydrodynamic flows-a constrained transport
  method.
\newblock {\em Astrophysical Journal, Part 1 (ISSN 0004-637X), vol. 332, Sept.
  15, 1988, p. 659-677.}, 332:659--677, 1988.

\bibitem{eymard2010convergence}
Robert Eymard, Thierry Gallou{\"e}t, Raphaele Herbin, and Jean-Claude
  Latch{\'e}.
\newblock Convergence of the {MAC} scheme for the compressible {S}tokes
  equations.
\newblock {\em SIAM Journal on Numerical Analysis}, 48(6):2218--2246, 2010.

\bibitem{eymard2010convergent}
Robert Eymard, Thierry Gallou{\"e}t, Raphaele Herbin, and Jean-Claude
  Latch{\'e}.
\newblock A convergent finite element-finite volume scheme for the compressible
  {S}tokes problem. part {II}: the isentropic case.
\newblock {\em Mathematics of Computation}, 79(270):649--675, 2010.

\bibitem{fu2018globally}
Pei Fu, Fengyan Li, and Yan Xu.
\newblock Globally divergence-free discontinuous galerkin methods for ideal
  magnetohydrodynamic equations.
\newblock {\em Journal of Scientific Computing}, 77:1621--1659, 2018.

\bibitem{fuchs2009stable}
Franz~Georg Fuchs, Kenneth Aksel~Hvistendahl Karlsen, Siddharta Mishra, and
  Nils~Henrik Risebro.
\newblock Stable upwind schemes for the magnetic induction equation.
\newblock {\em ESAIM: Mathematical Modelling and Numerical
  Analysis-Mod{\'e}lisation Math{\'e}matique et Analyse Num{\'e}rique},
  43(5):825--852, 2009.

\bibitem{gallouet2009convergent}
Thierry Gallou{\"e}t, Raphaele Herbin, and Jean-Claude Latch{\'e}.
\newblock A convergent finite element-finite volume scheme for the compressible
  {S}tokes problem. part {I}: The isothermal case.
\newblock {\em Mathematics of Computation}, 78(267):1333--1352, 2009.

\bibitem{godunov1961interesting}
Sergei~Konstantinovich Godunov.
\newblock An interesting class of quasi-linear systems.
\newblock In {\em Doklady Akademii Nauk}, volume 139-3, pages 521--523. Russian
  Academy of Sciences, 1961.

\bibitem{godunov1972symmetric}
Sergue\"i~Konstantinovitch Godunov.
\newblock Symmetric form of the magnetohydrodynamic equation.
\newblock Technical report, Computer Center, Novosibirsk, USSR, 1972.

\bibitem{gottlieb2009high}
Sigal Gottlieb, David~I. Ketcheson, and Chi-Wang Shu.
\newblock High order strong stability preserving time discretizations.
\newblock {\em Journal of Scientific Computing}, 38(3):251--289, 2009.

\bibitem{guillard2009behavior}
Herv{\'e} Guillard.
\newblock On the behavior of upwind schemes in the low {M}ach number limit.
  {IV}: P0 approximation on triangular and tetrahedral cells.
\newblock {\em Computers \& Fluids}, 38(10):1969--1972, 2009.

\bibitem{guillard2017behaviour}
Herv{\'e} Guillard and Boniface Nkonga.
\newblock On the behaviour of upwind schemes in the low {M}ach number limit: A
  review.
\newblock {\em Handbook of Numerical Analysis}, 18:203--231, 2017.

\bibitem{hairer1974butcher}
Ernst Hairer and Gerhard Wanner.
\newblock On the butcher group and general multi-value methods.
\newblock {\em Computing}, 13:1--15, 1974.

\bibitem{allen2001algebraic}
Allen Hatcher.
\newblock {\em Algebraic Topology}.
\newblock Cambridge University Press, 2001.

\bibitem{helzel2011unstaggered}
Christiane Helzel, James~Alexander Rossmanith, and Bertram Taetz.
\newblock An unstaggered constrained transport method for the 3d ideal
  magnetohydrodynamic equations.
\newblock {\em Journal of Computational Physics}, 230(10):3803--3829, 2011.

\bibitem{heumann2010eulerian}
Holger Heumann and Ralf Hiptmair.
\newblock Eulerian and semi-{L}agrangian methods for convection-diffusion for
  differential forms.
\newblock {\em Discrete and Continuous Dynamical Systems}, 29(4):1471--1495,
  2011.

\bibitem{heumann2013stabilized}
Holger Heumann and Ralf Hiptmair.
\newblock Stabilized {G}alerkin methods for magnetic advection.
\newblock {\em ESAIM: Mathematical Modelling and Numerical
  Analysis-Mod{\'e}lisation Math{\'e}matique et Analyse Num{\'e}rique},
  47(6):1713--1732, 2013.

\bibitem{heumann2012fully}
Holger Heumann, Ralf Hiptmair, Kun Li, and Jinchao Xu.
\newblock Fully discrete semi-{L}agrangian methods for advection of
  differential forms.
\newblock {\em BIT Numerical Mathematics}, 52(4):981--1007, 2012.

\bibitem{hiptmair2001discrete}
Ralf Hiptmair.
\newblock Discrete {H}odge operators.
\newblock {\em Numerische Mathematik}, 90:265--289, 2001.

\bibitem{hiptmair2002finite}
Ralf Hiptmair.
\newblock Finite elements in computational electromagnetism.
\newblock {\em Acta Numerica}, 11:237--339, 2002.

\bibitem{hyman1997natural}
James~Mac Hyman and Mikhail Shashkov.
\newblock Natural discretizations for the divergence, gradient, and curl on
  logically rectangular grids.
\newblock {\em Computers \& Mathematics with Applications}, 33(4):81--104,
  1997.

\bibitem{jeltsch2006curl}
Rolf Jeltsch and Manuel Torrilhon.
\newblock On curl-preserving finite volume discretizations for shallow water
  equations.
\newblock {\em BIT Numerical Mathematics}, 46:35--53, 2006.

\bibitem{johnson1986analysis}
Claes Johnson and Juhani Pitk{\"a}ranta.
\newblock An analysis of the discontinuous galerkin method for a scalar
  hyperbolic equation.
\newblock {\em Mathematics of computation}, 46(173):1--26, 1986.

\bibitem{Jung2024}
Jonathan Jung and Vincent Perrier.
\newblock Behavior of the discontinuous {G}alerkin method for compressible
  flows at low {M}ach number on triangles and tetrahedrons.
\newblock {\em SIAM Journal on Scientific Computing}, 46(1):A452--A482, 2024.

\bibitem{Jung2024b}
Jonathan Jung and Vincent Perrier.
\newblock A curl preserving finite volume scheme by space velocity enrichment.
  application to the low {M}ach number accuracy problem.
\newblock {\em Journal of Computational Physics}, 515:113252, 2024.

\bibitem{lebedev1964difference}
Vyacheslav~Ivanovich Lebedev.
\newblock Difference analogues of orthogonal decompositions, basic differential
  operators and some boundary problems of mathematical physics. {I}.
\newblock {\em USSR Computational Mathematics and Mathematical Physics},
  4(3):69--92, 1964.

\bibitem{lepsky2000analysis}
Olga Lepsky, Changqing Hu, and Chi-Wang Shu.
\newblock Analysis of the discontinuous galerkin method for hamilton--jacobi
  equations.
\newblock {\em Applied Numerical Mathematics}, 33(1-4):423--434, 2000.

\bibitem{lesaint1974finite}
Pierre Lesaint and Pierre-Arnaud Raviart.
\newblock On a finite element method for solving the neutron transport
  equation.
\newblock {\em Publications des s{\'e}minaires de math{\'e}matiques et
  informatique de Rennes}, (S4):1--40, 1974.

\bibitem{li2005reinterpretation}
Fengyan Li and Chi-Wang Shu.
\newblock Reinterpretation and simplified implementation of a discontinuous
  galerkin method for hamilton--jacobi equations.
\newblock {\em Applied Mathematics Letters}, 18(11):1204--1209, 2005.

\bibitem{li2012arbitrary}
Fengyan Li and Liwei Xu.
\newblock Arbitrary order exactly divergence-free central discontinuous
  galerkin methods for ideal mhd equations.
\newblock {\em Journal of Computational Physics}, 231(6):2655--2675, 2012.

\bibitem{licht2017complexes}
Martin~Werner Licht.
\newblock Complexes of discrete distributional differential forms and their
  homology theory.
\newblock {\em Foundations of Computational Mathematics}, 17(4):1085--1122,
  2017.

\bibitem{lipnikov2014mimetic}
Konstantin Lipnikov, Gianmarco Manzini, and Mikhail Shashkov.
\newblock Mimetic finite difference method.
\newblock {\em Journal of Computational Physics}, 257:1163--1227, 2014.

\bibitem{mclachlan2016b}
Robert~Iain McLachlan, Klas Modin, Hans Munthe-Kaas, and Olivier Verdier.
\newblock B-series methods are exactly the affine equivariant methods.
\newblock {\em Numerische Mathematik}, 133:599--622, 2016.

\bibitem{mclachlan2024functional}
Robert~Iain McLachlan and Ari Stern.
\newblock Functional equivariance and conservation laws in numerical
  integration.
\newblock {\em Foundations of Computational Mathematics}, 24(1):149--177, 2024.

\bibitem{milani2022artificial}
Riccardo Milani, J{\'e}r{\^o}me Bonelle, and Alexandre Ern.
\newblock Artificial compressibility methods for the incompressible
  {N}avier--{S}tokes equations using lowest-order face-based schemes on
  polytopal meshes.
\newblock {\em Computational Methods in Applied Mathematics}, 22(1):133--154,
  2022.

\bibitem{mishra2010stability}
Siddhartha Mishra and Magnus Sv{\"a}rd.
\newblock On stability of numerical schemes via frozen coefficients and the
  magnetic induction equations.
\newblock {\em BIT Numerical Mathematics}, 50:85--108, 2010.

\bibitem{munz2000divergence}
Claus-Dieter Munz, Pascal Omnes, Rudolf Schneider, \'Eric Sonnendr{\"u}cker,
  and Ursula Voss.
\newblock Divergence correction techniques for {M}axwell solvers based on a
  hyperbolic model.
\newblock {\em Journal of Computational Physics}, 161(2):484--511, 2000.

\bibitem{nedelec1980mixed}
Jean-Claude N{\'e}d{\'e}lec.
\newblock Mixed finite elements in {$\R^3$}.
\newblock {\em Numerische Mathematik}, 35:315--341, 1980.

\bibitem{nicolaides1996analysis}
Roy Nicolaides and X.~Wu.
\newblock Analysis and convergence of the {MAC} scheme. {II}. {N}avier-{S}tokes
  equations.
\newblock {\em Mathematics of Computation}, 65(213):29--44, 1996.

\bibitem{nicolaides1992analysis}
Roy~A. Nicolaides.
\newblock Analysis and convergence of the {MAC} scheme. {I}. the linear
  problem.
\newblock {\em SIAM Journal on Numerical Analysis}, 29(6):1579--1591, 1992.

\bibitem{pagliantini2016computational}
Cecilia Pagliantini.
\newblock {\em Computational magnetohydrodynamics with discrete differential
  forms}.
\newblock PhD thesis, ETH Zurich, 2016.

\bibitem{firstVersion}
Vincent Perrier.
\newblock Development of discontinuous {G}alerkin methods for hyperbolic
  systems that preserve a curl or a divergence constraint.
\newblock available at https://doi.org/10.48550/arXiv.2405.04347.

\bibitem{PerrierExact2024}
Vincent Perrier.
\newblock discrete de-{R}ham complex involving a discontinuous finite element
  space for velocities: the case of periodic straight triangular and
  {C}artesian meshes.
\newblock {\em Annales Henri Lebesgue}, 8:417--452, 2025.

\bibitem{peterson1991note}
Todd~Edmund Peterson.
\newblock A note on the convergence of the discontinuous galerkin method for a
  scalar hyperbolic equation.
\newblock {\em SIAM Journal on Numerical Analysis}, 28(1):133--140, 1991.

\bibitem{powell1994approximate}
Kenneth~Grant Powell.
\newblock An approximate {R}iemann solver for magnetohydrodynamics (that works
  in more than one dimension).
\newblock Technical Report 94-24, ICASE, 1994.

\bibitem{powell1999solution}
Kenneth~Grant Powell, Philip~L. Roe, Timur~J. Linde, Tamas~I. Gombosi, and
  Darren~L. De~Zeeuw.
\newblock A solution-adaptive upwind scheme for ideal magnetohydrodynamics.
\newblock {\em Journal of Computational Physics}, 154(2):284--309, 1999.

\bibitem{raviart1977mixed}
Pierre-Arnaud Raviart and Jean-Marie Thomas.
\newblock A mixed finite element method for 2-nd order elliptic problems.
\newblock In {\em Mathematical aspects of finite element methods}, pages
  292--315. Springer, 1977.

\bibitem{raviart1977primal}
Pierre-Arnaud Raviart and Jean-Marie Thomas.
\newblock Primal hybrid finite element methods for 2nd order elliptic
  equations.
\newblock {\em Mathematics of Computation}, 31(138):391--413, 1977.

\bibitem{tavelli2017pressure}
Maurizio Tavelli and Michael Dumbser.
\newblock A pressure-based semi-implicit space--time discontinuous {G}alerkin
  method on staggered unstructured meshes for the solution of the compressible
  {N}avier--{S}tokes equations at all {M}ach numbers.
\newblock {\em Journal of Computational Physics}, 341:341--376, 2017.

\bibitem{teyssier2019numerical}
Romain Teyssier and Beno{\^\i}t Commer{\c{c}}on.
\newblock Numerical methods for simulating star formation.
\newblock {\em Frontiers in Astronomy and Space Sciences}, 6:51, 2019.

\bibitem{torrilhon2005locally}
Manuel Torrilhon.
\newblock Locally divergence-preserving upwind finite volume schemes for
  magnetohydrodynamic equations.
\newblock {\em SIAM Journal on Scientific Computing}, 26(4):1166--1191, 2005.

\bibitem{torrilhon2004constraint}
Manuel Torrilhon and Michael Fey.
\newblock Constraint-preserving upwind methods for multidimensional advection
  equations.
\newblock {\em SIAM Journal on numerical analysis}, 42(4):1694--1728, 2004.

\bibitem{toth2000b}
G{\'a}bor T{\'o}th.
\newblock The $\nabla \cdot {B} = 0$ constraint in shock-capturing
  magnetohydrodynamics codes.
\newblock {\em Journal of Computational Physics}, 161(2):605--652, 2000.

\bibitem{veiga2021arbitrary}
Maria~Han Veiga, David~Aar\'on Velasco-Romero, Quentin Wenger, and Romain
  Teyssier.
\newblock An arbitrary high-order spectral difference method for the induction
  equation.
\newblock {\em Journal of Computational Physics}, 438:110327, 2021.

\bibitem{Whitney1957}
Hassler Whitney.
\newblock {\em Geometric integration theory}.
\newblock Princeton University Press, Princeton, NJ, 1957.

\bibitem{wu2018positivity}
Kailiang Wu.
\newblock Positivity-preserving analysis of numerical schemes for ideal
  magnetohydrodynamics.
\newblock {\em SIAM Journal on Numerical Analysis}, 56(4):2124--2147, 2018.

\bibitem{wu2023geometric}
Kailiang Wu and Chi-Wang Shu.
\newblock Geometric quasilinearization framework for analysis and design of
  bound-preserving schemes.
\newblock {\em SIAM Review}, 65(4):1031--1073, 2023.

\bibitem{yee1966numerical}
Kane Yee.
\newblock Numerical solution of initial boundary value problems involving
  {M}axwell's equations in isotropic media.
\newblock {\em IEEE Transactions on antennas and propagation}, 14(3):302--307,
  1966.

\end{thebibliography}
\end{document}